\documentclass[10pt,a4paper,final]{article}  
\usepackage{babel}                       
\usepackage{geometry}
\usepackage[utf8]{inputenc}
\usepackage[T1]{fontenc}
\usepackage[reqno]{amsmath} 
\usepackage[reqno]{amsmath}  
\usepackage{amsthm}                    
\usepackage{mathtools}
\usepackage[colorlinks=true, linkcolor=red, urlcolor=red,citecolor=red]{hyperref} 
\usepackage{comment}

\usepackage{tikz}

\tikzstyle{dashed}= [dash pattern= on 2pt off 2pt]

\usetikzlibrary{3d,calc}
\usetikzlibrary{arrows, arrows.meta}
\usetikzlibrary{math}
\usepackage{pgfplots}
\usepackage{esint}

\usepackage{amssymb,latexsym,verbatim} 
\usepackage{amstext}
\usepackage{amsfonts}
\usepackage{amsbsy}
\usepackage{amsopn}
\usepackage{enumerate}
\usepackage{graphics}
\usepackage[figurename=Fig.]{caption}
\usepackage{tcolorbox}
\usepackage{comment}

\usepackage{todonotes}
\usepackage{showlabels}

\graphicspath{{pictures/} }

\def\vint_#1{\mathchoice%
      {\mathop{\kern 0.2em\vrule width 0.6em height 0.69678ex depth -0.58065ex
              \kern -0.8em \intop}\nolimits_{\kern -0.4em#1}}%
      {\mathop{\kern 0.1em\vrule width 0.5em height 0.69678ex depth -0.60387ex
              \kern -0.6em \intop}\nolimits_{#1}}%
      {\mathop{\kern 0.1em\vrule width 0.5em height 0.69678ex depth -0.60387ex
              \kern -0.6em \intop}\nolimits_{#1}}%
      {\mathop{\kern 0.1em\vrule width 0.5em height 0.69678ex depth -0.60387ex
              \kern -0.6em \intop}\nolimits_{#1}}}
\def\vintslides_#1{\mathchoice%
      {\mathop{\kern 0.1em\vrule width 0.5em height 0.697ex depth -0.581ex
              \kern -0.6em \intop}\nolimits_{\kern -0.4em#1}}%
      {\mathop{\kern 0.1em\vrule width 0.3em height 0.697ex depth -0.604ex
              \kern -0.4em \intop}\nolimits_{#1}}%
      {\mathop{\kern 0.1em\vrule width 0.3em height 0.697ex depth -0.604ex
              \kern -0.4em \intop}\nolimits_{#1}}%
      {\mathop{\kern 0.1em\vrule width 0.3em height 0.697ex depth -0.604ex
              \kern -0.4em \intop}\nolimits_{#1}}}

\newcommand{\aveint}[2]{\mathchoice%
      {\mathop{\kern 0.2em\vrule width 0.6em height 0.69678ex depth -0.58065ex
              \kern -0.8em \intop}\nolimits_{\kern -0.45em#1}^{#2}}%
      {\mathop{\kern 0.1em\vrule width 0.5em height 0.69678ex depth -0.60387ex
              \kern -0.6em \intop}\nolimits_{#1}^{#2}}%
      {\mathop{\kern 0.1em\vrule width 0.5em height 0.69678ex depth -0.60387ex
              \kern -0.6em \intop}\nolimits_{#1}^{#2}}%
      {\mathop{\kern 0.1em\vrule width 0.5em height 0.69678ex depth -0.60387ex
              \kern -0.6em \intop}\nolimits_{#1}^{#2}}}

\newcommand{\dist}{\mathrm{dist }}
\newcommand{\tr}{\mathrm{tr}}
\newcommand{\id}{\mathrm{id}}

\newcommand{\diam}{\mathrm{diam}}

\newcommand{\loc}{\mathrm{loc}}

\newcommand{\R}{\mathbb{R}}

\newcommand{\N}{\mathbb{N}}

\newcommand{\pa}{\partial}
\newcommand{\eps}{\varepsilon}
\newcommand{\beq}{\begin{equation}}
\newcommand{\eeq}{\end{equation}}
\newcommand{\T}{\mathrm{T}}

\newcommand{\solT}{$(\Omega_t,u,p)_{t \in [0,T)} \,$}

\renewcommand{\d}{\mathop{}\!\mathrm{d}}
\DeclareMathOperator{\diver}{div}
\renewcommand{\div}{\diver}
\renewcommand{\H}{\mathcal{H}}
\renewcommand{\R}{\mathbb R}

\newcommand{\nablasym}{\nabla_{\mathrm{sym}}}

\newcommand{\curv}{\kappa} 
\newcommand{\n}{\nu} 
\renewcommand{\P}{P}
\newcommand{\inner}[2]{\left\langle #1, #2 \right\rangle}
\newcommand{\freearg}{\, \cdot \,} 
\newcommand{\spt}{\mathrm{spt} \,}

\newcommand{\bP}{\begin{proof}}
\newcommand{\eP}{\end{proof}}

\theoremstyle{definition}

\newtheorem{thm}{Theorem}[section]
\newtheorem{cor}[thm]{Corollary}
\newtheorem{defi}[thm]{Definition}
\newtheorem{lemma}[thm]{Lemma}
\newtheorem{prop}[thm]{Proposition}
\newtheorem{rema}[thm]{Remark}

\newtheorem{exam}[thm]{Example}

\newcommand{\abs}[1]{{\left | #1 \right|}}
\newcommand{\norm}[1]{{\left \| #1 \right\|}}
\numberwithin{equation}{section}

\parskip\medskipamount

\title{Regularity estimates of a fluid-free surface evolution}
\author{Malte Kampschulte\footnote{Department of Mathematical Analysis,
Faculty of Mathematics and Physics,
Charles University,
Sokolovsk\'a 83,
186 75 Praha 8, Czechia; {\em kampschulte@karlin.mff.cuni.cz, niinikoski@karlin.mff.cuni.cz, schwarz@karlin.mff.cuni.cz}}
, Joonas Niinikoski$^*$, Sebastian Schwarzacher$^*$\footnote{Department of Mathematics, Box 480
751 06 Uppsala, Sweden}}

\begin{document}

\allowdisplaybreaks

\maketitle

\begin{abstract}
In this work the evolution of a fluid droplet in vacuum is considered. This means that the surface tension and the fluid forces are in equilibrium at the free boundary. The fluid is governed by the incompressible quasi-steady Stokes equation. 

We present higher order energy estimates for this setting in the planar case. In particular bounds of the curvature and its tangential derivative combined with the second and third spacial derivatives of the fluid velocity as respective dissipation. These estimates are shown to hold until the point of a topological degeneracy. They provide quantitative bounds, that depend on specific properties of the initial geometry only. 

The work contrasts previous approaches, which are based on the use of local coordinates and instead performs all estimates in an Eulerian setting. Indeed, the estimates provided here are geometrically intrinsic and collapse only once these intrinsic qualities break.
\end{abstract}



\section{Introduction}

Free-surface problems form a major branch in the study of partial differential equations and the literature is very rich, containing some of the most celebrated results in PDE theory. A main application is the question how the motion of a liquid drop is evolving under surface tension.

For this setting many results have been shown. Concerning the existence and regularity theory this includes a general existence theory for short times of regular solutions and the global existence theory for states that are close to (smooth) equilibria, see~\cite{solonnikov1977solvability,solonnikov1986unsteady,
solonnikov1991initial,solonnikov2002estimates,padula2010local}. 
Much effort has also been put into the strongly related setting of two fluids, that are separated by surface tension, see~\cite{denisova1994solvability,pruss2011analytic}, the monograph~\cite{pruss2016moving} and the references therein. For this setting it was shown that under suitable but general assumptions solutions become analytic instantaneously and stay so until a blow up is approached~\cite{pruss2011analytic,pruss2016moving}. These results provide a qualitative theory that seems very satisfactory, in particular due to the generality of its range of applicability.
Moreover, it may be assumed that the methods developed there are also applicable for the ``one phase flow'', which we consider, in two or three space dimensions, even so, up to the authors' knowledge, a clear reference seems to be missing.

The previously mentioned results generally rely on the use of (local in time) coordinate systems (e.g.\ height functions or Lagrangian coordinates) that allow to treat the fluid with respect to a fixed domain. The construction is then obtained via linearisation and a contraction principle. Accordingly the related regularity estimates are rather implicit.
What seems generally to be missing in the literature are {\em quantitative regularity estimates} with {\em intrinsic dependence on the geometric evolution.}
 In this work we aim to progress on that question. Higher order energy estimates are provided for what is probably the simplest possible scenario of a free fluid motion under surface tension; that is the planar case of a quasi-steady incompressible fluid. The estimates are produced in a purely {\em  intrinsic Eulerian description} and are derived without the use of a flow map. This implies their validity under very general geometric situations. The estimates can be closed and are valid up to the point of topological degeneracy. Moreover, a respective blow-up can be excluded for a (short) but fixed time interval depending only on the given initial geometric constants of the droplet.

As Example~\ref{example} shows, degenerations in finite time can happen. Hence their general avoidance can not be expected. However a proper quantification of their appearance is still far from being understood. We believe that our quantitative estimates provide a valid first step in that direction.

{\bf The setting:} We introduce the set-up that will be studied in this paper. More details and notation can be found in Section~\ref{sec:prelim}. We define $\nu_t$ as the (time-changing) outer normal of the time-changing domain $\Omega_t$ and $\tau_t$ as the right orientated tangential as well as $\curv_t$ as the curvature. For a vector field $w$ we define $\div_{\nu_t} w:=\langle\partial_{\nu_t} w,\nu_t\rangle$ and $\nablasym w$ as the symmetric gradient.

Our aim is to consider the following system of partial differential equations. We seek a family of domains $\Omega_t\subset \R^2$ with $t\in [0,T]$, evolving with (scalar) normal velocity $v$, as well as for each $t$ a fluid velocity $u:\Omega_t\to \R^2$ and fluid pressure $p:\Omega_t\to \R$, satisfying
\begin{equation} \label{eq:main}
 \left\{
 \begin{aligned}
  \vartheta\Delta u -\nabla p &= 0 &\text{ in }& \Omega_t \\
  \div u &= 0 &\text{ in }& \Omega_t \\
\curv_t \n_t &=[pI-2\vartheta \nablasym u]\n_t &\text{ on }& \pa\Omega_t\\
v&=\langle u, \nu_t\rangle& \text{ on } & \pa\Omega_t
 \end{aligned}
 \right.
\end{equation}

Since we are in two dimensions, the third equation can also be rewritten as a pair of scalar equations
\begin{align}
\label{eq:surfacetension1}
\curv_t &= p - 2\vartheta \div_{\n_t} u (t, \, \cdot \,)  &\text{ on }& \pa\Omega_t \\
\label{eq:surfacetension2}
0 &= \langle\pa_{\n_t} u (t, \, \cdot \,),\tau_t  \rangle +  \langle\pa_{\tau_t} u (t, \, \cdot \,),\n_t \rangle  
&\text{ on } & \pa\Omega_t.
\end{align}

This PDE is coupled to the evolution of $\partial\Omega_t$, that is transported by the normal fluid velocity. The system is closed by providing the initial condition for the domain\footnote{Note that the initial geometry implies the initial velocity by solving the respective Stokes equation (uniquely up to rigid motions), see Subsection~\ref{ssec:stokes} for more details.}
\begin{align}
\label{eq:initial}
\Omega_0=\Omega.
\end{align}
The equation has a gradient-flow structure which provides a weak but natural estimate on the perimeter and first order derivatives of the velocity (see Subsection~\ref{ssec:grad}). For this setting we now aim to show higher order energy estimates.

The first of these, will be an $L^2$-estimate of the curvature, which implies bounds of the Willmore energy of the mean curvature combined with the second order derivatives of the fluid velocity. In this setting the error terms appearing seem to be too strong to allow closing the estimates. Indeed, in order to derive it, we require higher regularity that {\em does not follow} from what the estimate yields back even for short times.

In contrast once the second, $H^1$-estimate of the curvature is considered {\em the cycle can be closed}. This means that strong solutions do indeed satisfy uniform bounds up to geometric degeneration. The bound include the tangential derivative of the mean curvature $\curv$ and third spatial derivatives of the fluid velocity as respective dissipation.

Already in this setting of a two dimensional quasi steady evolution, obtaining higher order energy estimates is a non-trivial task, as they can not simply be read of the equation. It was necessary to rewrite and carefully combine the regularity theory for the (incompressible) mean curvature flow~\cite{JN2023,julin2024flat} with the regularity theory for the steady Stokes operator in a geometrically intrinsic manner. This made it necessary to revisit and specify important analytic tools as {\em Korn} and {\em Poincar\'e{} inequalities} as well as {\em interpolation estimates} with clear control of the geometry (See Subsection~\ref{subsec:domainIndep}).

The $H^1$-control seems not to be strong enough to imply {\em uniform bounds} in the three dimensional setting. It can however be speculated that higher order energy estimates can be closed also for higher spacial dimensions. This relates to previous works in the subject. Indeed, at least since the pioneering works of Solonnikov and his co-authors on the existence of solutions for short times, it was necessary to consider a functional framework beyond {\em the regularity threshold} of bounded curvature.

As mentioned before, our notion of regular solution as well as the main proof itself {\em is independent of the use of any local coordinates}. In particular, no height function, flow map or other type of fixed coordinate system is used. Nevertheless, the potential degeneracy of the geometry can not be excluded for arbitrary times and is actually not expected to hold for a motion of a drop in vacuum.\footnote{ It remains a difficult open question, whether in case of a two phase flow geometric degeneracy can be excluded a-priori; at least in the quasi-steady planar case.} As a measure of degeneracy of the moving surface we take the so-called {\em uniform ball condition (UBC)}. It seems to be a somewhat natural characterisation of non-degeneracy that is well suited for free-boundary problems and mean-curvature flows. In the two dimensional case it essentially means that the curvature $\curv$ is uniformly bounded and that the surface has a positive distance in between intrinsically distant surface points. This size is precisely quantifiable by our estimates which hence allows us to close the circle of estimates, which is the main result of the paper:

\begin{thm}\label{thm:main}
Let \solT be a $C^3$-regular solution to \eqref{eq:main} according to Definition \ref{def:classicsol}.
There exists a time $T^*>0$ depending only on the maximal UBC radius
$r_0$ of the initial domain, the initial perimeter $P_0$, the
viscosity constant $\vartheta$ and the initial $L^2$-energy $\|\partial_{\tau_0} \curv_0\|_{L^2(\pa \Omega_0)}$
of the curvature, such that for all $0\leq t,s < \min\{T,T^*\}$
it holds for the maximal UBC radii $r_t$ and $r_s$ of the domains $\Omega_t$ and $\Omega_s$ respectively that
\begin{align}
|r_t-r_s| \leq \frac{C(P_0,\|\partial_{\tau_0} \curv_0\|_{L^2(\pa \Omega_0)})}{r_0^{8}} \left(\frac{|t-s|}{\vartheta}\right)^\frac13
\end{align}
as well as $r_t,r_s \geq r_0/2$. The dependence of $T^*$ on the initial parameters can be given explictily, for this see Theorem \ref{thm:UBC} and Corollary \ref{cor:time_lowerbound}.
Further we have an explicit time control on $\|\pa_{\tau_t} \curv_t\|_{L^2(\pa \Omega_t)}$ in $[0,T^*)$, see \eqref{growth2} and $\norm{u}_{H^3(\Omega_t)}$ is quantitatively bounded in $L^2(0,T^*)$, see Lemma \ref{lem:normcontrol}. These estimates are based on the estimates \eqref{est:H2control} for $\|u\|_{H^2(\Omega_t)}$ as well as
\eqref{1stvar_higherwillmore:estB} for $\frac{\d}{\d t}\|\partial_{\tau_t} \kappa_t\|^2_{L^2(\partial\Omega_t)}$ and $\|\nabla^3 u\|_{L^2(\Omega_t)}$ respectively.
\end{thm}

We expect that similar estimates can also be obtained for further terms of even higher order. These might even form a step towards a similar estimate in higher dimensions. However the number of geometric correction terms occuring in the calculation seems to increase exponentially with the order, which is why we have restricted ourselves to the case at hand for now.

Due to the high sensitivity of the setting it is natural to consider weaker concepts of solutions which allow for geometric degeneracies.
Most progress in that direction has been achieved for two-fluid flows that are separated by surface tension. Such flows are naturally approximated by so-called smooth interfaces, where a thin layer of mixture is representing the phase transition. The separation is approximated using either Allen-Cahn or Cahn-Hillard equations. Here weak solutions a la Leray are known to exist in three dimensions and strong solutions in two dimensions~\cite{abels2009existence,abels2012thermodynamically}.
By asymptotic expansion they can be related to sharp interface limits. Recently it was shown that assuming a smooth solution for the sharp interface exists the smooth interface converges to this solution~\cite{abels2018sharp,abels2022non,hensel2023sharp,abels2023approximation}. Moreover, in some cases passing to the limit with the smooth interface allows to construct so-called varifold solutions of the sharp interface~\cite{abels2014sharp}. These are known to coincide with the unique classical solutions, if the latter exists~\cite{fischer2020weak}.  Finally it is very classical to consider the Euler flow with free surface. This is in some sense the opposite point of view to the present article. Here we refer to the recent works~\cite{julin2024priori,baldi2024liquid}, which are about respective energy estimates and regularity questions.

The paper is structured as follows. A preliminary section where the notation and some analytic results are introduced follows a section where the concept of solution is introduced as well as an instructive example. Section 4 and 5 are dedicated for the derivation of the $L^2$ and $H^1$ curvature estimates, respectively. In Section 6 these estimates are used to conclude the proof of Theorem~\ref{thm:main}.

\subsection*{Acknowledgment}
M.~K., J.~N.\ and S.~S.~are supported by the ERC-CZ Grant CONTACT LL2105 funded
by the Ministry of Education, Youth and Sport of the Czech Republic and the University Centre
UNCE/SCI/023 of Charles University. M.~K.\ is additionally funded by the Czech Science Foundation (GA\v{C}R) under grant No.\ 23-04766S. S.~S.~also acknowledges the support of the VR Grant
2022-03862 of the Swedish Research Council and the supported by Charles University Research Centre program No.\ UNCE/24/SCI/005. S.~S.~is a member of the Necas center for mathematical modeling. The authors wish to thank Vesa Julin for the many valuable discussions on the subject.


\section{Preliminaries}
\label{sec:prelim}


\subsection{Notation}
We denote a ball (or a disk) of radius $r \in \R_+$ centered at $x \in \R^2$ by $B_r(x)$. If $x=0$, then we simply write $B_r$. Again $\langle \freearg,\freearg \rangle$ is reserved for the scalar product.  
The notation $\spt f$ denotes the support of given a function $f$.
For a scalar function $\phi$ we use interchangeably  the notation $\nabla \phi$
for the differential and gradient. The symmetric gradient operator is given by $\nablasym = (\nabla +\nabla^\T)/2$. The notation $\otimes$ stands for the generic outer product between given $\omega \in V$, $V$ a finite dimensional real vector space , and $x \in \R^2$, that is, $\omega \otimes x$ defines 
a linear map $\R^2 \rightarrow V$ given by the rule $[\omega \otimes x]y= \langle x,y \rangle \omega$.
Using this convention we have for every vector field $X \in C^k(U; \R^2)$ and every orthonormal frame $\{v,w\}$ of $\R^2$
\[
\nabla^k X = \pa_v \nabla^{k-1} X  \otimes v + \pa_w \nabla^{k-1} X  \otimes w.
\]
If $M$ is sufficiently regular $2 \times 2$ matrix field, we denote $\div M = [\pa_v M] v + [\pa_w M] w$.

For given subsets $A,A' \subset \R^2$ we denote the usual Hausdorff distance between them by $d_\H(A,A')$.
For any measurable subset $A \subset \R^2$ we denote its area (volume) by $|A|$.
If $|A|<\infty$ and $T$ is an  integrable tensor field in $A$, then
$T_A$ denotes the integral average of $T$ over $A$. Unless otherwise stated, $\Omega$ stands for a non-empty \textbf{bounded} domain in $\R^2$, that is, $\Omega \neq \varnothing$ is a bounded, open and connected subset of $\R^2$. If $\Omega$ has a finite
perimeter we denote it by $P(\Omega)$. For domains of the regularity we consider we have $P(\Omega)=\H^1(\pa \Omega)$.

Finally, $c$ and $C$ are reserved for generic positive constants for lower and upper estimates respectively.
These constants may change their values even in the same estimate chain.
If there is a dependency on a domain $\Omega$ or a relevant quantity $q$, we shall denote
$C=C(\Omega)$, $C=C(q)$, $C=C(\Omega,q)$ and so forth. Otherwise, we say that $c$ or $C$ is independent which can be read as ``dimensional'' albeit we are strictly restricted to $\R^2$.


\subsection{Regular domains and tangential differentiation}

We say that a bounded domain $\Omega \subset \R^2$ is $C^k$-regular for $k \in \N$ 
provided that we can write it locally as a subgraph of a $C^k$-regular function by possibly rotating the coordinates.
In particular $\pa \Omega$ is an embedded $C^k$-regular (not necessarily connected) manifold. As all $1$-dimensional manifolds are curves, we can always find a local parametrization around $x_0 \in \pa \Omega$ in form of a map $\gamma \in C^k (I ;\pa \Omega)$ from an open interval $I\subset \R$ such that $\gamma' \neq 0$, $\gamma(I)$ is open in $\pa \Omega$ and $\gamma : I \rightarrow \gamma(I)$ is a homeomorphism. Other classes such as $C^{k,\alpha}$-regular domains are defined in similar fashion.

We denote the outer unit normal field of $\pa \Omega$ by $\n:\partial \Omega \to \R^2$ and the unit tangent vector $\tau : \pa \Omega \to \R^2$ such that $\{\n,\tau\}$ forms a right-handed coordinate frame, i.e.\ $\tau = (-\n_2,\n_1)$
in the standard coordinates. Thus $\pa \Omega$ is oriented counter-clockwise and $\n, \tau \in C^{k-1}(\pa \Omega;\R^2)$ form an orthonormal basis of $\R^2$ at every point $x \in \pa \Omega$.

We define the tangential derivative for $f \in C^{l}(\pa \Omega; \R^m)$, $1 \leq l \leq k$, 
by setting $\pa_\tau f = [\nabla \tilde f ] \tau$ where $\tilde f$ is any (local) $ C^{l}$-extension of $f$. 
Then $\pa_{\tau}$ defines a linear map $C^{l}(\pa \Omega; \R^m) \rightarrow C^{l-1}(\pa \Omega; \R^m)$.
Again, we define recursively $\pa_\tau^l = \pa_\tau \pa_\tau^{l-1}$  for every $l=1,\ldots,k$. The tangential divergence of $X \in C^1(\pa \Omega;\R^2)$ is given by  $\div_\tau X := \langle \pa_\tau X, \tau \rangle$.
It should be noted while tangential derivative is tied to the choice of orientation, the tangential divergence is an orientation free expression.
We also set $\pa_\n f = [\nabla f] \nu$ for every $f \in C^1(\overline \Omega;\R^m)$. Although this expression
heavily depends on how $f$ is defined in the vicinity of $\pa \Omega$ it is useful once we extend $\nu$ and $\tau$
into a tubular neighborhood in a consistent manner provided that $\Omega$ has sufficient regularity.

If $k \geq 2$, the (scalar) curvature $\curv$ of $\pa \Omega$ can be defined as
\[
\curv = \div_\tau \n.
\]
With the chosen orientation
$\curv$ is always non-negative when $\Omega$ is convex. We recall the fundamental Weingarten identities
\begin{equation}
\label{weingarten}
\pa_\tau \n = \curv \, \tau \qquad \text{and} \qquad \pa_\tau \tau = - \curv \, \n.
\end{equation}

Combining Stokes' theorem with the second indentity results in the classical surface divergence theorem
which states
\begin{equation}
\label{eq:surfacedivthm} 
\int_{\pa \Omega} \div_\tau  X \d \H^1 = \int_{\pa \Omega} \curv \langle X , \n \rangle \d \H^1 
\end{equation}
for every $X \in C^1(\pa \Omega;\R^2)$. This identity can be seen as a distributional definition of curvature.

We will also make use of the two-dimensional Reilly's formula \cite{Reilly}  which in 
this context reduces into
\begin{align}
\label{vReilly2D}
\begin{split}
&\int_\Omega |\nabla^2 X|^2 - |\Delta X|^2 \d x \\
&=- \int_{\pa \Omega}  2\langle \pa_\n X, \pa_\tau^2 X\rangle + \curv |\nabla X|^2 \d \H^1.
\end{split}
\end{align}
for every $X \in C^2(\overline \Omega;\R^2)$. For the reader's convenience we derive this identity in the Appendix,
see \eqref{ReillyComp}.

In the case of time-dependent domain $\Omega_t$ we use notations $\n_t$, $\tau_t$, $\curv_t$ and so forth
just to avoid confusion.


\subsection{Uniform ball condition, signed distance and tubular neighborhoods} \label{sec:UBC}

\begin{defi}[Uniform ball condition]
\label{def:UBC}
We say that a domain $\Omega \subset \R^2$ satisfies the uniform ball condition (UBC) with radius $r \in \R_+$,
if for every $x \in \pa \Omega$ we find open balls $B_r(x_-) \subset \Omega$ and $B_r(x_+) \subset \R^2 \setminus \Omega$
such that $\pa B_r(x_-) \cap \pa B_r(x_+)=\{x\}$. 
\end{defi}
Note that the definition generalizes into the $n$-dimensional case and $\Omega$ can be replaced with any non-empty set. Study of such sets can be traced back to Federer \cite{FedererCurvature}.
We also mention \cite{Fu} as a general framwork.

In the context of bounded domains the uniform ball condition equals to $C^{1,1}$-regularity.
Indeed, the UBC with $r$ for $\Omega$ implies $C^{1,1}$-regularity with explicit bounds depending only on $r$. For a quick proof see e.g. \cite[Prop 2.7]{JN2023}.
Conversely, if $\Omega$ is $C^{1,1}$-regular, it satisfies the UBC with some radius $r \in \R_+$.  
Then $x_- = x - r \nu (x)$, $x_+ = x + r \nu (x)$ in Definition \ref{def:UBC} and it holds  
\begin{equation}
\label{est:Lipnu}
\left[\nu \right]_{C^{0,1}(\pa \Omega)} \leq 1/r \quad  \text{w.r.t.\ the Euclidean metric}.
\end{equation}

The \emph{maximal UBC radius} for a $C^{1,1}$-regular domain $\Omega$ is defined as
\begin{equation}
r_\Omega = \sup \{r \in \R_+ : \Omega \text{ satisfies UBC with } r\}.
\end{equation}
Since $\Omega$ is bounded, the number $r_\Omega$ is finite 
and can be estimated 
as follows
\beq
\label{est:rOmega}
r_\Omega \leq \diam(\Omega) \leq \P(\Omega) \qquad \text{and} \qquad 
r_\Omega \leq \sqrt{|\Omega|/\pi}.
\eeq  

We are mainly interested in the classical case that $\Omega$ is $C^k$-regular with $k \geq 2$. In this case we infer $|\curv| \leq 1/r_\Omega$  from \eqref{est:Lipnu}. In fact the only other bound on $r_\Omega$ is given by lack of separation of (intrinsically) distant parts of the boundary, which allows us to conclude the following:

\begin{prop}\label{prop:UBC} If $\Omega \subset \R^2$ is a $C^2$-regular domain, at least one of the following conditions holds true.
\begin{itemize}
\item[(i)] $\|\curv\|_{L^\infty(\pa \Omega)}= 1/r_\Omega$
\item[(ii)] There are $x,y \in \pa \Omega$ such that $y=x\pm2r_\Omega \nu (x)$ and $\n(x)=-\n(y)$.
\end{itemize}
\end{prop}
See Figure \ref{fig:uniformBall} for an illustration. Note that for the validity of Proposition \ref{prop:UBC} it is 
necessary to require $\Omega$ to be bounded.
 The geometric interpretation of the case  $r_\Omega \ll 1/\|\curv\|_{L^\infty(\pa \Omega)}$ is that $\Omega$ has a relatively thin neck or slit compared to turning radius of the boundary. 

The uniform ball condition for $\Omega$ can be equivalently seen as regularity of distance to $\Omega$.
We recall that the \emph{signed distance function} $d_\Omega:\R^2 \rightarrow \R$ of a given domain $\Omega$ is given by the rule
\[
d_\Omega (x) =
\begin{cases}
\dist(x,\Omega), &\text{if $x$ in $\R^2 \setminus \Omega$} \\
-\dist(x,\Omega), &\text{if $x$ in $\Omega$}.
\end{cases}
\] 
Then one may show that $\Omega$ satisfies UBC with a radius $r$ exactly when $d_\Omega$ is differentiable
in the \emph{tubular neighborhood}
\[
\mathcal N_r(\pa \Omega) = \{ x \in \R^2 : \dist(x,\pa \Omega)<r \}.
\]
Again, this is equivalent to say that the least distance projection $\pi_{\pa \Omega}$ is well-defined from $\mathcal N_r(\pa \Omega)$ onto $\pa \Omega$.  In this case, the map $\pa \Omega \times (-r,r) \rightarrow \mathcal N_r(\pa \Omega)$, $(x,\rho) \mapsto x +\rho \, \nu (x)$,
is a bijection with the inverse $z \mapsto \left(\pi_{\pa \Omega} (z), d_\Omega(z)\right)$ which, in turn, results in the identities 
\begin{equation}
\label{eq:sgndistance1}
\id = \pi_{\pa \Omega} + d_\Omega  \nabla d_\Omega \qquad \text{and} \qquad \nabla d_\Omega = \nu \circ \pi_{\pa \Omega}  \qquad \text{in} \ \ \mathcal N_r(\pa \Omega).
\end{equation}
We also remark that each sublevel set $\{d_\Omega < \rho\}$, $\rho \in (-r,r)$, satisfies UBC with a radius of at least $r - |\rho|$. 

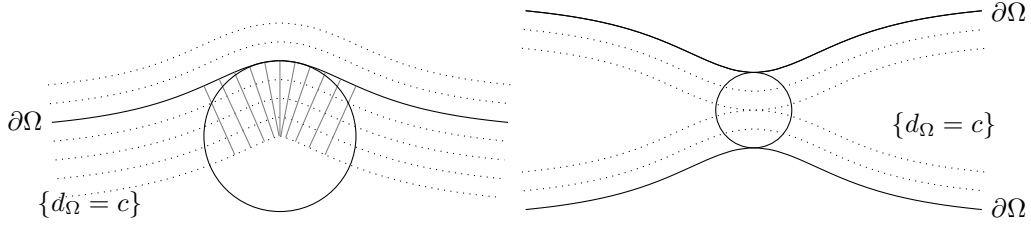
\begin{figure}[ht]
\begin{center}
 \begin{tikzpicture}
  
  \draw plot[smooth,domain=-3:3] (\x,{1/(1+.5*\x*\x)});
  \draw[dotted] plot[smooth,domain=-3:3] ({\x-.25*\x/sqrt((1+.5*\x*\x)^4+\x*\x)  },{1/(1+.5*\x*\x)-.25/sqrt(1+\x*\x/(1+.5*\x*\x)^4) });
  \draw[dotted] plot[smooth,domain=-3:3] ({\x-.5*\x/sqrt((1+.5*\x*\x)^4+\x*\x)  },{1/(1+.5*\x*\x)-.5/sqrt(1+\x*\x/(1+.5*\x*\x)^4) });
  \draw[dotted] plot[smooth,domain=-3:3] ({\x-.75*\x/sqrt((1+.5*\x*\x)^4+\x*\x)  },{1/(1+.5*\x*\x)-.75/sqrt(1+\x*\x/(1+.5*\x*\x)^4) });
  \draw[dotted] plot[smooth,domain=-3:3] ({\x-\x/sqrt((1+.5*\x*\x)^4+\x*\x)  },{1/(1+.5*\x*\x)-1/sqrt(1+\x*\x/(1+.5*\x*\x)^4) });
  
  \draw[dotted] plot[smooth,domain=-3:3] ({\x+.25*\x/sqrt((1+.5*\x*\x)^4+\x*\x)  },{1/(1+.5*\x*\x)+.25/sqrt(1+\x*\x/(1+.5*\x*\x)^4) });
  \draw[dotted] plot[smooth,domain=-3:3] ({\x+.5*\x/sqrt((1+.5*\x*\x)^4+\x*\x)  },{1/(1+.5*\x*\x)+.5/sqrt(1+\x*\x/(1+.5*\x*\x)^4) });
  
  \foreach \x in {-1,-.8,...,1} {
  \draw[gray] (\x,{1/(1+.5*\x*\x)}) -- ({\x-\x/sqrt((1+.5*\x*\x)^4+\x*\x)  },{1/(1+.5*\x*\x)-1/sqrt(1+\x*\x/(1+.5*\x*\x)^4) });
}
  
  \draw (0,0) circle (1);
  
  \node[left] at (-3,.2) {$\partial \Omega$};
  \node[below] at (-2.5,-.6) {$\{d_\Omega = c\}$};
 \end{tikzpicture}
 \begin{tikzpicture}
  \draw plot[smooth,domain=-3:3] (\x,{1/(1+.5*\x*\x)});
  \draw plot[smooth,domain=-3:3] (\x,{3-1/(1+.5*\x*\x)});
  \draw[dotted] plot[smooth,domain=-3:3] ({\x+.25*\x/sqrt((1+.5*\x*\x)^4+\x*\x)  },{1/(1+.5*\x*\x)+.25/sqrt(1+\x*\x/(1+.5*\x*\x)^4) });
  \draw[dotted] plot[smooth,domain=-3:3] ({\x+.5*\x/sqrt((1+.5*\x*\x)^4+\x*\x)  },{1/(1+.5*\x*\x)+.5/sqrt(1+\x*\x/(1+.5*\x*\x)^4) });
  \draw plot[smooth,domain=-3:3] (\x,{3-1/(1+.5*\x*\x)});
  \draw[dotted] plot[smooth,domain=-3:3] ({\x+.25*\x/sqrt((1+.5*\x*\x)^4+\x*\x)  },{3-1/(1+.5*\x*\x)-.25/sqrt(1+\x*\x/(1+.5*\x*\x)^4) });
  \draw[dotted] plot[smooth,domain=-3:3] ({\x+.5*\x/sqrt((1+.5*\x*\x)^4+\x*\x)  },{3-1/(1+.5*\x*\x)-.5/sqrt(1+\x*\x/(1+.5*\x*\x)^4) });

  \draw (0,1.5) circle (.5);

  \node[right] at (3,2.8) {$\partial \Omega$};
  \node[right] at (3,.2) {$\partial \Omega$};
  \node[above] at (2.5,1) {$\{d_\Omega = c\}$};
 \end{tikzpicture}
 \end{center}
 
 \caption{
 \label{fig:uniformBall} Relationship between uniform ball condition, regularity of level sets for the signed distance function and existence of normal tubular neighborhoods. Note how in the left picture, the degeneracy is entirely due to the curvature, while in the right picture it is instead a result of different parts of $\partial \Omega$ getting too close.}
\end{figure}

To make notation lighter we simply abbreviate $\mathcal N(\pa \Omega) := \mathcal N_{r_\Omega} (\pa \Omega )$. Thus $\mathcal N(\pa \Omega)$ is the largest tubular neighborhood in which $d_\Omega$ is differentiable. It can be shown that
that
$d_\Omega \in C^{1,1}_\loc(\mathcal N(\pa \Omega))$ and $\pi_{\pa\Omega} \in C^{0,1}_\loc(\mathcal N(\pa \Omega);\R^2)$, see \cite[Thm 4.8]{FedererCurvature}.

Within this neighborhood we can consistently define the \emph{extended normal} and \emph{extended tangent} vectors in $\mathcal N_r(\partial \Omega) \to \R^2$ as
the extension $\n \circ \pi_{\pa \Omega}$ and  $\tau \circ \pi_{\pa \Omega}$ respectively. By slightly abusing the notation we denote these extension by $\nu$ and $\tau$.
Thus  $\{\n,\tau\}$ forms an orthonormal, righthanded coordinate frame in $\mathcal N(\pa \Omega)$.  Note that while the directions of these vectors correspond to the partial derivatives of $(x,\rho) \mapsto x +\rho \, \nu (x)$, the latter will not have unit length in tangential direction. For a detailed discussion of this and some of the resulting identities, we refer to Appendix \ref{app:ExtNormal}.
Again, for more extensive study of the (signed) distance and related topics we refer to 
\cite{Bellettini}.

For many of the calculations, it is useful to have a cutoff-function that allows us to split $\Omega$ into a tubular neighborhood of the boundary and the rest of the bulk in a controlled manner.

\begin{defi}[Cutoff function] Let $\Omega$ be a UBC domain. We define a cutoff $\eta :\R^2 \to [0,1]$ supported in $\mathcal N(\pa \Omega)$ by setting
\begin{equation}
\label{eta}
\eta =  \phi \circ (d_\Omega/r_\Omega),
\end{equation}
where we choose a fixed symmetric bump function $\phi \in C^\infty_0(\R)$ s.t.  
\begin{itemize}
\item[(i)] $\phi$ is non-increasing in $(0,\infty)$ and
\item[(ii)] $\phi = 1$ in $(0,1/4)$  and $\phi=0$ in $(1/2,\infty)$.
\end{itemize} 
\end{defi}

\begin{lemma}[Properties of the cutoff] Let $\Omega$ and $\eta$ be as in the previous definition.
 The cutoff function $\eta$ satisfies the following properties in $\mathcal N(\pa \Omega)$:
\begin{equation}
\label{nablaeta}
\nabla \eta = \pa_{\n} \eta \,  \n,
\end{equation}
\begin{equation}
\label{diffeta1}
\pa_{\n}^k \eta = \frac{\phi^{(k)} \circ (d_\Omega/r_\Omega)}{r_\Omega^k} \qquad \text{and} \qquad  \pa_{\tau} \pa_{\n}^k \eta = 0 \qquad \text{for every $k=0,1,2, \ldots$}
\end{equation}
As well as
\begin{equation}
\label{diffeta2}
|\pa_{\n} \eta| \leq \frac{C}{r_\Omega} \qquad \text{and} \qquad |\pa_{\n}^2 \eta| \leq \frac{C}{r_\Omega^2}
\end{equation}
where $C$ is an independent constant.
\end{lemma}

\begin{proof}
Differentiating \eqref{eta} in $\mathcal N(\pa \Omega)$ gives us
the first three identities. The second identity immediately yields
\eqref{diffeta2}. Then $C=\max\{\max |\phi'|, \max |\phi''|\}$ depends only on the choice of $\phi$.
\end{proof}

Moreover, it follows from \eqref{eta} and (ii) that 
\[
\spt \eta \,  \subset \mathcal N_{r_\Omega/2}(\pa \Omega).
\]
If $\Omega$ is $C^2$-regular, then recalling the curvature bound $|\curv| \leq 1 /r_\Omega$ and \eqref{divn2D} we have
\begin{align}
\label{div n est}
|\div \, \n| \leq \frac2{r_\Omega} \qquad \text{in} \ \  \spt \eta.
\end{align}

\begin{rema}
\label{eta:rem}
By replacing $\eta$ with $\eta^2$ and $C$ with $2C + 2C^2$ we may assume that that $\eta^\frac12$  satisfies the same properties as $\eta$. 
\end{rema}


\subsection{Domain independent integral inequalities}
\label{subsec:domainIndep}

In this section we state a few integral inequalities which are well known in the case of a fixed domain. However, in those inequalities the constants are domain dependent. As we are dealing with families of changing domains and many of those constants will actually diverge for bad domains, e.g.\ when the 
uniform ball condition collapses, we need to state them explicitly in terms of 
maximal UBC radius and other geometric parameters. Although their proofs are not that different from the standard proofs, we need to quantify the constants explicitly and, hence, we will provide them in the appendix.

We start with the standard Poincaré inequality. 

\begin{lemma}[Poincaré inequality] \label{lem:Poincare}
 Let $\Omega \subset \R^2$ be a UBC domain. There exists an independent constant $C$ such that
 \begin{align}
\label{Poincare:est}
  \int_\Omega \abs{f - f_\Omega}^2 \d x \leq C \frac{|\Omega| P(\Omega)}{r_\Omega}
\int_\Omega \abs{\nabla f}^2 \d x
 \end{align}
 for every $f\in H^1(\Omega)$.
\end{lemma}

It should be noted that the scaling in terms of $r_\Omega$ is optimal in the previous estimate. This can be seen true by considering a suitable 
dumbbell-shaped domain. We also need the following basic interpolation inequalities.

\begin{lemma}[Gagliardo-Nirenberg]
\label{lem:interpolation}
 Let $\Omega \subset \R^2$ be a UBC domain. There exists an independent constant $C$
such that
\beq
\label{interpolation:L2}
\norm{\nabla f}_{L^2(\Omega)}\leq C\left(\norm{\nabla^2 f}_{L^2(\Omega)}^\frac12\norm{f}_{L^2(\Omega)}^\frac12+
\frac{\norm{f}_{L^2(\Omega)}}{r_\Omega}\right)
\eeq
and
\beq
\label{interpolation:Linfty}
\| f\|_{L^\infty(\Omega)} \leq C\left( \|\nabla^2 f\|_{L^2(\Omega)}^\frac12 \|f\|_{L^2(\Omega)}^\frac12 + \frac{\|f\|_{L^2(\Omega)}}{r_\Omega} \right)
\eeq
for every $f \in H^2(\Omega)$.
\end{lemma}

Here we remark that the first interpolation inequality holds generally in arbitrary dimension with the same scaling and
$C$ as a dimensional constant. For second estimate the uniform ball condition can be replaced with weaker \emph{uniform interior ball condition}, i.e., for fixed $r \in \R_+$ every $x \in \Omega$ can be contained in a ball of radius $r$ which, in turn, is a subset of
$\Omega$.

Since we are dealing with Stokes systems, we need the following basic interior regularity estimate.

\begin{lemma}[Interior regularity estimate]
\label{lem:interiorreg}
 Let $\Omega \subset \R^2$ be a domain. There exists an independent constant $C$
such that 
\beq
\label{est:interiorreg}
\int_{\{d_\Omega < - \rho\}} |\nabla^3 X|^2 \d x \leq C \int_{\Omega} \frac1{\rho^2} |\nabla^2 X|^2
+\frac1{\rho^4} |\nabla X|^2+\frac1{\rho^6} |X|^2 \d x
\eeq
for every $\rho \in \R_+$ with a non-empy sublevel set $\{d_\Omega < - \rho\}$ and weakly biharmonic vector field $X \in H^2(\Omega;\R^2)$.
\end{lemma}

Finally, as the natural energy estimates for the Stokes equation involves the symmetrized gradient, we need properly quantified versions of the corresponding Korn-inequalities.

\begin{lemma}[Korn's inequality]
\label{lem:Korn1}
Let $\Omega \subset \R^2$ be a UBC domain. There exists an independent constant $C$
such that
\begin{equation}
\label{Korn:est1}
\int_{\Omega}|\nabla X|^2 \d x \leq C \left(\int_{\Omega}|\nablasym X|^2 \d x  +\frac1{r_\Omega^2}\int_\Omega |X|^2 \d x\right)
\end{equation}
 for every $X \in H^1(\Omega;\R^2)$. Moreover, if $\eta$ is the cutoff for $\Omega$ defined in \eqref{eta}, then 
\begin{align}
\label{Korn_cutoff:est} 
\int_\Omega \eta \, |\nabla X|^2 \d x
\leq C\left( \int_\Omega \eta \, |\nablasym X|^2 \d x +\frac1{r_\Omega^2}\int_{\Omega \cap \, \spt \eta}  |X|^2  \d x\right)
\end{align}
for every $X \in H^1\left( \Omega \cap \mathcal N_{r_\Omega/2}(\pa \Omega);\R^2\right)$.
\end{lemma}

\begin{lemma}[Korn-Poincaré inequality]
\label{lem:Korn2}
 Let $\Omega \subset \R^2$ be a UBC domain. There exists an independent constant $C$ 
such that
 \begin{align}
\label{Korn:est2}
\int_\Omega \abs{\nabla X - ( \nabla X )_\Omega}^2 \d x \leq  C \frac{|\Omega|P(\Omega)}{r_\Omega^3}  \int_\Omega \abs{\nablasym X - ( \nablasym X)_\Omega}^2 \d x
 \end{align}
 for all vector fields $X \in H^1(\Omega;\R^2)$ and
 \begin{align}
\label{Korn:est3}
  \int_\Omega |\nabla X|^2 \d x\leq C \frac{|\Omega|P(\Omega)}{r_\Omega^3}  \int_\Omega |\nablasym X|^2 \d x
 \end{align}
 for every $X \in H^1(\Omega;\R^2)$ with symmetric $\int_\Omega \nabla X \d x$.
\end{lemma}


\section{Concept of solution}
\label{sec:sol}


\subsection{Evolution of free surfaces under Stokes flow}
\label{ssec:stokes}

We will now try to give a precise definition of a solution to \eqref{eq:main}. For this we first need to define what is meant by normal velocity. See also Figure \ref{fig:normalVelocity}.

\begin{figure}[ht]

\tikzmath{\a=0.25*pi;}
\tikzmath{\b=-0.5*pi;}
\begin{center}
\begin{tikzpicture}
\draw[thick] plot[smooth,domain=-pi:pi]  ({\x}, {((\x+1)/(-0.3*cos((\x*180)/pi)+1))+0.5)});
\node[scale=1, right] at (pi,3.7) {$\partial \Omega_t$};

\draw[dotted, thick] plot[smooth,domain=-pi:pi]  ({\x}, {((\x+1)/(-0.65*cos((\x*180)/pi)+1))+0.5)-0.5});
\node[scale=1, right] at (pi,2.6) {$\partial \Omega_{t-h}$};

\filldraw[black] ({\a},{((\a+1)/(-0.3*cos((\a*180)/pi)+1))+0.5)}) circle (1pt);
\node[scale=1, left] at ({\a},{((\a+1)/(-0.3*cos((\a*180)/pi)+1))+0.5)+0.1}) {$x$};

\filldraw[black] ({\b},{((\b+1)/(-0.3*cos((\b*180)/pi)+1))+0.5)}) circle (1pt);
\node[scale=1, right] at ({\b},{((\b+1)/(-0.3*cos((\b*180)/pi)+1))+0.5)-0.1}) {$y$};


\draw[thick,->]({\a},{((\a+1)/(-0.3*cos((\a*180)/pi)+1))+0.5)})
--  ({\a+1.5*((1-0.3*cos((\a*180)/pi)-0.3*(\a+1)*sin((\a*180)/pi))/((1-0.3*cos((\a*180)/pi))^2))},{-1.5+((\a+1)/(-0.3*cos((\a*180)/pi)+1))+0.5});
\node[scale=1, below] at ({\a+1.5*((1-0.3*cos((\a*180)/pi)-0.3*(\a+1)*sin((\a*180)/pi))/((1-0.3*cos((\a*180)/pi))^2))},{-1.5+((\a+1)/(-0.3*cos((\a*180)/pi)+1))+0.5}) {$v(t,x)\nu_t(x)$};

\draw[thick,->]({\b},{((\b+1)/(-0.3*cos((\b*180)/pi)+1))+0.5)})
--  ({\b- ((1-0.3*cos((\b*180)/pi)-0.3*(\b+1)*sin((\b*180)/pi))/((1-0.3*cos((\b*180)/pi))^2))},{1+((\b+1)/(-0.3*cos((\b*180)/pi)+1))+0.5});
\node[scale=1, above] at  ({\b- ((1-0.3*cos((\b*180)/pi)-0.3*(\b+1)*sin((\b*180)/pi))/((1-0.3*cos((\b*180)/pi))^2))},{1+((\b+1)/(-0.3*cos((\b*180)/pi)+1))+0.5}) {$v(t,y)\nu_t(y)$};
\end{tikzpicture}
\end{center}
 \caption{
 \label{fig:normalVelocity} An illustration of normal velocity and the evolution of a boundary.}
\end{figure}
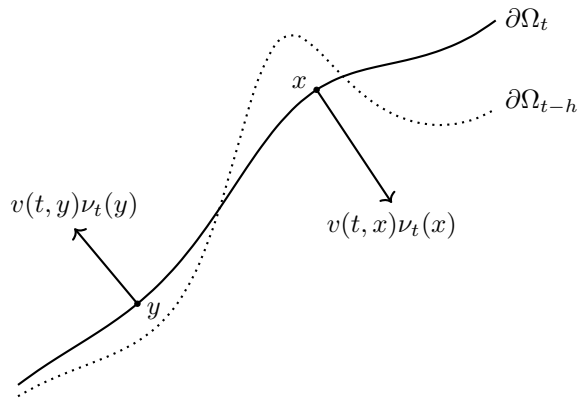

\begin{defi}[Normal velocity]
\label{def:normalvelocity}
Let $I\subset \R$ be an interval and $(\Omega_t)_{t \in I} \subset \R^2$ a time parametrized family of $C^1$-regular domains. Assume that for fixed $t \in I$, $x \in \partial \Omega_{t}$ and all $h \in \R$ with $|h|$ small enough there exists a point $z \in \partial \Omega_{t+h}$ such that $z = x + s_{t,x}(h) \n_t(x)$, where $s_{t,x}(h) \in \R$ is a unique number closest to zero for which this is the case. Then, assuming it exists, the derivative $v(t,x):= s_{t,x}'(0)$ is the \emph{normal velocity} of $(\Omega_t)_{t \in I}$ at $(t,x)$.
\end{defi}

Since we have now defined all the relevant quantities, we can give a precise definition of a solution.

\begin{defi}[Regular solution]
\label{def:classicsol}
For given an integer $k \geq 2$ we say that a triplet $(\Omega_t,u(t,\freearg),p(t,\freearg))_{t \in [0,T)}$ is a $C^k$-regular solution to \eqref{eq:main}
if the following conditions are satisfied.
\begin{itemize}
\item[(i)] Each $\Omega_t \subset \R^2$ is a $C^k$-regular domain for which $u(t,\freearg) \in  C^k(\overline \Omega_t;\R^2)$ and $p(t,\freearg) \in  C^{k-1}(\overline \Omega_t;\R^2)$ solve
the problem \eqref{eq:main} in the classical sense with a time-independent viscosity constant $\vartheta \in \R_+$. 
\item[(ii)]  Each $u(t,\freearg)$ is \emph{moment-preserving} meaning that 
$\int_{\Omega_t} u(t,\freearg) \d x = 0$ and $\int_{\Omega_t} \nabla u(t,\freearg) \d x$ is symmetric.
\item[(iii)] The parametrized family $(\Omega_t)_{t \in [0,T)}$ moves according to the normal velocity $v (t,\freearg) = \langle u(t,\freearg), \n_t \rangle$ on $\pa \Omega_t$ for every $t \in [0,T)$.  
\item[(iv)] The map $u$ and its spatial derivatives $\nabla^l u$ up to $l=k$ are continuous in the time-space cylinder  $\{(t,x) : t \in [0,T) , x \in \overline \Omega_t  \}$.
\end{itemize}
\end{defi}

Unless otherwise stated, we use implicitly the abbreviations
\[
u:= u(t,\freearg), \ \ p:=p(t,\freearg), \ \ r_t := r_{\Omega_t},\ \ \pi_t := \pi_{\pa \Omega_t},  \ \ \text{and} \ \ \P_t := \P(\Omega_t) 
\]
for a solution according to the previous definition. 

As we neglect the influence of inertia, it is also not surprising that our problem is invariant under Galilean transformations.

\begin{rema}[Galilean transformations]
\label{rem:rigidmotion}
Assume that $(\Omega_t,u,p)_{t \in [0,T)}$ is a $C^k$-regular solution, except possibly not satisfying the moment preservation. Now let $R:[0,T) \to SO(2)$
and $a:[0,T) \to \R^2$ be sufficiently regular maps. If we define
 \begin{align*}
  \widetilde{\Omega}_t &:= R(t) \Omega_t + a(t) \\
  \tilde{u}(t,R(t)x+a(t)) &:= R'(t)x+ u(t,x) +a'(t) \\
  \tilde{p}(t,R(t)x+a(t)) &:= p(t,x)
 \end{align*}
then it can be shown that $(\widetilde{\Omega}_t,\tilde{u},\tilde{p})_{t \in [0,T)}$ is also a $C^k$-regular solution, except for the moment preservation (the condition (ii)).
Conversely though, as we can substract the average movement and rotation (note that the antisymmetric matrices are exactly the tangent space of $SO(2)$), for any classical solution, we can find a moment-preserving solution with the same initial data by a Galilean transformation.
\end{rema}

In fact, this is not only the physically expected invariance but also (at least on a formal level) the precise characterisation of the mathematical non-uniqueness inherent in the problem. First of all, note that the only initial data for the problem is the intial domain $\Omega_0$. In fact for an arbitrary fixed time $t$, the sections $u(t,\freearg)$ and $p(t,\freearg)$ can be uniquely determined by $\Omega_t$.  This follows directly from the standard theory of the Stokes equation.  In particular, for the regularity, see e.g. \cite[Thm IV. 4.3]{Galdi}. 

\begin{lemma}[Uniqueness of $u(t,\freearg)$ and $p(t,\freearg)$]
 Let \solT be a $C^k$-regular solution and $t \in [0,T)$. Then the pair $(u(t,\freearg),p(t,\freearg))$ is the unique solution to the Stokes problem with Neumann data:
\begin{equation*}
 \left\{
 \begin{aligned}
  \vartheta\Delta \tilde{u} -\nabla \tilde{p} &= 0 &\text{ in }& \Omega_t \\
  \div \tilde{u} &= 0 &\text{ in }& \Omega_t \\
\curv_t &=\tilde p - 2\vartheta \,  \div_{\n_t} \tilde{u} &\text{ in }& \partial \Omega_t \\
0&=  (I - \n_t \otimes \n_t) [\nablasym \tilde{u}] \n_t &\text{ in }& \partial\Omega_t
 \end{aligned}
 \right.
\end{equation*}
in the class $\int_{\Omega_t} \tilde{u} \d x = 0$ and $\int_{\Omega_t} \nabla \tilde{u} \d x$ symmetric. In addition we have that $u(t,\freearg) \in C^\infty(\Omega_t;\R^2)$ and $p(t,\freearg) \in C^\infty(\Omega_t)$.
\end{lemma}


\subsection{Gradient flow estimates}
\label{ssec:grad}
Let us note that formally our solution has the structure of a gradient flow in the space of domains, where we take the energy to be the perimeter $\P_t$ and the distance/dissipation functional to be $2 \vartheta  \int_{\Omega_t} |\nablasym u|^2 \d x$.

\begin{lemma}[Gradient flow structure] \label{lem:gradFlow}
 Let $(\Omega_t, u,p)_{t \in [0,T)}$ be a $C^2$-regular solution. Then it satisfies
 \begin{align}
\label{eq:dissipation}
  \frac{\d}{\d t} \P_t = \int_{\pa \Omega_t} \curv_t \langle u, \n_t \rangle \d t = - 2 \vartheta \int_{\Omega_t} |\nablasym u|^2 \d x
 \end{align}
 for all $t\in [0,T)$.
\end{lemma}

\begin{proof}
First we note that since $u$ is divergence free in $\Omega_t$  we may write
\[
\Delta u = \div \nabla u =   \div \nabla u + \nabla \div u = \div \nabla u  + \div (\nabla u)^\T= 2 \div \nablasym u
\ \ \text{in} \ \ \Omega_t.
\]
Then we test the Stokes equation with $u$ itself, use the previous identity and the divergence theorem as well as recall
the boundary condition in \eqref{eq:main}
\begin{align*}
  0 &= \int_{\Omega_t} \inner{\vartheta \Delta u -\nabla p}{u} \d x \\
 &= \int_{\Omega_t} \langle 2\vartheta \div \nablasym u -\nabla p,u\rangle \d x\\
 &= \int_{\Omega_t} \div\left(2\vartheta [\nablasym u] u\right) - 2\vartheta \, \tr(\nablasym u \nabla u)
-\div(pu) + p \div u \d x\\
  &= \int_{\partial \Omega_t} \langle [2\vartheta\nablasym u-pI]u, \n_t\rangle \d \H^1 -2 \vartheta  \int_{\Omega_t} 
\langle \nablasym u,\nablasym u\rangle \d x\\
 &= \int_{\partial \Omega_t} \langle u, [2\vartheta\nablasym u-pI]\n_t\rangle \d \H^1 -2 \vartheta  \int_{\Omega_t} 
|\nablasym u |^2 \d x\\
   &=- \int_{\partial \Omega_t} \curv_t \inner{u}{\n_t} \d \H^1 -2 \vartheta  \int_{\Omega_t} 
|\nablasym u |^2 \d x
 \end{align*} 
 and we conclude proof by recalling the first variation of perimeter \eqref{dP/dt}. 
\end{proof}

We further observe that testing the Stokes equation in the previous proof with any a.e.\ divergence free
$X \in W^{1,1}(\pa \Omega_t)$ yields more generally
\[
 \int_{\partial \Omega_t} \curv_t \inner{X}{\n_t} \d \H^1 = - 2 \vartheta \int_{\Omega_t} \langle \nablasym u, \nablasym X \rangle
\d x
\]
where $X$ at $\pa \Omega_t$ as a $L^1(\pa \Omega_t;\R^2)$-function is given by the trace operator associated to $\Omega_t$.
Additionally, if $\langle X,\n_t \rangle =\langle u,\n_t \rangle$ a.e.\ in the sense of the $\H^1$-measure on $\pa \Omega_t$, we conclude
 \begin{align*}
  \frac{\d}{\d t} \P_t \geq -2 \vartheta\int_{\Omega_t} |\nablasym X|^2 \d x
 \end{align*}
with equality precisely when $u(t,\freearg)-X$ modulo translation equals (almost everywhere) to an antisymmetric linear map.

As a direct consequence of the gradient flow structure
we have monotonicity of perimeter and a $L^2L^2$-control over the symmetric energy of a classical solution $(\Omega_t, u,p)_{t \in [0,T)}$, that is, 
\begin{align}
\label{eq:dissipation2}
P_s =P_t+2\vartheta \int_s^t\norm{\nablasym u}_{L^2(\Omega_h)}^2 \d h
\end{align}
for all $0 \leq s \leq t < T$. By combining this with the estimate \eqref{est:rOmega} 
we conclude
\beq
\label{est:r_t}
r_t \leq \P_0 .
\eeq
Note that the moment-preserving condition 
yields a quantitative control of the $H^1$-energy by the symmetric energy.
 Indeed, since $u_{\Omega_t} = 0$ and $\int_{\Omega_t} \nabla u \d x$ is symmetric, 
the quantified Poincaré inequality \eqref{Poincare:est}, the Korn-Poincaré inequality 
\eqref{Korn:est3} and $|\Omega_t|=|\Omega_0|$ combined yield
 yield
\beq
\label{est:H1control}
\|u\|_{L^2(\Omega_t)}^2 \leq C \frac{|\Omega_0|P_t}{r_t}\norm{\nabla u}_{L^2(\Omega_t)}^2
\ \ \ \text{and} \ \ \ \|\nabla u\|_{L^2(\Omega_t)}^2 \leq  C \frac{|\Omega_0|P_t}{r_t^3} \norm{\nablasym u}_{L^2(\Omega_t)}^2
\eeq
We also observe that the dissipation equation \eqref{eq:dissipation} 
allows us to control the symmetric energy itself at every time in terms of the geometry of evolving  domain, thus resulting in the total control of the $H^1$-energy.
\begin{prop}
\label{prop:H1control}
Let $(\Omega_t,u,p)_{t \in [0,T)}$ be a $C^2$-regular solution. There exists an independent constant $C$
such that for every time $t \in [0,T)$ it holds
\beq
\label{est:symmcontrol}
\vartheta \norm{\nablasym u}_{L^2(\Omega_t)} \leq C \frac{|\Omega_0|P_t}{r^{5/2}_t}
\|\curv_t\|_{L^2(\pa \Omega_t)}.
\eeq
\end{prop}

\begin{proof}
The dissipation equation \eqref{eq:dissipation} yields
\[
\vartheta \|\nablasym u\|_{L^2(\Omega_t)}^2 \leq \|\curv_t\|_{L^2(\pa \Omega_t)}
 \|\langle u, \n_t \rangle \|_{L^2(\pa \Omega_t)}.
\]
Let $\eta$ be the cutoff for $\Omega_t$ defined in \eqref{eta}. 
Then using the identities \eqref{basicid}, \eqref{div(f*n)} and the estimates \eqref{est:rOmega},
\eqref{diffeta2}, \eqref{div n est} and \eqref{est:H1control} 
we have
\begin{align*}
\int_{\pa \Omega_t} \langle u,\nu_t \rangle^2 \d \H^1 
&=\int_{\Omega_t} 2 \eta \, \langle u,\nu_t \rangle \langle \pa_{\n_t} u,\nu_t \rangle + \pa_{\n_t} \eta \langle u,\nu_t \rangle^2 
+ \div \n_t \, \eta \, \langle u,\n_t \rangle^2 \d x \\
&\leq C \norm{u}_{L^2(\Omega_t)} \norm{\nablasym u}_{L^2(\Omega_t)} +  \frac{C}{r_t}
\norm{u}_{L^2(\Omega_t)}^2\\
&\leq C \frac{|\Omega_0| P_t }{r_t^2}\norm{\nablasym u}_{L^2(\Omega_t)}^2 +  C \frac{|\Omega_0|^2P_t^2}{r_t^5}\norm{\nablasym u}_{L^2(\Omega_t)}^2 \\
&\leq C \frac{|\Omega_0|^2P_t^2}{r_t^5}\norm{\nablasym u}_{L^2(\Omega_t)}^2.
\end{align*}
Thus combining the estimates concludes the proof.
\end{proof}


\subsection{A simple example of evolution}

We close this section with an example that shows that there is indeed a change in topology. 
For simplicity we assume $\vartheta = 1$.

\begin{exam}[Evolution of an annulus] \label{example}
 Consider an annulus $\Omega$ of inner radius $l$ and outer radius $L$, i.e.\  $\Omega = B_L \setminus \overline{B_l}$. We first solve the Stokes problem
such that the solution is moment-preserving in $\Omega$. 
 
 By radial symmetry of $\Omega$, the only solution that does not result in additional rotation or translation has to have the same radial symmetry. As the velocity $u$ has to be divergence free, this means that $u := \lambda\id / |\id|^2$ for some $ \lambda \in \R$. Then by a quick calculation
 \begin{align*}
  \nabla u  = \frac{ \lambda}{|\id|^2} - 2 \lambda \frac{\id \otimes \id}{|\id|^4} \quad \text{ and } \quad \Delta u = 0 \ \ \text{in} \ \ \Omega.
 \end{align*}
 As a result the pressure is given by a constant $p$. We can now study the boundary conditions at the inner and outer radius where $\curv=-1/l$ and $\curv=1/L$ respectively. There we have 
$[\nablasym u] \n = - \lambda \n /|\id|^2$. Then the boundary condition 
$p \n - 2[\nablasym u]\n = \curv \n$
yields the equations
 \begin{align*}
  p  +  \lambda\frac{2}{l^2} = -\frac{1}{l} \quad \text{ and } \quad
  p  +  \lambda\frac{2}{L^2} = \frac{1}{L}
 \end{align*}
and thus $ \lambda =-\frac{Ll}{2(L-l)}$ and $p = \frac{1}{L-l}$.

Then we consider a time evolving annulus $\Omega_t$ where $l=l(t)$ and $L=L(t)$.
Now $v(t,\freearg)=\langle u, \n_t \rangle$ so $v(t,\freearg)= - \lambda/l$ 
and $v(t,\freearg)= \lambda/L$ on the inner and outer rings respectively.
On the other hand, they must equal to the rates of $-l$ and $L$.
Combining these gives us the changes of radii:
 \begin{align*}
  \frac{\d}{\d t}l = \frac{\lambda}{l} = -\frac{L}{2(L-l)} \quad \text{ and } \quad
  \frac{\d}{\d t} L = \frac{\lambda}{L} = -\frac{l}{2(L-l)}.
 \end{align*}
 Note that also the enclosed area $A:= \pi(L^2-l^2)$ is conserved, so this can be rewritten as one-dimensional ODEs
 \begin{align*}
  \frac{\d}{\d t}l = -\frac{\sqrt{A/\pi+l^2}}{2\sqrt{A/\pi+l^2} - 2 l} \quad \text{ and }\quad
  \frac{\d}{\d t} L =  -\frac{\sqrt{-A/\pi+L^2}}{2L-2\sqrt{-A/\pi+L^2}}
 \end{align*}
 which can be solved explicitly as
 \begin{align*}
  l(t) := \frac{A/\pi-(b_0+t/2)^2}{2b_0+t} \quad \text{ and }\quad
  L(t) := \frac{A/\pi + (b_0+t/2)^2}{2b_0+t}
 \end{align*}
 where $b_0 = \sqrt{A/\pi + l^2(0)}-l(0)$. We note that this corresponds to an actual solution as long as $l(t) > 0$, i.e. there is a collapse into a (stable) ball at $t_0 =2\sqrt{A/\pi} - 2b_0$. Note that prior to this point the normal velocity at the collapsing 
inner ring converges to $1/2$ but despite this we have
\[
\int_{\partial \Omega_t} \curv_t^2\, \d \H^1=\int_{\partial B_{l(t)}}\frac{1}{l^2(t)}\, \d \H^1+\int_{\partial B_{L(t)}}\frac{1}{L^2(t)}\, \d\H^1  \sim \frac{1}{l(t)},
\]
which becomes unbounded as $t \rightarrow t_0$ and thus demonstrates that in general no curvature estimate can be expected to hold unconditionally.
\end{exam}


\section{\texorpdfstring{$L^2$}{L2}-Curvature estimates}

In this section we estimate the rate
\[
\frac{\d}{\d t}  \int_{\pa \Omega_t} \curv_t^2 \d \H^1
\]
in terms of maximal UBC radius $r_t$ for the domain $\Omega_t$.
In the subsection, we generally assume that \solT is a $C^2$-regular solution. 

\begin{lemma} 
 Let \solT be a $C^2$-regular solution. Then it satisfies
\begin{align}
\label{1stvar_willmoreD}
\begin{split}
\frac{\d}{\d t}  \int_{\pa \Omega_t} \curv_t^2 \d \H^1
=&-\vartheta \int_{\Omega_t}  |\nabla^2 u|^2 \d x + 2 \vartheta \int_{\pa \Omega_t}  \div_{\n_t} u \, \langle \Delta  u , \n_t \rangle \d \H^1 \\
&- 4 \vartheta \int_{\pa \Omega_t} \curv_t (\div_{\tau_t} u)^2 \d \H^1-\int_{\pa \Omega_t} \curv_t^2 \, \div_{\tau_t} u   \d \H^1.
\end{split}
\end{align}
\end{lemma}

\begin{proof}
 We start from a general formula for the first variation of the energy $\|\curv_t\|_{L^2(\pa \Omega_t)}^2$ in formula \eqref{eq:k2derivative}, using that the boundary evolves with the normal part of $u$.
\begin{equation}
\label{1stvar_willmoreA}
\frac{\d}{\d t}  \int_{\pa \Omega_t} \curv_t^2 \d \H^1
=-\int_{\pa \Omega_t}  2\curv_t\left(\langle \pa_{\tau_t}^2 u , \n_t \rangle
+ \curv_t \div_{\tau_t} u \right) + \curv_t^2 \, \div_{\tau_t} u   \d \H^1.
\end{equation} 
Differentiating \eqref{eq:surfacetension2} in the tangential direction and using incompressibility as well as the decomposition \eqref{Delta X}
of the operator $\Delta$ restricted to $\pa \Omega_t$ yields the following technical identity
\begin{equation}
\label{<Delta u,n>}
2 \langle \pa_{\tau_t}^2 u, \n_t\rangle = \langle \Delta u, \n_t \rangle - 2 \curv_t \,\div _{\tau_t} u
\qquad 
\end{equation}
on $\pa \Omega_t$. Thus by substituting \eqref{<Delta u,n>} into \eqref{1stvar_willmoreA} and recalling the expression \eqref{eq:surfacetension1} for $p$ on $\pa \Omega_t$ we
conclude
\begin{equation}
\label{1stvar_willmoreB}
\frac{\d}{\d t}  \int_{\pa \Omega_t} \curv_t^2 \d \H^1
=-\int_{\pa \Omega_t}  p \langle \Delta  u , \n_t \rangle - 2 \vartheta \div_{\n_t} u \, \langle \Delta  u , \n_t \rangle
+ \curv_t^2 \, \div_{\tau_t} u 
\d \H^1.
\end{equation}
Since the incompressibility implies $\div \, \Delta u = \Delta\, \div u = 0$ in $\Omega_t$, then
by applying divergence  theorem\footnote{Since $\Delta u \in C^\infty(\Omega;\R^2)$, we have 
$\int_{\pa \{d_t <-\rho\}} g \langle \Delta u, \n_t \rangle \ \H^1 = \int_{\{d_t<-\rho\}} \langle 
\Delta u, \nabla g\rangle \d x$ for every $\rho \in (0,r_t)$ when $g \in C^{0,1}(\Omega)$.
Thus passing $\rho$ to zero yields $\int_{\pa \Omega_t} g \langle \Delta u, \n_t \rangle \ \H^1 = \int_{\Omega_t} \langle 
\Delta u, \nabla g\rangle \d x$. 
}
and Stokes equation in \eqref{eq:main} on the first term  we further obtain
\begin{align}
\label{1stvar_willmoreC}
\begin{split}
\frac{\d}{\d t}  \int_{\pa \Omega_t} \curv_t^2 \d \H^1
=&-\vartheta \int_{\Omega_t}  |\Delta u|^2 \d x + 2 \vartheta \int_{\pa \Omega_t}  \div_{\n_t} u \, \langle \Delta  u , \n_t \rangle \d \H^1 \\
&-\int_{\pa \Omega_t} \curv_t^2 \, \div_{\tau_t} u   \d \H^1.
\end{split}
\end{align}
Taking incompressibility and \eqref{eq:surfacetension2} into account, the two dimensional Reilly's formula \eqref{vReilly2D} for $u$ takes the form
\begin{equation}
\label{improved2DReilly}
\int_{\Omega_t} |\nabla^2 u|^2 - |\Delta u|^2 \d x = - 4  \int_{\pa \Omega_t} \curv_t (\div_{\tau_t} u)^2 \d \H^1.
\end{equation}
Hence substitution into \eqref{1stvar_willmoreC} yields the desired equality.
\end{proof}

 While at first glance the identity in the previous lemma seems very useful, it has the issue that its lower order terms are all on the boundary. Thus we cannot immediately estimate them. Instead we first need to transform them into the bulk. This is done in the following.

\begin{lemma}
Let \solT be a $C^2$-regular solution. Then this solution satisfies
\begin{align}
\label{1stvar_willmoreE}
\begin{split}
&\phantom{{}={}}\frac\d{\d t} \int_{\pa \Omega_t}  \curv_t^2 \d \H^1 \\
&=-\vartheta\int_{\Omega_t} \eta  \left(\langle \Delta u, \tau_t \rangle^2+\langle \pa_{\n_t}^2 u-\pa_{\tau_t}^2 u - \div \, \n_t \, \pa_{\n_t} u, \n_t \rangle^2\right) \d x \\
&\phantom{{}={}}-2\vartheta\int_{\Omega_t} \eta  \left( \langle \pa_{\tau_t}^2 u,\tau_t\rangle^2 + \langle \pa_{\n_t}\pa_{\tau_t} u, \n_t\rangle^2\right) \d x
-\vartheta\int_{\Omega_t} (1-\eta) \, |\nabla^2 u|^2 \d x \\
&\phantom{{}={}}+4\vartheta\int_{\Omega_t} \div \, \n_t \, \eta \,  \left \langle  \pa_{\n_t}^2 u, -2 \langle \pa_{\n_t} u, \n_t \rangle \n_t+
\langle \pa_{\tau_t} u, \tau_t\rangle 
 \tau_t \right\rangle \d x \\
&\phantom{{}={}}+2\vartheta\int_{\Omega_t} \pa_{\n_t} \eta \, \langle \pa_{\n_t} u, \n_t \rangle  \langle \Delta u , \n_t \rangle \d x
-\int_{\pa \Omega_t} \curv_t^2 \, \div_{\tau_t} u   \d \H^1
\\
&\phantom{{}={}}+2\vartheta\int_{\Omega_t} \eta \, ( \div \, \n_t \, \langle \pa_{\n_t} u, \tau_t \rangle )^2 \d x - 4\vartheta\int_{\Omega_t}
\div \, \n_t \, \pa_{\n_t} \eta \, \langle \pa_{\n_t} u, \n_t \rangle^2 \d x
 \\
&\phantom{{}={}}+\vartheta\int_{\Omega_t} \div \, \n_t \, \pa_{\n_t} \eta \left(\langle\pa_{\tau_t}  u, \n_t \rangle -
\langle\pa_{\n_t}  u, \tau_t \rangle\right)^2 \d x.
\end{split}
\end{align}
\end{lemma}

\begin{proof}
We transform the second term on the RHS of \eqref{1stvar_willmoreD}
into a bulk term by utilizing the cut-off function $\eta$ for $\Omega_t$ introduced in Section \ref{sec:UBC}.
By applying the divergence theorem on the second term  
and recalling $\div  \Delta u = 0$ and  \eqref{nablaeta}
we obtain
\begin{align}
\label{badterm1}
\begin{split}
&\phantom{{}={}}2 \vartheta \int_{\pa \Omega_t}  \div_{\n_t} u \, \langle \Delta  u , \n_t \rangle \d \H^1 \\
&= 2 \vartheta \int_{\Omega_t}  \eta \, \langle \nabla \div_{\n_t} u, \Delta  u\rangle \d x
+2\vartheta \int_{\Omega_t}  \div_{\n_t} u \, \pa_{\n_t} \eta \, \langle \n_t , \Delta  u\rangle \d x.
\end{split}
\end{align}
To decompose the first integral on RHS of \eqref{badterm1}, we use the following  
technical identities
\begin{equation}
\nabla \div_{\n_t} u =\langle \pa_{\n_t}^2 u, \n_t \rangle \n_t - \langle \pa_{\tau_t}^2 u, \tau_t \rangle \tau_t + \div \, \n_t \langle \pa_{\tau_t} u, \n_t \rangle \tau_t 
\end{equation}
and 
\begin{equation}
\langle\pa_{\tau_t}^2 u, \tau_t \rangle = - \pa_{\n_t} \langle \pa_{\tau_t} u, \n_t \rangle - \div \, \n_t \, \langle \pa_{\n_t} u, \tau_t \rangle
\end{equation}
in  $\Omega_t \cap \mathcal N_{r_t}(\pa \Omega_t$ which can be derived by using the incompressibility condition $\div u=0$ and the basic identities
\eqref{basicid} and \eqref{Delta X} for the extended
normal and tangent fields.

Then by using the previous identities
and decompositions \eqref{Delta X} and \eqref{|nabla^2X|^2} for $\Delta  u$ and $|\nabla^2 u|^2$ respectively  we may decompose $2\langle \nabla \div_{\n_t} u , \Delta u\rangle$ in $\Omega_t \cap \mathcal N(\pa \Omega_t)$ as follows
\begin{align}
\begin{split}
 &2\langle \nabla \div_{\n_t} u , \Delta u\rangle \\
= 
&-\left(\langle \Delta u, \tau_t \rangle^2+\langle \pa_{\n_t}^2 u-\pa_{\tau_t}^2 u - \div \, \n_t \, \pa_{\n_t} u, \n_t \rangle^2+ 2 \langle \pa_{\tau_t}^2 u,\tau_t\rangle^2 + 2 \langle \pa_{\n_t}\pa_{\tau_t} u, \n_t\rangle^2\right)
 \\
&+ \div \, \n_t \, \pa_{\n_t} \left(\langle\pa_{\n_t}  u, \tau_t \rangle^2 + 2 \langle\pa_{\n_t}  u, \tau_t \rangle \langle\pa_{\tau_t}  u, \n_t \rangle - \langle\pa_{\tau_t}  u, \n_t \rangle^2\right)
+2\, ( \div \, \n_t \, \langle \pa_{\n_t} u, \tau_t \rangle )^2 \\
&+|\nabla^2 u|^2
\end{split}
\end{align}
and further
\begin{align*}
 2\vartheta& \, \eta \, \langle \nabla \div_{\n_t} u , \Delta u\rangle \\
= 
&-\vartheta \, \eta \, \left(\langle \Delta u, \tau_t \rangle^2+\langle \pa_{\n_t}^2 u-\pa_{\tau_t}^2 u - \div \, \n_t \, \pa_{\n_t} u, \n_t \rangle^2+ 2 \langle \pa_{\tau_t}^2 u,\tau_t\rangle^2 + 2 \langle \pa_{\n_t}\pa_{\tau_t} u, \n_t\rangle^2\right)
 \\
&+ \vartheta \, \div \, \n_t \, \pa_{\n_t} \left(\eta \left(\langle\pa_{\n_t}  u, \tau_t \rangle^2 + 2 \langle\pa_{\n_t}  u, \tau_t \rangle \langle\pa_{\tau_t}  u, \n_t \rangle - \langle\pa_{\tau_t}  u, \n_t \rangle^2\right)\right)\\
&- \vartheta \, \div \, \n_t \, \pa_{\n_t} \eta \left(\langle\pa_{\n_t}  u, \tau_t \rangle^2 + 2 \langle\pa_{\n_t}  u, \tau_t \rangle \langle\pa_{\tau_t}  u, \n_t \rangle - \langle\pa_{\tau_t}  u, \n_t \rangle^2\right)
+2 \vartheta \, \eta \, ( \div \, \n_t \, \langle \pa_{\n_t} u, \tau_t \rangle )^2 \\
&+ \vartheta \, \eta \, |\nabla^2 u|^2
\end{align*}
in $\spt \eta \, \cap  \Omega_t$. Integrating the previous identity over $\Omega_t$ and
invoking \eqref{div(divn*f*n)} and \eqref{eq:surfacetension2} yields
\begin{align}
\label{badterm2}
\begin{split}
2 \vartheta &\int_{\Omega_t}  \eta \, \langle \nabla \div_{\n_t} u, \Delta  u\rangle \d x \\
=&-\vartheta\int_{\Omega_t} \eta \, \left(\langle \Delta u, \tau_t \rangle^2+\langle \pa_{\n_t}^2 u-\pa_{\tau_t}^2 u - \div \, \n_t \, \pa_{\n_t} u, \n_t \rangle^2\right) \d x \\
&-2\vartheta\int_{\Omega_t} \eta \, \left( \langle \pa_{\tau_t}^2 u,\tau_t\rangle^2 + \langle \pa_{\n_t}\pa_{\tau_t} u, \n_t\rangle^2\right) \d x
+\vartheta\int_{\Omega_t} \eta \, |\nabla^2 u|^2 \d x \\
&+2\vartheta\int_{\pa \Omega_t} \curv_t \langle \pa_{\n_t} u, \tau_t \rangle^2 \d \H^1
+2\vartheta\int_{\Omega_t} \eta \, ( \div \, \n_t \, \langle \pa_{\n_t} u, \tau_t \rangle )^2 \d x\\
&+\vartheta\int_{\Omega_t} \div \, \n_t \, \pa_{\n_t} \eta \left(\langle\pa_{\tau_t}  u, \n_t \rangle^2 - 2 \langle\pa_{\n_t}  u, \tau_t \rangle \langle\pa_{\tau_t}  u, \n_t \rangle -
\langle\pa_{\n_t}  u, \tau_t \rangle^2\right) \d x.
\end{split}
\end{align}
Applying  \eqref{div(divn*f*n)} on the boundary integral in \eqref{badterm2} gives
\[
2 \vartheta \int_{\pa \Omega_t} \curv_t \langle \pa_{\n_t} u, \tau_t \rangle^2 \d \H^1 =  2 \vartheta  \int_{\Omega_t}2 \, \div \, \n_t \, \eta \,  \langle \pa_{\n_t} u, \tau_t \rangle \langle  \pa_{\n_t}^2 u, \tau_t \rangle
+ \div \, \n_t \pa_{\n_t} \eta \,  \langle \pa_{\n_t} u, \tau_t \rangle^2  \d x.
\]
By substituting the previous into \eqref{badterm2} and, in turn, injecting \eqref{badterm2} into \eqref{badterm1} 
we get 
\begin{align}
\label{badterm3}
\begin{split}
2 \vartheta &\int_{\pa \Omega_t}  \div_{\n_t} u \, \langle \Delta  u , \n_t \rangle \d \H^1  \\
=&-\vartheta\int_{\Omega_t} \eta \, \left(\langle \Delta u, \tau_t \rangle^2+\langle \pa_{\n_t}^2 u-\pa_{\tau_t}^2 u - \div \, \n_t \, \pa_{\n_t} u, \n_t \rangle^2\right) \d x \\
&-2\vartheta\int_{\Omega_t} \eta \, \left( \langle \pa_{\tau_t}^2 u,\tau_t\rangle^2 + \langle \pa_{\n_t}\pa_{\tau_t} u, \n_t\rangle^2\right) \d x
+\vartheta\int_{\Omega_t} \eta \, |\nabla^2 u|^2 \d x \\
&+4\vartheta\int_{\Omega_t} \div \, \n_t \, \eta \,  \langle \pa_{\n_t} u, \tau_t \rangle \langle  \pa_{\n_t}^2 u, \tau_t \rangle \d x
+2\vartheta \int_{\Omega_t}  \div_{\n_t} u \, \pa_{\n_t} \eta \, \langle \n_t , \Delta  u\rangle \d x
\\
&+2\vartheta\int_{\Omega_t} \eta \, ( \div \, \n_t \, \langle \pa_{\n_t} u, \tau_t \rangle )^2 \d x+\vartheta\int_{\Omega_t} \div \, \n_t \, \pa_{\n_t} \eta \left(\langle\pa_{\tau_t}  u, \n_t \rangle -
\langle\pa_{\n_t}  u, \tau_t \rangle\right)^2 \d x.
\end{split}
\end{align}
Again, applying  \eqref{div(divn*f*n)} on the third integral on RHS of \eqref{1stvar_willmoreD} yields
\[
- 4 \vartheta \int_{\pa \Omega_t} \curv_t (\div_{\tau_t} u)^2 \d \H^1= - 4 \vartheta \int_{\Omega_t} 2 \, \div \, \n_t \, \eta \, \langle \pa_{\n_t} u, \n_t \rangle
\langle \pa_{\n_t}^2 u, \n_t \rangle + \div \, \n_t \, \pa_{\n_t} \eta \, \langle \pa_{\n_t} u, \n_t \rangle^2 \d x.
\]
Thus we infer from the above, \eqref{1stvar_willmoreD} and \eqref{badterm3} that
\eqref{1stvar_willmoreE} holds.
\end{proof}

\begin{thm}
Let \solT be a $C^2$-regular solution. Then there exists independent constants $c$ and $C$ such that
\begin{equation}
\label{1stvar_willmore:estD}
\frac\d{\d t} \int_{\pa \Omega_t}  \curv_t^2 \d \H^1
\leq-\vartheta c \int_{\Omega_t}|\nabla^2 u|^2 \d x + \vartheta \frac{C}{r_t^2}\int_{\Omega_t} |\nabla u|^2 \d x +  \frac1{\vartheta} \frac{C}{r_t^3} \P_t.
\end{equation}
\end{thm}
\begin{proof}
First by using \eqref{|nabla^2X|^2}, \eqref{diffeta1} and \eqref{div n est} we infer from the previous decomposition
\begin{align}
\label{1stvar_willmore:estA}
\begin{split}
\frac\d{\d t} &\int_{\pa \Omega_t}  \curv_t^2 \d \H^1 \\
\leq&-\vartheta\int_{\Omega_t} \eta \, \left(\langle \Delta u, \tau_t \rangle^2+\langle \pa_{\n_t}^2 u-\pa_{\tau_t}^2 u - \div \, \n_t \, \pa_{\n_t} u, \n_t \rangle^2\right) \d x \\
&-2\vartheta\int_{\Omega_t} \eta \, \left( \langle \pa_{\tau_t}^2 u,\tau_t\rangle^2 + \langle \pa_{\n_t}\pa_{\tau_t} u, \n_t\rangle^2\right) \d x
-\vartheta\int_{\Omega_t} (1-\eta) \, |\nabla^2 u|^2 \d x \\
&+ \vartheta C \int_{\Omega_t} \frac{|\nabla u|}{r_t} |\nabla^2 u| \d x + \vartheta \frac{C}{r_t^2}\int_{\Omega_t} |\nabla u|^2 \d x-\int_{\pa \Omega_t} \curv_t^2 \, \div_{\tau_t} u   \d \H^1.
\end{split}
\end{align}
Recall that $C$ and $c$ are positive independent constants which may change from line to line.
The second order homogenous part controls the full second derivative of $u$. To this end, we use the incompressibility of the system 
as well as the decomposition $\eqref{|nabla^2X|^2}$ of $|\nabla^2u|^2$ wrt.\ the frame $\{\n_t, \tau_t\}$
to obtain 
\begin{align}
\notag
|\nablasym \pa_{\tau_t} u|^2 
&= \langle \pa_{\n_t} \pa_{\tau_t} u, \n_t \rangle^2 +\langle \pa_{\tau_t}^2 u, \tau_t \rangle^2 + \frac12\langle \pa_{\n_t}^2 u - \pa_{\tau_t}^2 u, \n_t \rangle^2 \ \ \ \text{and} \\
\label{|nabla^2|^2_control}
|\nabla^2 u|^2 
&\leq C\left( \langle\Delta u, \tau_t \rangle^2+ |\nabla \pa_{\tau_t} u|^2 +(\div \, \n_t)^2|\nabla u|^2 \right)
\end{align}
in $\spt \eta \, \cap \Omega_t$. Then we multiply the identity by $\eta$, integrate it over $\Omega_t$,
use the incompressibility and apply Korn's inequality \eqref{Korn_cutoff:est} on the integrand $2 \eta \,  |\nablasym \pa_{\tau_t} u|^2$ to obtain
\begin{align*}
&\phantom{{}={}}\int_{\Omega_t} \eta \, \left(\langle \Delta u, \tau_t \rangle^2+\langle \pa_{\n_t}^2 u-\pa_{\tau_t}^2 u - \div \, \n_t \, \pa_{\n_t} u, \n_t \rangle^2\right) \d x \\
&+2\int_{\Omega_t} \eta \, \left( \langle \pa_{\tau_t}^2 u,\tau_t\rangle^2 + \langle \pa_{\n_t}\pa_{\tau_t} u, \n_t\rangle^2\right) \d x \\
&\geq  c \int_{\Omega_t} \eta \, \left(\langle \Delta u, \tau_t \rangle^2 +  |\nabla \pa_{\tau_t} u|^2 \right) \d x
+\int_{\Omega_t}  \div \, \n_t \, \eta \, \langle \pa_{\n_t} u,\n_t\rangle (\langle \pa_{\n_t}\pa_{\tau_t} u, \tau_t\rangle +  \langle \pa_{\tau_t}^2 u, \n_t \rangle) \d x \\
&-\frac{C}{r_t^2}\int_{\Omega_t} |\nabla u|^2 \d x.
\end{align*}
Using \eqref{|nabla^2|^2_control}, Young's inequality as well as \eqref{div n est} further yields
\begin{align}
\label{badterm_est}
\begin{split}
&\phantom{{}={}}\int_{\Omega_t} \eta \, \left(\langle \Delta u, \tau_t \rangle^2+\langle \pa_{\n_t}^2 u-\pa_{\tau_t}^2 u - \div \, \n_t \, \pa_{\n_t} u, \n_t \rangle^2\right) \d x \\
&+2\int_{\Omega_t} \eta \, \left( \langle \pa_{\tau_t}^2 u,\tau_t\rangle^2 + \langle \pa_{\n_t}\pa_{\tau_t} u, \n_t\rangle^2\right) \d x \\
&\geq  c \int_{\Omega_t} \eta \, |\nabla^2 u|^2 \d x-\frac{C}{r_t^2}\int_{\Omega_t} |\nabla u|^2 \d x.
\end{split}
\end{align}
Injecting this into \eqref{1stvar_willmore:estA} and applying Young's inequality on the mixed term results in 
\begin{align}
\label{1stvar_willmore:estC}
\begin{split}
\frac\d{\d t} &\int_{\pa \Omega_t}  \curv_t^2 \d \H^1 \\
\leq&
-\vartheta c \int_{\Omega_t}|\nabla^2 u|^2 \d x + \vartheta \frac{C}{r_t^2}\int_{\Omega_t} |\nabla u|^2 \d x - \int_{\pa \Omega_t} \curv_t^2 \, \div_{\tau_t} u   \d \H^1.
\end{split}
\end{align}
The last term on RHS is an original term from \eqref{1stvar_willmoreA}. Note that it doesn't have $\vartheta$ as a scaling factor. 
To transform it into a bulk integral we first apply \eqref{div(divn*f*n)} on $\eta \, \div \, \n_t \, \div_{\tau_t} u$
and then recall the basic identities \eqref{basicid} and \eqref{pa_n H} for the frame $\{\n_t, \tau_t \}$ to obtain 
\begin{align*}
\int_{\pa \Omega_t} &\curv_t^2 \, \div_{\tau_t} u   \d \H^1 \\
&= \int_{\Omega_t} (\div \, \n_t)^2\left( \pa_{\n_t} \eta \, \langle \pa_{\tau_t} u, \tau_t\rangle - \eta \,  \div \, \n_t \, \langle \pa_{\tau_t} u, \tau_t\rangle + \eta \, \langle \pa_{\n_t} \pa_{\tau_t} u, \tau_t\rangle \right)   \d x.
\end{align*}
Hence, it follows from \eqref{|nabla^2X|^2}, \eqref{diffeta1} and \eqref{div n est} 
\[
\left|\int_{\pa \Omega_t} \curv_t^2 \, \div_{\tau_t} u   \d \H^1 \right|
\leq \int_{\spt \eta \, \cap \, \Omega_t} \frac{1}{r_t^2} \frac{|\nabla u|}{r_t} + \eta\frac{1}{r_t^2} |\nabla^2 u| \d x.
\]
Finally, by using a Young's inequality to absorb mixed terms at the cost of a potentially higher constant and noting that $|\spt \eta \, \cap \, \Omega_t| \leq r_t \P_t$, we infer from \eqref{1stvar_willmore:estC} and the previous estimate that \eqref{1stvar_willmore:estD} holds. \end{proof}

As a byproduct of this discussion we observe that with increased regularity we may control the $L^2$-norm of $\nabla^2 u$
as follows.
\begin{lemma}
\label{lem:H2control}
 Let \solT  be a $C^3$-regular solution. Then it holds
\begin{equation}
\label{est:H2control}
\vartheta^2\int_{\Omega_t} |\nabla^2 u|^2 \d x \leq r_t \int_{\pa \Omega_t} |\pa_{\tau_t} \curv_t|^2 \d \H^1 + \vartheta^2\frac{C}{r_t^2} \int_{\Omega_t} |\nabla u|^2 \d x.
\end{equation}
\end{lemma}
\begin{proof}
We first compute
\begin{align*}
\int_{\Omega_t} |\Delta u|^2 \d x 
&= \frac1\vartheta\int_{\pa \Omega_t} p \langle \Delta u, \n_t\rangle \d \H^1 \\
&=\frac1\vartheta\int_{\pa \Omega_t} \curv_t \langle \Delta u, \n_t\rangle \d \H^1 + 2 \int_{\pa \Omega_t} \div_{\n_t} u \langle \Delta u, \n_t\rangle \d \H^1 \\
&=\frac2\vartheta\int_{\pa \Omega_t} \pa_{\tau_t} \curv_t \langle \pa_{\n_t} u, \tau_t\rangle \d \H^1 + 2 \int_{\pa \Omega_t} \div_{\n_t} u \langle \Delta u, \n_t\rangle \d \H^1.
\end{align*}
Here the first equality follows after using Stokes equation in \eqref{eq:main}, 
$\div \Delta u = 0$ and divergence theorem. The second equality is
due to the boundary condition
\eqref{eq:surfacetension1} and for the 
last one we used \eqref{<Delta u,n>}, \eqref{basicid}, integration by parts as well as the boundary condition \eqref{eq:surfacetension2}. 
Recalling Reilly formula \eqref{improved2DReilly} for $u$ we then obtain
a decomposition
\begin{align}
\label{int:nabla^2u_decomp}
\begin{split}
\int_{\Omega_t} |\nabla^2 u|^2 \d x 
= &\frac2\vartheta\int_{\pa \Omega_t} \pa_{\tau_t} \curv_t \langle \pa_{\n_t} u, \tau_t\rangle \d \H^1 + 2 \int_{\pa \Omega_t} \div_{\n_t} u \langle \Delta u, \n_t\rangle \d \H^1 \\
 &-4\int_{\pa \Omega_t} \curv_t (\div_{\tau_t} u)^2 \d \H^1.
\end{split}
\end{align}
We gather from \eqref{badterm3}, \eqref{badterm_est}, \eqref{Delta X}, \eqref{|nabla^2X|^2}, \eqref{diffeta2}
and \eqref{div n est} 
\begin{align*}
2& \int_{\pa \Omega_t} \div_{\n_t} u \langle \Delta u, \n_t\rangle \d \H^1 \\
&\leq (1-c)\int_{\Omega_t} |\nabla^2 u|^2 \d x+ \frac{C}{r_t} 
\int_{\Omega_t} |\nabla^2 u||\nabla u| \d x + \frac{C}{r_t^2} 
\int_{\Omega_t} |\nabla u|^2 \d x.
\end{align*}
Injecting this estimate into \eqref{int:nabla^2u_decomp}
and applying Young's inequality on the first integral results in
\begin{align*}
\int_{\Omega_t} |\nabla^2 u|^2 \d x
 &\leq \frac{r_t}{2\vartheta^2} \int_{\pa \Omega_t} |\pa_{\tau_t} \curv_t|^2 \d \H^1
+\frac{C}{r_t} \int_{\pa \Omega_t} \langle \pa_{\n_t} u, \tau_t \rangle^2 \d \H^1
-4\int_{\pa \Omega_t} \curv_t (\div_{\tau_t} u)^2 \d \H^1 \\
&+\int_{\Omega_t} |\nabla^2 u||\nabla u| \d x + \frac{C}{r_t^2} 
\int_{\Omega_t} |\nabla u|^2 \d x.
\end{align*}
Thus we infer \eqref{est:H2control} from this estimate by applying \eqref{div(f*n)} on the second integral, 
\eqref{div(divn*f*n)} on the third one, using the estimates  \eqref{diffeta2}
and \eqref{div n est} and finally by applying Young's inequality on the products $|\nabla^2 u|(
|\nabla u|/r_t)$.
\end{proof}


\section{\texorpdfstring{$H^1$}{H1}-Curvature estimates}

We now proceed in a similar fashion with the derivative of the curvature
for $C^3$-regular solutions. Note the analogy between the expressions \eqref{1stvar_willmoreB} and \eqref{1stvar_higherwillmoreA}.
\begin{lemma} Let \solT be a $C^3$-regular solution. Then it satisfies
\begin{align}
\label{1stvar_higherwillmoreA}
\begin{split}
\frac{\d}{\d t} &\int_{\pa \Omega_t} |\partial_{\tau_t} \curv_t|^2 \d \H^1 \\
=&-\int_{\pa \Omega_t}  \pa_{\tau_t} p \langle \pa_{\tau_t} \Delta  u, \n_t \rangle \d \H^1 + 2 \vartheta \int_{\pa \Omega_t}
  \pa_{\tau_t}\div_{\n_t} u \langle \pa_{\tau_t} \Delta  u, \n_t \rangle \d \H^1\\
&-\int_{\pa \Omega_t} \left( \frac{\curv_t}{\vartheta}+3 \, \div_{\tau_t} u \right)(\pa_{\tau_t} \curv_t)^2  \d \H^1.
\end{split}
\end{align}
\end{lemma}

\begin{proof}
Again, for the change of  the norm $\|\partial_{\tau_t} \curv_t\|_{L^2(\pa \Omega_t)}^2$ we have \eqref{eq:dk2derivative} and again using that $\partial \Omega_t$ evolves with the normal part of $u$ to obtain
\begin{align*}
\frac{\d}{\d t} &\int_{\pa \Omega_t} |\partial_{\tau_t} \curv_t|^2 \d \H^1 \\
=&-2\int_{\pa \Omega_t}  \pa_{\tau_t} \curv_t \pa_{\tau_t}\left( \langle \pa_{\tau_t}^2 u, \n_t \rangle +  \curv_t \div_{\tau_t} u\right) +
 \curv_t\pa_{\tau_t} \curv_t \pa_{\tau_t} \div_{\tau_t} u \d \H^1 \\
&-3 \int_{\pa \Omega_t}\div_{\tau_t} u \, (\pa_{\tau_t} \curv_t)^2 \d \H^1.
\end{align*}
Injecting \eqref{<Delta u,n>} into the previous expression yields
\begin{align*}
\frac{\d}{\d t} &\int_{\pa \Omega_t} |\partial_{\tau_t} \curv_t|^2 \d \H^1 \\
=&-\int_{\pa \Omega_t}  \pa_{\tau_t} \curv_t \langle \pa_{\tau_t} \Delta  u, \n_t \rangle +
  \curv_t\pa_{\tau_t} \curv_t  ( \langle \Delta u , \tau_t \rangle - 2 \pa_{\tau_t} \div_{\n_t} u) \d \H^1 \\
&-3 \int_{\pa \Omega_t} \div_{\tau_t} u (\pa_{\tau_t} \curv_t)^2 \d \H^1.
\end{align*}
Differentiating \eqref{eq:surfacetension1} in the tangential direction and using Stokes equation in \eqref{eq:main} gives us 
\begin{equation}
\label{pa_tau H}
\pa_{\tau_t} \curv_t = \partial_\tau p - 2\vartheta \pa_{\tau_t} \div_{\n_t} u= \vartheta(\langle \Delta u, \tau_t\rangle - 2 \pa_{\tau_t} \div_{\n_t} u)
\end{equation}
and thus the previous decomposition for $\frac\d{\d t}\|\partial_{\tau_t} \curv_t\|_{L^2(\pa \Omega_t)}^2$ becomes \eqref{1stvar_higherwillmoreA}.
\end{proof}

Next we proceed in a similar fashion as before in order to turn boundary terms into bulk terms.
We will show that $\frac\d{\d t}\|\partial_{\tau_t} \curv_t\|_{L^2(\pa \Omega_t)}^2$ can be expressed in the in the form of a sum of squares of third order derivatives of $u$ multiplied by the cutoff $\eta$ plus a remainder term $R$.

Compared to the lower order case the highest order term will not be a real bulk term due to the presence of the cut-off.
The remainder $R$ will be further split into different terms $R_i$ which we will collect in the process.
These consist of both bulk and boundary integrals and the highest order bulk integrands are mixed terms of the form a third derivative times a lower order term.

\begin{lemma}  Let \solT be a $C^3$-regular solution. We have
\begin{align}
\label{1stvar_higherwillmoreB}
\begin{split}
\frac{\d}{\d t}&\int_{\pa \Omega_t}|\partial_{\tau_t} \curv_t|^2  \d \H^1 \\
= &- \vartheta \int_{\Omega_t}  \eta \,
\langle \pa_{\tau_t}( - \pa_{\n_t}^2 u+ \pa_{\tau_t}^2 u + \div \, \n_t \, \pa_{\n_t} u), \n_t\rangle^2
\d x \\
&- \vartheta \int_{\Omega_t} \eta \, \left(\langle \pa_{\tau_t} \Delta u, \tau_t\rangle^2
+4\langle \pa_{\tau_t}^3 u, \tau_t \rangle^2
  \right) \d x +\vartheta R
\end{split}
\end{align}
with reminder term $R= R_0+R_1+R_2+R_3+R_4$, defined below.
\end{lemma}

\begin{proof}
Let us first restate \eqref{1stvar_higherwillmoreA}
\begin{align*}
\frac{\d}{\d t} &\int_{\pa \Omega_t} |\partial_{\tau_t} \curv_t|^2 \d \H^1 \\
=&-\int_{\pa \Omega_t}  \pa_{\tau_t} p \langle \pa_{\tau_t} \Delta  u, \n_t \rangle \d \H^1 + 2 \vartheta \int_{\pa \Omega_t}
  \pa_{\tau_t}\div_{\n_t} u \langle \pa_{\tau_t} \Delta  u, \n_t \rangle \d \H^1\\
&-\int_{\pa \Omega_t} \left( \frac{\curv_t}{\vartheta}+3 \, \div_{\tau_t} u \right)(\pa_{\tau_t} \curv_t)^2  \d \H^1.
\end{align*}
We get the first remainder term from the last term on RHS of  \eqref{1stvar_higherwillmoreA} after injecting \eqref{pa_tau H} into it, that is,
\begin{equation}
\label{lastpart}
-\int_{\pa \Omega_t} \left( \frac{\curv_t}{\vartheta}+3 \, \div_{\tau_t} u \right)(\pa_{\tau_t} \curv_t)^2  \d \H^1 = \vartheta R_0,
\end{equation}
where
\begin{equation}
\label{R0}
R_0 = - \int_{\pa \Omega_t} \curv_t (\langle \Delta u, \tau_t\rangle - 2 \pa_{\tau_t} \div_{\n_t} u)^2 + 3 \div_{\tau_t} u \, \pa_{\tau_t} \curv_t (\langle \Delta u, \tau_t\rangle - 2 \pa_{\tau_t} \div_{\n_t} u)
 \d \H^1.
\end{equation}

The first integral on RHS of \eqref{1stvar_higherwillmoreA} is rather straightforward to compute 
but we need some technical identities. To this end, we recall $\div  \Delta u = 0$ in $\Omega_t$ due to incompressibility of the system. Again, it follows from Stokes equation in \eqref{eq:main} that $\nabla \Delta u$ is symmetric in $\Omega_t$. These, in turn, yield
\begin{equation}
\label{|pa_tau Delta u|^2}
|\pa_{\tau_t} \Delta u|^2 = |\pa_{\n_t} \Delta u|^2 = \frac12|\nabla \Delta u|^2 \quad \text{in} \ \ \Omega_t \cap \mathcal N_{r_t}(\pa \Omega_t).
\end{equation}
By invoking $\div([A] X) = \langle \div(A^\T), X\rangle + \tr\left(A \nabla X\right)$, \eqref{basicid} and Stokes equation in \eqref{eq:main} we get the identities
\begin{align}
\label{divpataudelta}
\div ( \pa_{\tau_t}  \Delta u) &= - \div \n_t \, \pa_{\n_t}\langle  \Delta u, \tau_t\rangle \\
\label{nablapataup}
 \text{and} \quad \quad \nabla \pa_{\tau_t} p &= \vartheta \, \pa_{\tau_t} \Delta u  - \vartheta \,\div \, \n_t \,  \langle \Delta u, \n_t\rangle\tau_t
\end{align}
in $\Omega_t \cap \mathcal N_r(\Omega_t)$.

By applying the divergence theorem on $\eta \, \pa_{\tau_t} p \, \langle \pa_{\tau_t}  \Delta u, \n_t \rangle $
and then injecting Stokes equation \eqref{eq:main} and technical identities \eqref{divpataudelta} and  \eqref{nablapataup} into it we conclude
\begin{equation}
\label{nicepart}
-\int_{\pa \Omega_t}  \pa_{\tau_t} p \, \langle \pa_{\tau_t}  \Delta u, \n_t \rangle \d \H^1=
- \vartheta \int_{\Omega_t}  \eta \, | \pa_{\tau_t} \Delta u|^2 \d x + \vartheta R_1,
\end{equation}
where the remainder part can be written as
\begin{align}
\label{R1}
\begin{split}
R_1 = \int_{\Omega_t} (\pa_{\n_t} \eta  - \eta \, \div \, \n_t)\langle \Delta u, \tau_t\rangle\langle\pa_{\n_t} \Delta u, \tau_t \rangle - \div \, \n_t \, \eta \, \langle \Delta u, \n_t \rangle
\langle \pa_{\tau_t} \Delta u, \tau_t \rangle \d x.
\end{split}
\end{align}
The second integral on RHS of \eqref{1stvar_higherwillmoreA}
is more technical to treat and consists of several consequent steps. 
First, by substituting  \eqref{<Delta u,n>} into it, 
integrating by parts and using the incompressibility as well as the basic identities  \eqref{basicid}
we obtain
\begin{equation}
\label{badpart1}
2\vartheta \int_{\pa \Omega_t} \pa_{\tau_t} \div_{\n_t} u \, \langle \pa_{\tau_t} \Delta u,\n_t\rangle \d \H^1 = -4\vartheta \int_{\pa \Omega_t}
 \langle\pa_{\tau_t}^2 u, \tau_t \rangle \pa_{\tau_t}\langle \pa_{\tau_t}^2 u, \n_t \rangle \d \H^1+ \vartheta R_2,
\end{equation}
where
\begin{align}
\label{R2}
\begin{split}
R_2 = &-2\int_{\pa \Omega_t} 2 \pa_{\tau_t} \curv_t  \langle \pa_{\tau_t} u, \tau_t \rangle\langle \pa_{\tau_t}^2 u, \tau_t \rangle
+\pa_{\tau_t}\curv_t  \langle\pa_{\tau_t} u, \n_t \rangle  \langle \Delta u,\n_t\rangle
\d \H^1   \\
&-2\int_{\pa \Omega_t} 2\curv_t \langle \pa_{\tau_t}^2 u, \tau_t \rangle^2
+\curv_t  \langle\pa_{\tau_t}^2 u, \n_t \rangle  \langle \Delta u,\n_t\rangle
-\curv_t\langle\pa_{\tau_t}^2 u, \tau_t \rangle  \langle \Delta u,\tau_t\rangle 
\d \H^1\\
&-2\int_{\pa \Omega_t}
\curv_t^2 \left( \langle\pa_{\tau_t} u, \tau_t \rangle  \langle \Delta u,\n_t\rangle
+ \langle\pa_{\tau_t} u, \n_t \rangle \langle \Delta u,\tau_t\rangle 
-2 \langle \pa_{\tau_t} u, \n_t \rangle 
\langle \pa_{\tau_t}^2 u, \tau_t \rangle\right)
\d \H^1.
\end{split}
\end{align}
Before we turn the first term on RHS of \eqref{badpart1} into a bulk integral 
we need (yet) another batch of technical identities this time regarding the change of differentiation order componentwise in the coordinate frame $\{\n_t, \tau_t\}$
for 3rd order derivative of $u$.
We have the following rules 
\begin{align}
\label{3rdchange1}
\langle\pa_{\n_t} \pa_{\tau_t}^2 u , \n_t \rangle&=-\langle \pa_{\tau_t}^3 u, \tau_t\rangle -\pa_{\tau_t} \div \, \n_t \, \langle\pa_{\n_t} u, \tau_t\rangle +  2 \, \div \, \n_t \, \langle  
\pa_{\n_t}^2 u + \div \, \n_t \, \pa_{\n_t} u , \n_t \rangle\\
\label{3rdchange2}
 \langle \pa_{\tau_t}  \pa_{\n_t}^2 u, \n_t\rangle&=-\langle \pa_{\n_t} \pa_{\tau_t}^2 u, \tau_t\rangle
- \div \, \n_t \,  \langle \Delta u , \tau_t\rangle
-\div \, \n_t \,  \langle\pa_{\tau_t}^2 u, \tau_t\rangle  \\
\label{3rdchange3}
\langle\pa_{\tau_t} \pa_{\n_t}^2 u , \tau_t \rangle&=
-\langle \pa_{\n_t}^3 u, \n_t\rangle -2 \, \div \, \n_t \, \langle \pa_{\n_t}\pa_{\tau_t} u, \n_t \rangle
\\
\label{3rdchange4}
\langle \pa_{\n_t}^2 \pa_{\tau_t} u, \n_t \rangle 
&=-\langle\pa_{\n_t}  \pa_{\tau_t}^2 u, \tau_t \rangle
-\div \, \n_t \, \langle \pa_{\n_t}^2 u, \tau_t\rangle
+(\div \, \n_t)^2 \langle \pa_{\n_t} u, \tau_t\rangle\\
\label{3rdchange5}
\langle \pa_{\n_t}^2 \pa_{\tau_t} u, \tau_t \rangle&=-\langle \pa_{\n_t}^3 u, \n_t \rangle 
\end{align}
which can be obtained from the incompressibility condition $\div u = 0$ with help of the decomposition \eqref{Delta X} and the basic identities \eqref{basicid},
\eqref{pa_n H} and \eqref{changeoforder} for the frame. 

We first apply \eqref{div(f*n)} on $\eta \, \langle\pa_{\tau_t}^2 u, \tau_t \rangle \pa_{\tau_t}\langle \pa_{\tau_t}^2 u, \n_t \rangle$ and then 
use integration by parts, the previous identities \eqref{3rdchange1}  and \eqref{3rdchange2} ,  \eqref{diffeta1} for $\eta$  and the basic identities \eqref{basicid} and \eqref{changeoforder} 
to obtain
\begin{align}
\label{badpart2}
\begin{split}
&-4\vartheta \int_{\pa \Omega_t}
 \langle\pa_{\tau_t}^2 u, \tau_t \rangle \pa_{\tau_t}\langle \pa_{\tau_t}^2 u, \n_t \rangle \d \H^1 \\
&=4\int_{\Omega_t}
\eta \,  \langle \pa_{\tau_t} \pa_{\n_t}^2 u, \n_t\rangle \langle \pa_{\tau_t}^3 u, \n_t \rangle
-\eta \,  \langle \pa_{\tau_t}^3 u, \tau_t\rangle^2 
\d x +\vartheta R_3,
\end{split}
\end{align}
where
\begin{align}
\label{R3}
\begin{split}
&R_3=\\
4&\int_{\Omega_t}
\div \, \n_t \,\eta \,  
\left(
\langle \pa_{\tau_t} \pa_{\n_t}^2 u, \n_t\rangle \langle \pa_{\tau_t}^2 u, \tau_t \rangle
+ \left( \langle \Delta u , \tau_t\rangle + \langle\pa_{\tau_t}^2 u, \tau_t\rangle\right)\pa_{\tau_t}\langle \pa_{\tau_t}^2 u, \n_t \rangle
\right)
 \d x \\
-4&\int_{\Omega_t}\eta \, \pa_{\tau_t} \div \, \n_t \, \langle\pa_{\n_t} u, \tau_t\rangle \left(\langle\pa_{\tau_t}^3 u, \tau_t \rangle
-\div \, \n_t \, \langle\pa_{\tau_t}^2 u, \n_t \rangle \right)
\d x \\
+4&\int_{\Omega_t} \div \, \n_t \, \eta \, \left( \langle\pa_{\tau_t}^2 u, \n_t \rangle
+2\langle\pa_{\n_t}^2 u+ \div \, \n_t \, \pa_{\n_t} u, \n_t \rangle
\right)\langle \pa_{\n_t}^3 u, \tau_t\rangle
\d x \\
+4&\int_{\Omega_t}\pa_{\n_t} \eta \,  \langle \pa_{\tau_t}^2 u, \n_t \rangle
\langle \pa_{\n_t}^3 u, \tau_t\rangle \d x \\
-4&\int_{\Omega_t}2(\div \, \n_t )^2 \eta \, \langle  \pa_{\n_t}^2 u + \div \, \n_t \, \pa_{\n_t} u , \n_t \rangle
 \langle\pa_{\tau_t}^2 u, \n_t \rangle
+\div \, \n_t \, \pa_{\n_t} \eta \, \langle\pa_{\tau_t}^2 u, \n_t \rangle^2
 \d x.
\end{split}
\end{align}
But now recalling the decomposition \eqref{Delta X} for the Laplacian we observe
\begin{align}
\label{niceobservation}
\begin{split}
-&\int_{ \Omega_t} \eta \, \langle \pa_{\tau_t} \Delta u, \n_t \rangle^2 \d x + 4 \int_{ \Omega_t}
 \eta \, \langle \pa_{\tau_t} \pa_{\n_t}^2 u, \n_t\rangle \langle \pa_{\tau_t}^3 u, \n_t \rangle \d x \\
=-&\int_{\Omega_t} \eta \, \left\langle \pa_{\tau_t} \left(- \pa_{\n_t}^2 u
+\pa_{\tau_t}^2u + \div \, \n_t \, \pa_{\n_t} u\right), \n_t \right\rangle^2 \d x\\
+ 4&\int_{\Omega_t} \eta \, \langle \pa_{\tau_t} \pa_{\n_t}^2 u,\n_t \rangle \langle\pa_{\tau_t} (\div \, \n_t \, \pa_{\n_t} u), \n_t \rangle \d x.
\end{split}
\end{align}
Thus by setting
\begin{equation}
\label{R4}
R_4 = 4\int_{\Omega_t} \eta \, \langle \pa_{\tau_t} \pa_{\n_t}^2 u,\n_t \rangle \langle\pa_{\tau_t} (\div \, \n_t \, \pa_{\n_t} u), \n_t \rangle \d x
\end{equation}
we infer the lemma from \eqref{1stvar_higherwillmoreA}, \eqref{lastpart}, \eqref{nicepart}, \eqref{badpart1}, \eqref{badpart2} and \eqref{niceobservation}.
\end{proof}

We now need to estimate all those error terms.

\begin{thm}  Let \solT be a $C^3$-regular solution. There exist independent constants $C$ and $c$ such that
\begin{align}
\label{1stvar_higherwillmore:estB}
\begin{split}
\frac\d{\d t} &\int_{\pa \Omega_t}  |\partial_{\tau_t} \curv_t|^2 \d \H^1  \\
\leq&
-\vartheta c \int_{\Omega_t} |\nabla^3 u|^2 \d x +  \vartheta C r^2_t\int_{\Omega_t} |\nabla u|^2 \d x  \left(\int_{\pa \Omega_t} |\partial_{\tau_t} \curv_t|^2   \d \H^1 \right)^2\\
&+ \vartheta \frac{C}{r_t} \int_{\Omega_t} |\nabla u|^2 \d x \int_{\pa \Omega_t} |\partial_{\tau_t} \curv_t|^2   \d \H^1 \\
&+ \vartheta \frac{C}{r_t^2}\int_{\Omega_t} |\nabla^2 u|^2 \d x +  \vartheta \frac{C}{r_t^4}\int_{\Omega_t} |\nabla u|^2 \d x
 +\vartheta \frac{C}{r_t^6}\int_{\Omega_t} |u|^2 \d x.
\end{split}
\end{align}
\end{thm}

\begin{proof}
\textbf{Step 1: Estimating the leading part:}
We show that for the  leading part
\begin{align}
 \label{Leading_part}
\begin{split}
L:=- &\vartheta \int_{\Omega_t}  \eta \,
\langle \pa_{\tau_t}( - \pa_{\n_t}^2 u+ \pa_{\tau_t}^2 u + \div \, \n_t \, \pa_{\n_t} u), \n_t\rangle^2
\d x \\
- &\vartheta \int_{\Omega_t} \eta \, \left(\langle \pa_{\tau_t} \Delta u, \tau_t\rangle^2
+4\langle \pa_{\tau_t}^3 u, \tau_t \rangle^2
  \right) \d x
\end{split}
\end{align}
it holds
\begin{align}
 \label{Leading_est}
\begin{split}
L \leq -&\vartheta c \int_{\Omega_t} |\nabla^3 u|^2 \d  x
+\vartheta C  r_t \|\eta^\frac12 \pa_{\n_t} u\|_{L^\infty(\Omega_t)}^2 \int_{\pa \Omega_t}|\partial_{\tau_t} \curv_t|^2 \d \H^1
 \\
+  &\vartheta \frac{C}{r_t^2} \int_{\Omega_t} |\nabla^2 u|^2 \d x + \vartheta \frac{C}{r_t^4} \int_{\Omega_t} |\nabla u|^2 \d x
 +\vartheta \frac{C}{r_t^6} \int_{\Omega_t} |u|^2 \d x
\end{split}
\end{align}
We can decompose
\begin{align*}
|\nablasym \pa_{\tau_t}^2 u|^2  = \langle \pa_{\n_t}\pa_{\tau_t}^2 u, \n_t \rangle^2 +\langle \pa_{\tau_t}^3 u, \tau_t \rangle^2 + \frac12 \left( \langle \pa_{\n_t}\pa_{\tau_t}^2 u, \tau_t \rangle +  \langle \pa_{\tau_t}^3 u, \n_t \rangle \right)^2
\end{align*}
with respect to the frame $\{\n_t,\tau_t\}$ in $\spt \eta \, \cap  \Omega_t$.
Then it follows from the decomposition \eqref{|nabla^2X|^2} for $|\nabla^2 u|^2$ and the technical identities \eqref{3rdchange1} and \eqref{3rdchange2} that
\begin{align*}
|\nablasym \pa_{\tau_t}^2 u|^2
&\leq C\left( \langle \pa_{\tau_t}( - \pa_{\n_t}^2 u+ \pa_{\tau_t}^2 u + \div \, \n_t \, \pa_{\n_t} u), \n_t\rangle^2 + \langle \pa_{\tau_t}^3 u, \tau_t \rangle^2 \right) \\
&+C\left(|\pa_{\tau_t} \div \, \n_t \pa_{\n_t} u|^2 + (\div \, \n_t)^2 |\nabla^2 u|^2+(\div \, \n_t)^4|\nabla u|^2\right)
\end{align*}
in $\spt \eta \, \cap  \Omega_t$. Thus multiplying the previous estimate by $\eta$, using the estimate \eqref{div n est},
integrating over $\Omega_t$ and comparing against $L$ yields
\begin{align*}
L \leq -&\vartheta  \int_{\Omega_t}\eta \, \left(c |\nablasym \pa_{\tau_t}^2 u|^2 + \langle \pa_{\tau_t} \Delta u, \tau_t\rangle^2 \right) \d  x
+\vartheta C  \int_{\Omega_t} \eta \, |\pa_{\tau_t} \div \, \n_t \, \pa_{\n_t} u|^2 \d x
 \\
+  &\vartheta \frac{C}{r_t^2} \int_{\Omega_t} |\nabla^2 u|^2 \d x + \vartheta \frac{C}{r_t^4} \int_{\Omega_t} |\nabla u|^2 \d x.
\end{align*}
Recalling Korn's inequality \ref{Korn_cutoff:est} we further have
an estimate
\begin{align}
 \label{Leading1}
\begin{split}
L \leq -&\vartheta \int_{\Omega_t}\eta \, \left( c |\nabla \pa_{\tau_t}^2 u|^2 + \langle \pa_{\tau_t} \Delta u, \tau_t\rangle^2 \right) \d  x
+\vartheta C  \int_{\Omega_t} \eta \, |\pa_{\tau_t} \div \, \n_t \, \pa_{\n_t} u|^2 \d x
 \\
+  &\vartheta \frac{C}{r_t^2} \int_{\Omega_t} |\nabla^2 u|^2 \d x + \vartheta \frac{C}{r_t^4} \int_{\Omega_t} |\nabla u|^2 \d x.
\end{split}
\end{align}
Here we also used \eqref{|nabla^2X|^2}.
Starting from the decomposition \eqref{|nabla^3X|^2} for $|\nabla^3u|^2$ with respect to the frame $\{\n_t,\tau_t\}$
and using \eqref{|nabla^2X|^2} we obtain
\begin{align}
 \label{Leading2}
\begin{split}
|\nabla^3 u|^2 \leq &C\left(|\pa^3_{\n_t} u|^2+|\pa_{\n_t}^2 \pa_{\tau_t} u|^2+|\pa_{\tau_t}\pa_{\n_t}^2 u|^2+ |\nabla \pa_{\tau_t}^2u|^2+ |\pa_{\tau_t} \div \, \n_t \, \pa_{\n_t}u|^2 \right)\\
+ &C (\div \, \n_t)^2 |\nabla^2 u|^2 + C (\div \, \n_t)^4  |\nabla u|^2
\end{split}
\end{align}
in $\spt \eta \, \cap  \Omega_t$. Further, recalling \eqref{3rdchange2}, \eqref{3rdchange3}, \eqref{3rdchange4} and
\eqref{3rdchange5}
yields
\begin{align}
 \label{Leading3}
\begin{split}
|\nabla^3 u|^2 \leq &C\left(|\pa^3_{\n_t} u|^2 + |\nabla \pa_{\tau_t}^2u|^2 + |\pa_{\tau_t} \div \, \n_t \, \pa_{\n_t}u|^2 \right)\\
+ &C (\div \, \n_t)^2 |\nabla^2 u|^2 + C (\div \, \n_t)^4  |\nabla u|^2
\end{split}
\end{align}
in $\spt \eta \, \cap  \Omega_t$. To get rid of the term $|\pa^3_{\n_t} u|^2$
we use $\div \, \Delta u = 0$ and $(\nabla \Delta u)^\T=\nabla \Delta u$  in $\Omega_t$, the basic identities \eqref{basicid}, \eqref{pa_n H} and \eqref{changeoforder}
for the frame $\{\n_t,\tau_t\}$ , the decomposition
\eqref{Delta X} for $\Delta u$ as well as \eqref{3rdchange1} and \eqref{3rdchange2} to conclude
\begin{align}
\label{<pa_n^3 u,n>}
\begin{split}
\langle \pa_{\n_t}^3 u, \n_t\rangle
=-&\langle \pa_{\tau_t} \Delta u, \tau_t \rangle +\langle \pa_{\tau_t}^3 u, \tau_t \rangle+\pa_{\tau_t} \div \, \n_t \, \langle\pa_{\n_t} u, \tau_t\rangle \\
 - &3 \, \div \, \n_t \langle \pa_{\n_t}^2 u, \n_t \rangle - ( \div \, \n_t )^2 \langle \pa_{\n_t} u, \n_t \rangle  
\end{split}
\end{align}
and
\begin{align}
\label{<pa_n^3 u,tau>}
\begin{split}
\langle \pa_{\n_t}^3 u, \tau_t\rangle=
&
\langle \pa_{\n_t} \pa_{\tau_t}^2 u, \tau_t\rangle
 - 2\langle \pa_{\n_t} \pa_{\tau_t}^2 u, \tau_t \rangle - \pa_{\tau_t} \div \, \n_t \, \langle \pa_{\n_t} u,\n_t  \rangle\\
+&2 \, \div \, \n_t \,  \langle \pa_{\tau_t}^2 u + \, \div \, \n_\Omega \pa_{\n_t} u , \tau_\Omega\rangle -  \, \div \, \n_t \, \langle \pa_{\n_t} \pa_{\tau_t} u  , \n_t\rangle \\
-  &(\div \, \n_t)^2 \langle \pa_{\tau_t} u,\n_\Omega \rangle
\end{split}
\end{align}
in $\spt \eta \, \cap  \Omega_t$. Thus injecting  \eqref{<pa_n^3 u,n>} and \eqref{<pa_n^3 u,tau>}
into \eqref{Leading3} and using \eqref{div n est} gives us
\begin{align*}
|\nabla^3 u|^2 \leq &C\left(\langle \pa_{\tau_t} \Delta u, \tau_t \rangle^2 + |\nabla \pa_{\tau_t}^2u|^2 + |\pa_{\tau_t} \div \, \n_t \, \pa_{\n_t}u|^2 \right)\\
+ &\frac{C}{r_t^2} |\nabla^2 u|^2 + \frac{C}{r_t^4}  |\nabla u|^2.
\end{align*}
and, in turn, combining this with \eqref{Leading1} yields
\begin{align}
\label{Leading4}
\begin{split}
L \leq -&\vartheta c \int_{\Omega_t}\eta \, |\nabla^3 u|^2 \d  x
+\vartheta C  \int_{\Omega_t} \eta \, |\pa_{\tau_t} \div \, \n_t \, \pa_{\n_t} u|^2 \d x
 \\
+  &\vartheta \frac{C}{r_t^2} \int_{\Omega_t} |\nabla^2 u|^2 \d x + \vartheta \frac{C}{r_t^4} \int_{\Omega_t} |\nabla u|^2 \d x.
\end{split}
\end{align}
Note that it follows from the construction of $\eta$ that $\eta=1$ in $\{d_t \geq - r/4\}$. Thus recalling $\Delta^2 u= 0$ in $\Omega_t$ we may apply a standard interior estimate (Lemma \ref{lem:interiorreg})  
to obtain
\begin{equation}
\label{Leading5}
 \int_{\Omega_t}(1-\eta) |\nabla^3 u|^2 \d x \leq C \int_{\Omega_t} \frac{1}{r_t^2} |\nabla^2 u|^2 + \frac{1}{r_t^4}  |\nabla u|^2+
\frac{1}{r_t^6}  |u|^2 \d x.
\end{equation}
Finally, it follows from \eqref{divn2D} and $\spt \eta \subset\mathcal N_{r_t/2}(\pa \Omega_t)$
\begin{equation}
\label{int (pa_tau div n)^2 control}
\int_{\spt \eta \cap \Omega_t}  |\pa_{\tau_t} \div \, \n_t|^2 \d x
\leq C r_t  \int_{\pa \Omega_t} |\partial_{\tau_t} \curv_t|^2 \d \H^1
\end{equation}
and, hence, we infer \eqref{Leading_est} from \eqref{Leading4}, \eqref{Leading5} and \eqref{int (pa_tau div n)^2 control}.

\textbf{Step 2: Estimating the remainder:}
We show that the remainder contribution can be absorbed up to lower order terms  by the estimate \eqref{Leading_est}. To be more concrete, it holds
\begin{align}
\label{Remainder_est}
\begin{split}
R  \leq \,  &\frac{c}{2}\int_{\Omega_t} \eta \, |\nabla^3 u|^2 \d x+
C r_t \|\eta^\frac12 \pa_{\n_t} u\|_{L^\infty(\Omega_t)}^2
\int_{\pa \Omega_t} |\partial_{\tau_t} \curv_t|^2 \d \H^1
\\
+ &\frac{C}{r_t^2} \int_{\Omega_t}|\nabla^2u|^2 \d x + \frac{C}{r_t^4} \int_{\Omega_t}|\nabla u|^2 \d x
\end{split}
\end{align}
where $c$ is the same constant as in \eqref{Leading_est}.
We observe that decompositions \eqref{|nabla^2X|^2} and \eqref{|nabla^3X|^2} imply
\begin{equation}
\label{|pa_n^k| < |nabla^k|}
|\pa_{\n_t}^2 u| \leq |\nabla^2 u| \qquad \text{and} \qquad |\pa_{\n_t}^3 u| \leq  |\nabla^3 u|
\end{equation}
in $\spt \eta \, \cap \Omega_t$.

Recall that $R = \sum_{i=0}^4 R_i$. We estimate each $R_i$ separately.
Starting from $R_0$ we apply \eqref{div(divn*f*n)} on the first integral in \eqref{R0}
and Young's inequality for the second integral and consequently apply \eqref{div(f*n)} to
obtain
\begin{align*}
R_0 \leq 
 &-2\int_{\Omega_t} \div \, \n_ t \, \pa_{\n_t} \left(\eta \, (\langle \Delta u, \tau_t\rangle - 2 \pa_{\tau_t} \div_{\n_t} u)^2\right)\d x\\
& +\frac{r_t}{2} \int_{\pa \Omega_t} (\div_{\tau_t} u)^2 |\partial_{\tau_t} \curv_t|^2 \d \H^1 \\
&+\frac{9}{2 r_t}\int_{\Omega_t} \pa_{\n_t} \left(\eta \, (\langle \Delta u, \tau_t\rangle - 2 \pa_{\tau_t} \div_{\n_t} u)^2\right)
+\div \, \n_t \, \eta \,  (\langle \Delta u, \tau_t\rangle - 2 \pa_{\tau_t} \div_{\n_t} u)^2
 \d x
\end{align*}
and hence by the incompressibility, \eqref{diffeta2} and \eqref{div n est} we have
\begin{align}
\label{R0_est1}
\begin{split}
R_0 &\leq
 \frac{C}{r_t} \int_{\Omega_t}  \eta \, |\langle \Delta u, \tau_t\rangle - 2 \pa_{\tau_t} \div_{\n_t} u||\pa_{\n_t} (\langle \Delta u, \tau_t\rangle - 2 \pa_{\tau_t} \div_{\n_t} u)| \d x\\
& + r_t  \|\pa_{\n_t} u\|_{L^\infty(\pa \Omega_t)}^2 \int_{\pa \Omega_t} |\partial_{\tau_t} \curv_t|^2 \d \H^1 \\
&+\frac{C}{r_t^2}\int_{\spt \eta \cap \Omega_t}(\langle \Delta u, \tau_t\rangle - 2 \pa_{\tau_t} \div_{\n_t} u)^2\d x
 \d x.
\end{split}
\end{align}
Using the incompressibility, the basic identities \eqref{basicid} and \eqref{pa_n H} and the decomposition \eqref{Delta X} for $\Delta$
yields the identities
\[
 \langle \Delta u, \tau_t \rangle -2\pa_{\tau_t}\div_{\n_t} u = 2 \langle \Delta u, \tau_t \rangle - 2 \langle \pa_{\n_t}^2 u,\tau_t \rangle- 2 \, \div \, \n_t \langle \pa_{\n_t} u,\tau_t \rangle
\]
and
\begin{align*}
&\pa_{\n_t}\left(\langle \Delta u, \tau_t \rangle -2\pa_{\tau_t}\div_{\n_t} u\right) \\
&= 2 \langle \pa_{\n_t} \Delta u, \tau_t \rangle - 2 \langle \pa_{\n_t}^3 u,\tau_t \rangle- 
2 \, \div \, \n_t \langle \pa_{\n_t}^2 u,\tau_t \rangle
+ 2 (\div \, \n_t)^2 \langle \pa_{\n_t} u, \tau_t \rangle
\end{align*}
in $\spt \eta \, \cap \Omega_t$.
Further, using $|\Delta u| \leq C|\nabla^2 u|$, $|\nabla \Delta u| \leq C|\nabla^3 u|$, \eqref{|pa_n^k| < |nabla^k|},
\eqref{diffeta2} and \eqref{div n est} we have
\begin{align*}
 |\langle \Delta u, \tau_t \rangle -2\pa_{\tau_t}\div_{\n_t} u| &\leq C|\nabla^2 u| +\frac{C}{r_t}|\nabla u| \\
|\pa_{\n_t}\left(\langle \Delta u, \tau_t \rangle -2\pa_{\tau_t}\div_{\n_t} u\right)| &\leq   C|\nabla^3 u| +\frac{C}{r_t}|\nabla^2 u| + \frac{C}{r_t^2}|\nabla u|
\end{align*}
so injecting these into \eqref{R0_est1} and using Young's inequality yields
\begin{align}
\label{R0_est2}
\begin{split}
R_0 &\leq
C\int_{\Omega_t} \eta \, |\nabla^3 u| \left( \frac{|\nabla^2u|}{r_t}+ \frac{|\nabla u|}{r_t^2}\right) \d x + \frac{r_t}2\|\pa_{\n_t} u\|_{L^\infty(\pa \Omega_t)}^2 \int_{\pa \Omega_t} |\partial_{\tau_t} \curv_t|^2 \d \H^1 \\
&+ \frac{C}{r_t^2} \int_{\Omega_t}|\nabla^2u|^2 \d x + \frac{C}{r_t^4} \int_{\Omega_t}|\nabla u|^2 \d x.
\end{split}
\end{align}
The next remainder $R_1$ is straightforward to estimate. We start from 
\eqref{R1} and use \eqref{div n est} as well as $|\Delta u| \leq C|\nabla^2 u|$ and $|\nabla \Delta u| \leq C|\nabla^3 u|$
to obtain
\[
R_1 \leq 
C\int_{\Omega_t} \left( |\pa_{\n_t} \eta| + \frac{\eta}{r_t}  \right) |\nabla^2 u| |\nabla^3 u|\d x.
\]
We write
$\pa_{\n_t} \eta = 2 \eta^\frac12 \pa_{\n_t} \eta^\frac12$ and apply \eqref{diffeta2} on $\pa_{\n_t} \eta^\frac12$
to further have
\begin{equation}
\label{R1_est1}
R_1 \leq 
C\int_{\Omega_t} \eta^\frac12 |\nabla^3 u| \frac{|\nabla^2u|}{r_t}\d x.
\end{equation}
Estimating $R_2$ is rather similar to the case $R_0$. Indeed, we use the incompressibility, \eqref{eq:surfacetension2} and Young's inequality on
the first integral in \eqref{R2} to obtain
\begin{align*}
-&2\int_{\pa \Omega_t} 2 \pa_{\tau_t} \curv_t  \langle \pa_{\tau_t} u, \tau_t \rangle\langle \pa_{\tau_t}^2 u, \tau_t \rangle
+\pa_{\tau_t}\curv_t  \langle\pa_{\tau_t} u, \n_t \rangle  \langle \Delta u,\n_t\rangle
\d \H^1 \\
&\leq r_t \|\pa_{\n_t} u\|_{L^\infty(\pa \Omega_t)}^2 \int_{\pa \Omega_t}|\partial_{\tau_t} \curv_t|^2 \d \H^1 + \frac{10}{r_t}
 \int_{\pa \Omega_t} \langle \pa_{\tau_t}^2 u, \tau_t \rangle^2 + \langle \Delta u, \n_t\rangle^2 \d \H^1.
\end{align*}
Again applying  \eqref{div(f*n)} on the second integral and using \eqref{div n est} and \eqref{diffeta2} yields
\begin{align*}
 &\phantom{{}={}}\frac{10}{r_t} \int_{\pa \Omega_t} \langle \pa_{\tau_t}^2 u, \tau_t \rangle^2 + \langle \Delta u, \n_t\rangle^2 \d \H^1 \\
&\leq \frac{C}{r_t} \int_{\pa \Omega_t} \eta \, \left( |\langle \pa_{\tau_t}^2 u, \tau_t \rangle||\pa_{\n_t}\langle \pa_{\tau_t}^2 u, \tau_t \rangle| +
 |\langle \Delta u, \n_t\rangle|| \langle  \pa_{\n_t} \Delta u, \n_t\rangle| \right) \d x \\
&+\frac{C}{r_t^2} \int_{\spt \eta \cap  \Omega_t}  \langle \pa_{\tau_t}^2 u, \tau_t \rangle^2 +
 \langle \Delta u, \n_t\rangle^2 \d x. 
\end{align*}
Using the decomposition \eqref{Delta X} for $\Delta u$ we may write
$\langle \pa_{\tau_t}^2 u, \tau_t\rangle= \langle \Delta u, \tau_t \rangle-\langle \pa_{\n_t}^2 u, \tau_t \rangle-\div \, \n_t \, \langle \pa_{\n_t} u, \tau_t \rangle$
in $\spt \eta \, \cap \Omega_t$ and, hence, using the same argument as for $R_0$ we get estimates
\[
 |\langle \pa_{\tau_t}^2 u, \tau_t\rangle | \leq C \left(|\nabla^2 u| + \frac1{r_t}|\nabla u|\right)  \qquad \text{and} \qquad |\pa_{\n_t}\langle \pa_{\tau_t}^2 u, \tau_t \rangle| \leq
 C \left(|\nabla^3 u| + \frac1{r_t}|\nabla^2 u| + \frac1{r_t^2} |\nabla u|\right).
\]
Thus, combining the previous estimates and using Young's inequality yields
\begin{align}
\label{R2_est1}
\begin{split}
-2\int_{\pa \Omega_t} &2 \pa_{\tau_t} \curv_t  \langle \pa_{\tau_t} u, \tau_t \rangle\langle \pa_{\tau_t}^2 u, \tau_t \rangle
+\pa_{\tau_t}\curv_t  \langle\pa_{\tau_t} u, \n_t \rangle  \langle \Delta u,\n_t\rangle
\d \H^1 \\
\leq C&\int_{\Omega_t} \eta \, |\nabla^3 u| \left( \frac{|\nabla^2u|}{r_t}+ \frac{|\nabla u|}{r_t^2}\right) \d x + r_t\|\pa_{\n_t} u\|_{L^\infty(\pa \Omega_t)}^2 \int_{\pa \Omega_t} |\partial_{\tau_t} \curv_t|^2 \d \H^1 \\
+ \frac{C}{r_t^2} &\int_{\Omega_t}|\nabla^2u|^2 \d x + \frac{C}{r_t^4} \int_{\Omega_t}|\nabla u|^2 \d x.
\end{split}
\end{align}
For the remaining two integrals in \eqref{R2} we use the incompressibility, the boundary condition \eqref{eq:surfacetension2}
and  the decomposition \eqref{Delta X} for $\Delta$ to compute 
\begin{align*}
&-2\int_{\pa \Omega_t} 2\curv_t \langle \pa_{\tau_t}^2 u, \tau_t \rangle^2
+\curv_t  \langle\pa_{\tau_t}^2 u, \n_t \rangle  \langle \Delta u,\n_t\rangle
-\curv_t\langle\pa_{\tau_t}^2 u, \tau_t \rangle  \langle \Delta u,\tau_t\rangle 
\d \H^1\\
&-2\int_{\pa \Omega_t}
\curv_t^2 \left( \langle\pa_{\tau_t} u, \tau_t \rangle  \langle \Delta u,\n_t\rangle
+ \langle\pa_{\tau_t} u, \n_t \rangle \langle \Delta u,\tau_t\rangle 
-2 \langle \pa_{\tau_t} u, \n_t \rangle 
\langle \pa_{\tau_t}^2 u, \tau_t \rangle\right)
\d \H^1 \\
=&-2\int_{\pa \Omega_t} \curv_t (\langle \Delta u, \tau_t \rangle - 2\langle \pa_{\n_t}^2 u,\tau_t\rangle)(\langle \Delta u, \tau_t \rangle
-\langle \pa_{\n_t}^2 u, \tau_t \rangle 
)+\curv_t (\langle \Delta u,\n_t\rangle - \langle \pa_{\n_t}^2 u, \n_t \rangle  )  \langle \Delta u,\n_t\rangle \d \H^1 \\
&+4\int_{\pa \Omega_t} \curv_t^2 \langle \pa_{\n_t} u, \tau_t \rangle(\langle \Delta u, \tau_t \rangle - \langle \pa_{\n_t}^2 u,\tau_t\rangle)
+\curv_t^2 \langle\pa_{\n_t} u, \n_t \rangle  \langle \Delta u,\n_t\rangle\d \H^1.
\end{align*}
Further applying \eqref{div(f*n)} on the right hand side and using the
basic identities \eqref{basicid} and \eqref{pa_n H} for the frame $\{\n_t, \tau_t\}$ and the
 estimates \eqref{diffeta2}, \eqref{div n est}, $|\Delta u| \leq C|\nabla^2 u|$  and $|\nabla \Delta u| \leq C |\nabla^3 u|$  as well as \eqref{|pa_n^k| < |nabla^k|}
and finally applying Young's inequality
yields
\begin{align*}
&-2\int_{\pa \Omega_t} 2\curv_t \langle \pa_{\tau_t}^2 u, \tau_t \rangle^2
+\curv_t  \langle\pa_{\tau_t}^2 u, \n_t \rangle  \langle \Delta u,\n_t\rangle
-\curv_t\langle\pa_{\tau_t}^2 u, \tau_t \rangle  \langle \Delta u,\tau_t\rangle 
\d \H^1\\
&-2\int_{\pa \Omega_t}
\curv_t^2 \left( \langle\pa_{\tau_t} u, \tau_t \rangle  \langle \Delta u,\n_t\rangle
+ \langle\pa_{\tau_t} u, \n_t \rangle \langle \Delta u,\tau_t\rangle 
-2 \langle \pa_{\tau_t} u, \n_t \rangle 
\langle \pa_{\tau_t}^2 u, \tau_t \rangle\right)
\d \H^1 \\
\leq &C  \int_{\Omega_t} \eta \, |\nabla^3 u| \left( \frac{|\nabla^2u|}{r_t}+ \frac{|\nabla u|}{r_t^2}\right) \d x + \frac{C}{r_t^2} \int_{\Omega_t}|\nabla^2u|^2 \d x
+\frac{C}{r_t^4} \int_{\Omega_t}|\nabla u|^2 \d x.
\end{align*}
Combining this with \eqref{R2_est1} gives us
\begin{align}
\label{R2_est2}
\begin{split}
R_2 &\leq C\int_{\Omega_t} \eta \, |\nabla^3 u| \left( \frac{|\nabla^2u|}{r_t}+ \frac{|\nabla u|}{r_t^2}\right) \d x + r_t\|\pa_{\n_t} u\|_{L^\infty(\pa \Omega_t)}^2 \int_{\pa \Omega_t} |\partial_{\tau_t} \curv_t|^2 \d \H^1 \\
&+ \frac{C}{r_t^2} \int_{\Omega_t}|\nabla^2u|^2 \d x + \frac{C}{r_t^4} \int_{\Omega_t}|\nabla u|^2 \d x.
\end{split}
\end{align}
The remainder $R_3$ is more technical to estimate. Employing the same arguments for the last three integrals as previously
we conclude 
\begin{align}
\label{R3_est1}
\begin{split}
&\phantom{{}={}}4\int_{\Omega_t} \div \, \n_t \, \eta \, \left( \langle\pa_{\tau_t}^2 u, \n_t \rangle
+2\langle\pa_{\n_t}^2 u+ \div \, \n_t \, \pa_{\n_t} u, \n_t \rangle
\right)\langle \pa_{\n_t}^3 u, \tau_t\rangle
\d x \\
&+4\int_{\Omega_t}\pa_{\n_t} \eta \,  \langle \pa_{\tau_t}^2 u, \n_t \rangle
\langle \pa_{\n_t}^3 u, \tau_t\rangle \d x \\
&-4\int_{\Omega_t}2(\div \, \n_t )^2 \eta \, \langle  \pa_{\n_t}^2 u + \div \, \n_t \, \pa_{\n_t} u , \n_t \rangle
 \langle\pa_{\tau_t}^2 u, \n_t \rangle
+\div \, \n_t \, \pa_{\n_t} \eta \, \langle\pa_{\tau_t}^2 u, \n_t \rangle^2
 \d x \\
&\leq C  \int_{\Omega_t} \eta^\frac12 \, |\nabla^3 u| \left( \frac{|\nabla^2u|}{r_t}+ \frac{|\nabla u|}{r_t^2}\right) \d x + \frac{C}{r_t^2} \int_{\Omega_t}|\nabla^2u|^2 \d x
+\frac{C}{r_t^4} \int_{\Omega_t}|\nabla u|^2 \d x.
\end{split}
\end{align}
For the two first integrals in \eqref{R3} we need the following estimates
\begin{align*}
 |\pa_{\tau_t}\pa_{\n_t}^2 u| &\leq |\nabla^3 u|+ \sqrt 2 |\div \, \n_t||\nabla^2 u| \\
 |\pa_{\tau_t}^3 u| &\leq |\nabla^3 u|+ |\pa_{\tau_t} \div \, \n_t \, \pa_{\n_t} u| + \frac{3}{\sqrt2}|\div \, \n_t||\nabla^2 u| + (\div \, \n_t)^2 |\nabla u|
\end{align*}
in $\spt \eta \, \cap \Omega_t$ which follows from the decomposition \eqref{|nabla^3X|^2} for $|\nabla^3 u|^2$ in $\mathcal N_{r_t}(\pa \Omega_t)$. By using these and otherwise proceeding as earlier we obtain an estimate
\begin{align*}
&\phantom{{}={}}4\int_{\Omega_t}
\div \, \n_t \,\eta \,  
\left(
\langle \pa_{\tau_t} \pa_{\n_t}^2 u, \n_t\rangle \langle \pa_{\tau_t}^2 u, \tau_t \rangle
+ \left( \langle \Delta u , \tau_t\rangle + \langle\pa_{\tau_t}^2 u, \tau_t\rangle\right)\pa_{\tau_t}\langle \pa_{\tau_t}^2 u, \n_t \rangle
\right)
 \d x \\
&-4\int_{\Omega_t}\eta \, \pa_{\tau_t} \div \, \n_t \, \langle\pa_{\n_t} u, \tau_t\rangle \left(\langle\pa_{\tau_t}^3 u, \tau_t \rangle
-\div \, \n_t \, \langle\pa_{\tau_t}^2 u, \n_t \rangle \right)
\d x \\
&\leq C \int_{\Omega_t}\eta \, |\nabla^3 u|\left( \frac{|\nabla^2u|}{r_t}+ \frac{|\nabla u|}{r_t^2} + |\pa_{\tau_t} \div \, \n_t \pa_{\n_t} u| \right)\d x
+C \int_{\Omega_t}\eta \,|\pa_{\tau_t} \div \, \n_t \pa_{\n_t} u|^2 \d x \\
&+\frac{C}{r_t^2}  \int_{\Omega_t}|\nabla^2u|^2 \d x + \frac{C}{r_t^4} \int_{\Omega_t}|\nabla u|^2 \d x.
\end{align*}
We combine this with \eqref{R3_est1} to obtain
\begin{align}
\label{R3_est2}
\begin{split}
R_3 &\leq C \int_{\Omega_t}\eta^\frac12 |\nabla^3 u|\left( \frac{|\nabla^2u|}{r_t}+ \frac{|\nabla u|}{r_t^2} + |\eta^\frac12  \pa_{\tau_t} \div \, \n_t \pa_{\n_t} u| \right)\d x
 \\
&+C \int_{\Omega_t}|\eta^\frac12 \pa_{\tau_t} \div \, \n_t \pa_{\n_t} u|^2 \d x +\frac{C}{r_t^2}  \int_{\Omega_t}|\nabla^2u|^2 \d x + \frac{C}{r_t^4} \int_{\Omega_t}|\nabla u|^2 \d x.
\end{split}
\end{align}
Looking at  the expression \eqref{R4} of the last remainder we see that it can be estimated in similar fashion as $R_3$ and we get essentially the same estimate 
\begin{align}
\label{R4_est1}
\begin{split}
R_4 &\leq C \int_{\Omega_t}\eta^\frac12 |\nabla^3 u|\left( \frac{|\nabla^2u|}{r_t}+ \frac{|\nabla u|}{r_t^2} + |\eta^\frac12  \pa_{\tau_t} \div \, \n_t \pa_{\n_t} u| \right)\d x
 \\
&+C \int_{\Omega_t}|\eta^\frac12 \pa_{\tau_t} \div \, \n_t \pa_{\n_t} u|^2 \d x +\frac{C}{r_t^2}  \int_{\Omega_t}|\nabla^2u|^2 \d x + \frac{C}{r_t^4} \int_{\Omega_t}|\nabla u|^2 \d x.
\end{split}
\end{align}

Finally, by summing the estimates  \eqref{R0_est2}, \eqref{R1_est1},  \eqref{R2_est2}, \eqref{R3_est2} and \eqref{R4_est1} together and
applying Young's inequality we have
\begin{align*}
R  &\leq   \frac{c}{2}\int_{\Omega_t} \eta \, |\nabla^3 u|^2 \d x+ \frac{C}{r_t^2} \int_{\Omega_t}|\nabla^2u|^2 \d x + \frac{C}{r_t^4} \int_{\Omega_t}|\nabla u|^2 \d x \\
&+C r_t \|\pa_{\n_t} u\|_{L^\infty(\pa \Omega_t)}^2
\int_{\pa \Omega_t} |\partial_{\tau_t} \curv_t|^2 \d \H^1 + C \|\eta^\frac12 \pa_{\n_t} u\|_{L^\infty(\Omega_t)}^2 \int_{\spt \eta \cap \Omega_t} (\pa_{\tau_t} \div \, \n_t)^2 \d x.
\end{align*}
Therefore, recalling \eqref{int (pa_tau div n)^2 control} we infer \eqref{Remainder_est} from the above estimate.

\textbf{Step 3: The final estimate:}
By virtue of \eqref{Leading_est} and \eqref{Remainder_est} we have the following estimate
\begin{align}
\label{1stvar_higherwillmore:estA}
\begin{split}
\frac\d{\d t} &\int_{\pa \Omega_t}  |\partial_{\tau_t} \curv_t|^2 \d \H^1  \\
\leq&
-\vartheta c \int_{\Omega_t} |\nabla^3 u|^2 \d x
+ \vartheta C r_t \|\eta^\frac12 \nabla u\|^2_{L^\infty(\Omega_t)}  \int_{\pa \Omega_t} |\partial_{\tau_t} \curv_t|^2   \d \H^1
 \\
&+ \vartheta \frac{C}{r_t^2}\int_{\Omega_t} |\nabla^2 u|^2 \d x +  \vartheta \frac{C}{r_t^4}\int_{\Omega_t} |\nabla u|^2 \d x
+  \vartheta \frac{C}{r_t^6}\int_{\Omega_t} |u|^2 \d x.
\end{split}
\end{align}
In order to get rid of the $L^\infty$-norm, we apply \eqref{interpolation:Linfty} on $\eta^\frac12 \nabla u$ 
\begin{equation}
\|\eta^\frac12 \nabla u\|^2_{L^\infty(\Omega_t)} \leq C\left(\|\nabla^2(\eta^\frac12 \nabla u)\|_{L^2(\Omega_t)}\| \nabla u\|_{L^2(\Omega_t)} + \frac1{r_t^2}\|\nabla u\|_{L^2(\Omega_t)}^2 \right).
\end{equation}
Recalling \eqref{diffeta2} for $\eta^\frac12$ we have
\begin{align*}
|\nabla^2(\eta^\frac12 \nabla u)| 
&\leq \eta^\frac12 |\nabla^3 u| +  2 |\nabla \eta^\frac12||\nabla^2 u| + |\nabla^2 \eta^\frac12| |\nabla u| \\
&\leq \eta^\frac12 |\nabla^3 u| +  \frac{C}{r_t} |\nabla^2 u| + \frac{C}{r_t^2}|\nabla u|
\end{align*}
in $\Omega_t$ and, hence, it holds
\begin{align*}
\|&\eta^\frac12 \nabla u\|^2_{L^\infty(\Omega_t)} \\
 &\leq C\left(\|\eta^\frac12 \nabla^3 u\|_{L^2(\Omega_t)}\| \nabla u\|_{L^2(\Omega_t)} + \frac1{r_t}\|\nabla^2 u\|_{L^2(\Omega_t)}\| \nabla u\|_{L^2(\Omega_t)} + \frac1{r_t^2}\|\nabla u\|_{L^2(\Omega_t)}^2 \right).
\end{align*}
Again, combining this estimate with \eqref{1stvar_higherwillmore:estA} and applying Young's inequality 
results in \eqref{1stvar_higherwillmore:estB}.
\end{proof}


\section{Uniform curvature regularity}

We now show that the stability estimates do lead into closed regularity estimates.
 For this we begin with the uniform estimates on curvature and the uniform ball condition, before we show their continuity. Following the idea of Proposition \ref{prop:UBC}, we will have to consider both the change of curvature and the closeness of other parts of the boundary in normal direction. We will do so separately. In this section, we assume that \solT is a $C^3$-regular solution.


\subsection{Gronwall estimate for curvature control}

Here we deduce quantified growth estimates, from the derived stability estimates,
for $\norm{\partial_{\tau_t}\curv_t }_{L^2(\pa \Omega_t)}^2$ and some important 
consequences.
\begin{lemma}[Gronwall estimate]
\label{lem:gronwall}
Let \solT be a $C^3$-regular solution. There are
independent constants $c$ and $C$ such that
\beq
\label{est:gronwall1}
\vartheta c \|u\|_{H^3(\Omega_t)}^2 +\frac{\d}{\d t} \norm{\partial_{\tau_t}\curv_t }_{L^2(\pa \Omega_t)}^2
\leq \vartheta C\norm{\nabla u}_{L^2(\Omega_t)}^2 \frac{P_0^8+1}{r_t^6}
\left(1 +\norm{\partial_{\tau_t}\curv_t }_{L^2(\pa \Omega_t)}^4 \right)
\eeq
 and
\beq
\label{est:gronwall2}
\vartheta  \frac{\d}{\d t}  \norm{\partial_{\tau_t}\curv_t }_{L^2(\pa \Omega_t)}^2
 \leq C \frac{P_0^{21}+1}{r_t^{17}} 
\left(1 +\norm{\partial_{\tau_t}\curv_t }_{L^2(\pa \Omega_t)}^4 \right).
\eeq
\end{lemma}

\begin{proof}
Let us rewrite \eqref{1stvar_higherwillmore:estB}
\begin{align*}
\begin{split}
\frac\d{\d t} &\int_{\pa \Omega_t}  |\partial_{\tau_t} \curv_t|^2 \d \H^1  \\
\leq&
-\vartheta c \int_{\Omega_t} |\nabla^3 u|^2 \d x +  \vartheta C r^2_t\int_{\Omega_t} |\nabla u|^2 \d x  \left(\int_{\pa \Omega_t} |\partial_{\tau_t} \curv_t|^2   \d \H^1 \right)^2\\
&+ \vartheta \frac{C}{r_t} \int_{\Omega_t} |\nabla u|^2 \d x \int_{\pa \Omega_t} |\partial_{\tau_t} \curv_t|^2   \d \H^1 \\
&+ \vartheta \frac{C}{r_t^2}\int_{\Omega_t} |\nabla^2 u|^2 \d x +  \vartheta \frac{C}{r_t^4}\int_{\Omega_t} |\nabla u|^2 \d x
 +\vartheta \frac{C}{r_t^6}\int_{\Omega_t} |u|^2 \d x.
\end{split}
\end{align*}
By virtue of \eqref{interpolation:L2} we have an interpolation estimate
\[
\norm{\nabla^2 u}^2_{L^2(\Omega_t)}\leq C\norm{\nabla^3 u}_{L^2(\Omega_t)}\norm{\nabla u}_{L^2(\Omega_t)}+\frac{C\norm{\nabla u}^2_{L^2(\Omega_t)}}{r_t^2}.
\]
Thus using the interpolation, Young's inequality and the estimate \eqref{est:r_t} we infer \eqref{est:gronwall1}
from \eqref{1stvar_higherwillmore:estB}. 

For the second estimate \eqref{est:gronwall2} we combine the estimates 
\eqref{est:H1control} , \eqref{est:symmcontrol}, \eqref{est:r_t}, $P_t \leq P_0$,  $|\curv_t|\leq1/r_t$ as well as the isoperimetric inequality $4 \pi |\Omega_0| \leq P_0^2$
to conclude
\beq
\label{est:H1control2}
\norm{u}^2_{H^1(\Omega_t)} \leq C\left(1+\frac{|\Omega_0|P_t}{r_t}\right)\norm{\nabla u}^2_{L^2(\Omega_t)} 
 \leq C \frac{P_0^6+1}{r_t^4}\norm{\nablasym u}^2_{L^2(\Omega_t)}
\eeq
and
\beq
\label{est:symmcontrol2}
\vartheta^2 \norm{\nablasym u}^2_{L^2(\Omega_t)} \leq C \frac{|\Omega_0|^2P_t^2}{r^5_t}
\|\curv_t\|_{L^2(\pa \Omega_t)}^2 \leq C \left(\frac{P_0}{r_t}\right)^7.
\eeq
By combining these estimates with  \eqref{est:gronwall1} and increasing $C$, if necessary,
we obtain \eqref{est:gronwall2}.
\end{proof}

Since $t \mapsto \Omega_t$ is continuous in the $C^2$-topology, then $t \mapsto r_t$ is continous
and we may set the treshold time
\beq
\label{maxT}
T_0:=\sup \{t \in [0,T): \text{$r_s \geq r_0/2$  for every $s \in [0,t]$} \}
\eeq
as a positive number. The second Gronwall inequality then yields
\beq
\label{growth1}
\arctan \left( \norm{\pa_{\tau_t}\curv_t}_{L^2(\pa \Omega_t)}^2 \right) \leq 
\arctan \left( \norm{\pa_{\tau_0}\curv_0}_{L^2(\pa \Omega_0)}^2 \right) +
C \frac{P_0^{21}+1}{\vartheta r_0^{17}}t 
\eeq
for every $t \in [0,T_0)$. We again set
\begin{align}
\label{K0}
\arctan K_0
=&\frac{\arctan \left( \norm{\pa_{\tau_0}\curv_0}_{L^2(\pa \Omega_0)}^2 \right)+\pi/2}2  \\
\label{T1}
\text{and} \quad \quad \quad T_1=& c \frac{\vartheta r_0^{17}}{P_0^{21}+1}\left(
\arctan K_0 - \arctan \left( \norm{\pa_{\tau_0}\curv_0}_{L^2(\pa \Omega_0)}^2 \right)
\right)
\end{align}
where $c=1/C$ for the constant $C$ in \eqref{growth1}.
Then the right hand side of \eqref{growth1} is bounded by $\arctan K_0$ for every $0 \leq t < 
\min\{T_0,T_1\}$ and hence
\beq
\label{growth2}
\norm{\pa_{\tau_t} \curv_t}_{L^2(\pa \Omega_t)}^2 \leq \tan 
\left(
\arctan \left( \norm{\pa_{\tau_0}\curv_0}_{L^2(\pa \Omega_0)}^2 \right) +
C \frac{P_0^{21}+1}{\vartheta r_0^{17}}t 
\right) \leq  K_0.
\eeq
We finally set the second treshold time
\beq
\label{maxT2}
\widetilde T_0 = \sup \{0 \leq t<T : r_s \geq r_0/2 \ \ \text{and} \ \ \norm{\pa_{\tau_s}\curv_s}_{L^2(\pa \Omega_s)}^2 \leq K_0 \ \ \text{for every} \ \ s \in [0,t] \}. 
\eeq
Thanks to the previous discussion we have
\beq
\label{est:time_lowerbound}
 \widetilde T_0 \geq  \min\{T_0,T_1\}
\eeq
but $T_0$ still depends on the assumption of $r_t\geq r_0/2$. We will eventually show that
$\widetilde T_0 \geq  \min\{T,T^*\}$ for some positive number 
$T^*$ depending only on the quantities $P_0,r_0,\vartheta$ and 
$\norm{\pa_{\tau_0}\curv_0}_{L^2(\pa \Omega_0)}^2$, see Corollary
\ref{cor:time_lowerbound}.

We have a quantitative $L^2H^3/L^\infty H^2$-control
of $u$ in $[0,\widetilde T_0)$. For sake of simplicity we start to use notation $C(P_0,K_0)$ instead
of trying to write them explicitly in the following estimates.

\begin{lemma}
\label{lem:normcontrol}
Let \solT be a $C^3$-regular solution and let $\widetilde T_0$ and $K_0$ be according to \eqref{maxT2} and \eqref{K0} respectively.
Then it holds
\begin{align*}
\vartheta \int_0^{\widetilde T_0} \norm{u}_{H^3(\Omega_t)}^2 \d t  \leq 
 \frac{C(P_0,K_0)}{r_0^{10}}&,  \ \ \
\vartheta \, \sup_{\mathclap{t \in [0,\widetilde T_0)}} \norm{u}_{H^2(\Omega_t)}
 \leq 
 \frac{C(P_0,K_0)}{r_0^6}  \\
\text{and consequently} \ \ \  \vartheta \, \sup_{\mathclap{t \in [0,\widetilde T_0)}} \norm{u}_{L^\infty(\Omega_t)}
 &\leq 
 \frac{C(P_0,K_0)}{r_0^{7}}.
\end{align*}
\end{lemma}

\begin{proof}
Combining the estimate \eqref{est:H1control2},  $r_t\geq r_0/2$ and the the dissipation formula \eqref{eq:dissipation2} 
results in 
\[
\vartheta \int_0^{\widetilde T_0} \norm{u}_{H^1(\Omega_t)}^2 \d t \leq \vartheta C \frac{P_0^6+1}{(r_0/2)^4} \int_0^{\widetilde T_0} \norm{\nablasym u}_{L^2(\Omega_t)}^2 \d t \leq 
C \frac{P_0^7+1}{r_0^4}.
\]
Hence by integrating \eqref{est:gronwall1} in time and using  the previous estimate as well as
\eqref{est:r_t}\footnote{Note that since $r_t \leq P_0$, we always have $r_t^{-a} + r_t^{-b} \leq(1+ P_0^{|a-b|}) r_t^{-\max(a,b)}$
for positive $a$ and $b$, i.e.\ up to a constant only the highest order term is significant.} 
and $\norm{\partial_{\tau_0}\curv_0 }_{L^2(\pa \Omega_0)}^2 \leq K_0$
we estimate
\begin{align*}
\vartheta \int_0^{\widetilde T_0} \norm{u}_{H^3(\Omega_t)}^2 \d t
&\leq \frac{C(P_0)(K_0^2+1)}{(r_0/2)^6} \vartheta \int_0^{\widetilde T_0} \norm{u}_{H^1(\Omega_t)}^2 \d t + 
\norm{\partial_{\tau_0}\curv_0 }_{L^2(\pa \Omega_0)}^2
 \\
&\leq \frac{C(P_0)(K_0^2+1)}{r_0^{10}} + K_0 \\
&\leq \frac{C(P_0,K_0)}{r_0^{10}}.
\end{align*}
Again, by  the first inequality in \eqref{est:H1control}, \eqref{est:symmcontrol}, isoperimetric inequality,
$|\curv_t|\leq 1/r_t$, \eqref{est:r_t}  and $r_t \geq r_0/2$
we have
\[
\vartheta \norm{\nabla u}_{L^2(\Omega_t)} \leq 
\vartheta \frac{C(P_0)}{r_t^{3/2}} \norm{\nablasym u}_{L^2(\Omega_t)} \leq 
\frac{C(P_0)}{r_t^4} \|\curv_t\|_{L^2(\pa \Omega_t)}
\leq \frac{C(P_0)}{r_0^5}
\]
and, hence, using Lemma  \ref{lem:H2control} and the estimates \eqref{est:r_t}, $\norm{\partial_{\tau_0}\curv_0 }_{L^2(\pa \Omega_0)}^2 \leq K_0$  and $r_t \geq r_0/2$
\begin{align*}
\vartheta^2 \norm{\nabla^2 u}_{L^2(\Omega_t)}^2 
\leq P_0 K_0 + \frac{C}{(r_0/2)^2} \vartheta^2  \norm{\nabla u}^2_{L^2(\Omega_t)} 
\leq P_0 K_0 + \frac{C(P_0)}{r_0^{12}}
\leq \frac{C(P_0,K_0)}{r_0^{12}}.
\end{align*}
On the other hand, combining \eqref{est:H1control2}, \eqref{est:symmcontrol2} and $r_t \geq r_0/2$ gives us
\[
\vartheta^2 \norm{\nabla u}_{L^2(\Omega_t)}^2  \leq \frac{C(P_0,K_0)}{r_0^{11}}.
\]
and hence the second estimate of the lemma follows from the previous estimates.
Finally, the third estimate is a direct consequence of the second estimate, the interpolation 
\eqref{interpolation:Linfty} and $r_t \geq r_0/2$.
\end{proof}

The previous lemma together with the normal velocity condition implies quantified Lipschitz continuity of the Hausdorff-distance
between the boundaries in $[0,\widetilde T_0)$.
\begin{prop}[Continuity of Hausdorff distance]
\label{prop:Hausdorff}
Let \solT be a moment-preserving solution and let $\widetilde T_0$ and $K_0$ be according to \eqref{maxT2} and \eqref{K0} respectively. Then for every $t,s \in [0,\widetilde T_0)$
\[
\vartheta \, d_\H(\pa \Omega_t,\pa \Omega_s) \leq \frac{C(P_0,K_0)}{r_0^{7}} |t-s|.
\]
\end{prop}


\subsection{Continuity of the curvature}
We show that with the regained quantitative control on solution
we may establish a quantitative control on  the $1/3$-Hölder-continuity of the curvature in the space-time. 

\begin{prop} \label{prop:kappaCampanato}
 Let $\Omega$ be a bounded $C^2$-regular domain, $x_0 \in \pa \Omega$, $0 < \rho < r_\Omega$ and $\phi \in C_0(B_\rho(x_0) \cap \pa \Omega;[0,1])$ such that
$\phi$ is not identically zero. Then there exists an independent constant $C$ such that for all $x \in B_{\rho}(x_0) \cap \pa \Omega$ and all $g\in C^1(\partial\Omega)$ we find
 \begin{align*}
  \left| g(x) - \frac{1}{\int_{\pa\Omega} \phi \d \H^1} \int_{\pa \Omega}  \phi g \d \H^1 \right| \leq C \rho^{1/2} \|\partial_\tau g\|_{L^2(\pa \Omega)}.
 \end{align*}
\end{prop}

\begin{proof}
The UBC condition and $\rho<r_\Omega$ together imply that 
$B_\rho(x) \cap \pa \Omega$ is connected and 
\beq
\label{unifdensity}
2 \rho \leq \H^1(B_\rho(x) \cap \pa \Omega) \leq \pi \rho.
\eeq
 We can thus estimate as follows
 \begin{align*}
  &\phantom{{}={}} \left| g(x) - \frac{1}{\int_{\pa \Omega} \phi \d \H^1} \int_{\pa \Omega} \phi g\d \H^1 \right|
   \leq \frac{1}{\int_{\pa \Omega} \phi \d \H^1} \int_{\pa \Omega} | g(x)-g| \phi \d \H^1 \\
   & \leq \frac{1}{\int_{\pa \Omega} \phi \d \H^1} \int_{\pa \Omega}  \int_{B_\rho(x_0) \cap \pa \Omega} |\pa_\tau g| \d \H^1  \phi \d \H^1 = \int_{B_\rho(x_0) \cap \pa \Omega} |\pa_\tau g| \d \H^1 \\
   &\leq \H^1(B_\rho(x_0) \cap \pa \Omega)^\frac12 \|\pa_\tau g\|_{L^2(B_\rho(x) \cap \partial \Omega)} \leq \sqrt{\pi} \rho^\frac12 \|\partial_\tau g\|_{L^2(\partial \Omega)}.\qedhere
 \end{align*}
\end{proof}

We then use this to prove continuity of the curvature.

\begin{lemma} 
\label{lem:curvtime}
Let $(\Omega_t,u,p)_{t \in [0,T)}$ be a $C^3$-regular solution
and  let $\widetilde T_0$ and $K_0$ be according to \eqref{maxT2} and \eqref{K0} respectively. 
Then
\[
\abs{\curv_t(x)-\curv_s(y)}\leq \frac{C(P_0,K_0)}{r_0^{8}} \rho^\frac{1}{2}
\]
whenever $0 \leq s \leq t < \widetilde T_0$, $x \in \pa \Omega_t$ and $y \in \pa \Omega_s$ satisfy $\abs{t-s}\leq \vartheta \rho^\frac{3}{2}$ and $\abs{x-y}\leq \rho/16$ for $0<\rho<r_0/2$.
\end{lemma}

\begin{proof}
In the proof $c$ and $C$ are generic independent constants which may change from line to line.
Choose $\phi\in C^2_0(\R^2,[0,1])$, such that $\phi=1$ in $B_{\rho/8}(x)=1$, $\spt\, \phi \subset B_{\rho/4}(x)$
and 
\beq
\label{curvtime:proof1}
\norm{\nabla \phi}_{L^\infty(\R^2)} \leq \frac{C}\rho.
\eeq
Since $|x-y|\leq \rho/16$, then $\phi=1$ in $B_{\rho/16}(y)$ and $\spt \phi \subset B_{\rho/2}(y)$. Thus recalling the density estimate \eqref{unifdensity}
yields
\beq
\label{curvtime:proof2}
c \rho \leq \norm{\phi}_{L^1(\partial\Omega_s)}, \quad \norm{\phi}_{L^1(\partial\Omega_t)} \leq C\rho.
\eeq
Furthermore, it follows from the construction of $\phi$ that 
if there is $z \in \spt \phi \cap \pa \Omega_h$ for $h \in [0,T)$, 
then $\spt(\phi) \subset B_{\rho/2}(z)$. In particular, this implies via \eqref{unifdensity}
and  \eqref{curvtime:proof1}
\beq
\label{curvtime:proof3}
\|\phi\|_{L^1(\pa \Omega_h)} \leq C \rho, \ \ 
\|\phi\|_{L^2(\pa \Omega_h)} \leq C \rho^\frac12
 \ \ \text{and} \ \ \|\nabla \phi\|_{L^1(\pa \Omega_h)} \leq C
\eeq
for every $h \in [0,\widetilde T_0)$. Now we split time and space as follows:
\begin{align*}
\abs{\curv_t(x)-\curv_s(y)}&\leq \abs{\curv_t(x)-\frac{1}{\norm{\phi}_{L^1(\partial\Omega_t)}}\int_{\partial\Omega_t} \phi\curv_t \d \H^1}
\\
&+\abs{\frac{1}{\norm{\phi}_{L^1(\partial\Omega_t)}} - \frac{1}{\norm{\phi}_{L^1(\partial\Omega_s)}} }\abs{\int_{\partial\Omega_t} \phi \curv_t \d \H^1}
\\
&+\frac{1}{\norm{\phi}_{L^1(\partial\Omega_s)}}\abs{\int_{\partial\Omega_t}\phi \curv_t \d\H^1-\int_{\partial\Omega_s} \phi \curv_s\d \H^1}
\\
&+\abs{\curv_s(y)-\frac{1}{\norm{\phi}_{L^1(\partial\Omega_s)}}\int_{\pa \Omega_s} \phi \curv_s \d \H^1}.
\end{align*}
Since $\spt(\phi) \subset B_\rho(x) \cap B_\rho(y)$, the first and the last term can be estimated against $C K_0^\frac12\rho^{\frac{1}{2}}$ using Proposition \ref{prop:kappaCampanato}
and the assumption $t,s < \widetilde T_0$. By $|\curv_t| \leq 2/r_0$ and \eqref{curvtime:proof3}
we have
\[
\abs{\int_{\partial\Omega_t} \phi\curv_t\d \H^1} \leq \frac{C\rho}{r_0}. 
\]
Using the previous observations and \eqref{curvtime:proof2} we obtain
\begin{align}
\label{curvtime:proof4}
\begin{split}
&\abs{\curv_t(x)-\curv_s(y)}\\
&\leq CK_0^\frac12\rho^\frac12 + \frac{C}{r_0\rho} \int_s^t  \abs{ \frac{\d}{\d h}\int_{\pa \Omega_h} \!\phi \d \H^1} \d h 
+ \frac{C}{\rho} \int_s^t  \abs{\frac{\d}{\d h}\int_{\pa \Omega_h}  \phi \curv_h \d \H^1} \d h.
\end{split}
\end{align}
and it remains to estimate those time derivatives.
Starting from the identity \eqref{ratestructure_ambient}, integrating by parts and using
Weingarten identities \eqref{weingarten}, $|\curv_s|\leq 2/r_0$,  \eqref{est:r_t}  and  Lemma \ref{lem:normcontrol}
we have
\begin{align}
\label{curvtime:proof5}
\begin{split}
\abs{ \frac{\d}{\d h}\int_{\pa \Omega_h} \!\phi \d \H^1}
&= \abs{\int_{\pa \Omega_h} (\pa_{\n_t} \phi + \curv_t \phi)\langle u, \n_t \rangle \d \H^1} \\
&\leq \left(2 \norm{\nabla \phi}_{L^1(\pa \Omega_h)} + \frac1{r_0}  \norm{\phi}_{L^1(\pa \Omega_h)}\right) \|u\|_{L^\infty(\Omega_h)} \\
&\leq \frac{C(|\Omega_0|,P_0,K_0)}{\vartheta r_0^{8}}.
\end{split}
\end{align}
The second time derivative is more technical to estimate. First, by \eqref{daux2/dt} 
we have
 \begin{align}
\label{curvtime:proof6}
\begin{split}
&\frac{\d}{\d h}  \int_{\pa \Omega_h}  \phi \curv_h \d \H^1 \\
&=\int_{\pa \Omega_t} \pa_{\tau_t} \phi   \langle  \pa_{\tau_t} u, \n_t\rangle  \d \H^1
+ \int_{\pa \Omega_t} \curv_t   \langle \phi, \nabla u \rangle  \d \H^1 
\end{split}
\end{align}
For the second integral, by using
 $|\curv_s|\leq 2/r_0$, \eqref{curvtime:proof3} and
Lemma \ref{lem:normcontrol}
we obtain
\begin{align}
\label{curvtime:proof7}
\begin{split}
\abs{\int_{\pa \Omega_t} \curv_t  \langle \nabla \phi, u \rangle  \d \H^1}
&\leq \frac{C(P_0,K_0)}{\vartheta r_0^{8}}.
\end{split}
\end{align}
Assuming that $\spt(\phi) \cap \pa \Omega_h$ is non-empty
it holds $\spt(\phi) \subset \mathcal N_{r_h/2}(\pa \Omega_h)$.
We then transform the first integral into bulk using $\pa_{\tau_h} \phi$ as a cutoff function in
the identity \eqref{div(f*n)}. Using also \eqref{basicid}, \eqref{changeoforder} and integration by parts in the 
 $\tau_h$-direction we compute
\begin{align*}
&\int_{\pa \Omega_h}\pa_{\tau_h} \phi \langle \partial_{\tau_h} u , \n_h \rangle \d \H^1 \\
&=\int_{\Omega_h} \pa_{\n_h} \pa_{\tau_h} \phi \langle \partial_{\tau_h} u , \n_h \rangle 
+\pa_{\tau_h} \phi \langle \pa_{\n_h} \partial_{\tau_h} u , \n_h \rangle 
+ \div \n_h \pa_{\tau_h} \phi \langle \partial_{\tau_h} u , \n_h \rangle \d x \\
&=\int_{\Omega_h} \pa_{\tau_h} \pa_{\n_h} \phi  \langle \partial_{\tau_h} u , \n_h \rangle 
+\pa_{\tau_h} \phi \langle \pa_{\n_h} \partial_{\tau_h} u , \n_h \rangle \d x \\
&=\int_{\Omega_h} - \pa_{\n_h} \phi  \langle \partial_{\tau_h}^2 u , \n_h \rangle 
- \div \, \n_h \pa_{\n_h} \phi  \langle \partial_{\tau_h} u , \tau_h \rangle 
+\pa_{\tau_h} \phi \langle \pa_{\n_h} \partial_{\tau_h} u , \n_h \rangle \d x. 
\end{align*}
Thus recalling the decomposition \eqref{|nabla^2X|^2}, using  $|\div \nu_h| \leq 4/r_0$ in $\mathcal N_{r_h/2}(\pa \Omega_h)$,   \eqref{est:r_t}
and Lemma \ref{lem:normcontrol} 
as well as observing $\norm{\nabla \phi}_{L^2(\R^2)} \leq C$
we estimate
\begin{align}
\label{curvtime:proof8}
\begin{split}
\abs{\int_{\pa \Omega_h}\pa_{\tau_h} \phi \langle \partial_{\tau_h} u , \n_h \rangle \d \H^1}
&\leq C \norm{\nabla \phi}_{L^2(\R^2)} \left( \norm{\nabla^2 u}_{L^2(\Omega_t)} + \frac1{r_0} \norm{\nabla u}_{L^2(\Omega_t)} \right) \\
&\leq C  \left( 1 + \frac1{r_0}\right)\norm{u}_{H^2(\Omega_t)} \\
&\leq \frac{C(P_0,K_0)}{\vartheta r_0^{7}}.
\end{split}
\end{align}
Combining \eqref{curvtime:proof6}, \eqref{curvtime:proof7} and \eqref{curvtime:proof8}
yields
\[
 \abs{\frac{\d}{\d h}\int_{\pa \Omega_h} \curv_h\phi \d \H^1} \leq \frac{C(P_0,K_0)}{\vartheta r_0^8}
\]
and hence injecting this estimate and \eqref{curvtime:proof5} into \eqref{curvtime:proof4}
results in
\[
\abs{\curv_t(x)-\curv_s(y)} \leq \frac{C(P_0,K_0)}{r_0^8} \left(\rho^\frac12 + \frac{|t-s|}{\vartheta \rho} \right).
\]
This finishes the proof by using the assumption that $\abs{t-s}\leq\vartheta \rho^\frac{3}{2}$.
\end{proof}


\subsection{Continuity of maximal UBC radius}

\begin{lemma} 
\label{lem:normaltime}
Let \solT be a $C^3$-regular solution
and  let $\widetilde T_0$ and $K_0$ be according to \eqref{maxT2} and \eqref{K0} respectively. 
Then
\[
\abs{\n_t(x)-\n_s(y)}\leq \frac{C(P_0,K_0)}{r_0^{7}} \rho
\]
whenever $0 \leq s \leq t \leq \widetilde T_0$, $x \in \pa \Omega_t$ and $y \in \pa \Omega_s$ satisfy $\abs{t-s}\leq \vartheta \rho^2$ and $\abs{x-y}\leq \rho/8$ for $0<\rho<r_0/2$.
\end{lemma}

\begin{proof}
The proof is almost a copy of the proof of Lemma \ref{lem:curvtime} and we just highlight
the differences. 

First of all, using the argument in the proof of Proposition \ref{prop:kappaCampanato}
as well as \eqref{est:Lipnu} we infer 
\[
 \abs{\n (x) - \frac{1}{\int_{\pa\Omega} \phi \d \H^1} \int_{\pa \Omega} \phi \nu \d \H^1} \leq \frac{\rho}{r_\Omega}
\]
with the same assumptions as in the aforementioned proposition.

The cutoff is otherwise the same but now 
$B_{\rho/4}(x)=1$ and $\spt (\phi) \subset B_{\rho/2}(x)$.
We then proceed in similar fashion to obtain
\begin{align*}
\begin{split}
&\abs{\n_t(x)-\n_s(y)}\\
&\leq \frac{C}{r_0} \rho + \frac{C}{\rho} \int_s^t  \abs{ \frac{\d}{\d h}\int_{\pa \Omega_h} \!\phi \d \H^1} \d h 
+ \frac{C}{\rho} \int_s^t  \abs{\frac{\d}{\d h}\int_{\pa \Omega_h} \phi \nu_h \d \H^1} \d h.
\end{split}
\end{align*}
The first time derivate can be estimated
as in \eqref{curvtime:proof5}.
For the second time derivative we use identity \eqref{daux1/dt} 
compute
\[
\frac{\d}{\d h}\int_{\pa \Omega_h} \phi \nu_h \d \H^1 
=\int_{\pa \Omega_h} \langle u, \n_h\rangle \nabla \phi \d \H^1 
\]
Thus using the same arguments as in the proof of Lemma \ref{lem:curvtime} we conclude
\[
\abs{\frac{\d}{\d h} \int_{\pa \Omega_h} \phi \nu_h \d \H^1}
\leq \frac{C(P_0,K_0)}{\vartheta r_0^{7}}
\]
and the claim follows.
\end{proof}

\begin{thm}
\label{thm:UBC}
Let \solT be a classical solution and let $\widetilde T_0$ and $K_0$ be according to \eqref{maxT2} and \eqref{K0} respectively. Then for every $t,s \in [0,\widetilde T_0)$ it holds
\[
|r_t-r_s| \leq \frac{C(P_0,K_0)}{r_0^8} \left(\frac{|t-s|}{\vartheta}\right)^\frac13.
\]
\end{thm}

\begin{proof}
Without loss of generality we may assume that the solution is moment preserving. The second standing assumption is that $t,s \in [0,\widetilde T_0)$. Due to symmetricity of the employed arguments it suffices to show
\beq
\label{proof:UBC1}
r_t \leq r_s + \frac{C(P_0,K_0)}{r_0^8} \left(\frac{|t-s|}{\vartheta}\right)^\frac13.
\eeq
The main idea is to consider separately the cases 
\[
\frac{|t-s|}{\vartheta} \geq h \qquad \text{and} \qquad \frac{|t-s|}{\vartheta} < h 
\]
where $h$ is a treshold value to be determined.

\textbf{Step 1:} We want $h$ to be chosen such that 
we may utilize Lemma \ref{lem:curvtime} and Lemma \ref{lem:normaltime}
when we are below the treshold. By Proposition \ref{prop:Hausdorff} we have $d_\H(\pa \Omega_t, \pa \Omega_s) \leq  \theta |t-s|/\vartheta$
where the coefficient $\theta$ is of the form
\beq
\label{proof:UBC2}
\theta = \frac{C(P_0,K_0)}{r_0^{7}}.
\eeq
Assuming 
\beq
\label{proof:UBC3}
\theta \frac{|t-s|}{\vartheta} \leq \frac{r_0}{2}
\eeq
holds true we have $d_\H(\pa \Omega_t, \pa \Omega_s) \leq r_t$
and, in particular,
\beq
\label{proof:UBC4}
\norm{\id - \pi_t}_{L^\infty(\pa \Omega_s)} \leq d_\H(\pa \Omega_t, \pa \Omega_s).
\eeq
Under the assumption \eqref{proof:UBC3} we choose
$\rho = (|t-s| /\vartheta)^{2/3}$ in the statement of Lemma \ref{lem:curvtime}.
Then further assuming  
\beq
\label{proof:UBC5}
 \left(\frac{|t-s|}\vartheta\right)^\frac23 < \frac{r_0}{2} \qquad \text{and} \qquad 
d_\H(\pa \Omega_t, \pa \Omega_s) \leq \frac1{16}\left(\frac{|t-s|}\vartheta\right)^\frac23
\eeq
implies via Lemma  \ref{lem:curvtime}  and \eqref{proof:UBC4}
\beq
\label{proof:UBC6}
\norm{\curv_s - \curv_t \circ \pi_t}_{L^\infty(\pa \Omega_s)} \leq  \frac{C(P_0,K_0)}{r_0^{8}}\left(\frac{|t-s|}\vartheta\right)^\frac13.
\eeq
Again, under the same assumption, we choose
$\rho = (|t-s| /\vartheta)^{1/2}$ in the statement of Lemma \ref{lem:normaltime}.
Then additionally assuming
\beq
\label{proof:UBC7}
 \left(\frac{|t-s|}\vartheta\right)^\frac12 < \frac{r_0}{2} \qquad \text{and} \qquad 
d_\H(\pa \Omega_t, \pa \Omega_s) \leq \frac1{8}\left(\frac{|t-s|}\vartheta\right)^\frac12
\eeq
implies via Lemma  \ref{lem:normaltime}  and \eqref{proof:UBC4}
\beq
\label{proof:UBC8}
\norm{\nu_s - \nu_t \circ \pi_t}_{L^\infty(\pa \Omega_s)} \leq  \frac{C(P_0,K_0)}{r_0^{7}}\left(\frac{|t-s|}\vartheta\right)^\frac12.
\eeq
We then choose
\beq
\label{proof:UBC9}
h = \min\left\{ 1, \frac{r_0}{2\theta}, \left(\frac{r_0}{2}\right)^\frac32,\frac1{(16\theta)^3}, 
\left(\frac{r_0}{2}\right)^2, \frac1{(8\theta)^2} \right\}
\eeq
Since $d_\H(\pa \Omega_t, \pa \Omega_s) \leq  \theta |t-s|/\vartheta$, then 
it follows from the choice of $h$ that
the assumption
$|t-s|/\vartheta < h$ guarantees the conditions 
\eqref{proof:UBC3}, \eqref{proof:UBC5} and \eqref{proof:UBC7} are valid
and, thus, the estimates\eqref{proof:UBC4}, \eqref{proof:UBC6} and \eqref{proof:UBC8} hold true.

\textbf{Step 2:} The case $|t-s|/\vartheta \geq h$. 
It follows from the definition \eqref{proof:UBC9} and the estimate \eqref{est:r_t} 
that
\[
\frac1h \leq  \frac{C(P_0,K_0)}{r_0^{21}}.
\]
and, hence, we estimate
\begin{align}
\label{proof:UBC10}
\begin{split}
r_t 
&\leq r_s + |r_t-r_s| \\
&\leq r_s + 2P_0 \\
&\leq r_s +\frac{2P_0}{h^\frac13} \left(\frac{|t-s|}\vartheta\right)^\frac13\\
&\leq r_s + \frac{C(P_0,K_0)}{r_0^{7}} \left(\frac{|t-s|}\vartheta\right)^\frac13.
\end{split}
\end{align}

\textbf{Step 3:} The case $|t-s|/\vartheta < h$. 
Recalling Proposition \ref{prop:UBC}  we have two alternatives for $r_s$.
The first possibility is $r_s = 1/ \norm{\curv_s}_{L^\infty(\pa \Omega_s)}$.
Thus using \eqref{proof:UBC6} and $
\norm{\curv_t}_{L^\infty(\pa \Omega_t)} \leq 1/r_t
$
we estimate
\begin{align}
\label{proof:UBC11}
\begin{split}
r_s 
&=\frac{1}{\norm{\curv_s}_{L^\infty(\pa \Omega_s)}} \\
&= \frac{1}{\norm{\curv_t}_{L^\infty(\pa \Omega_t)}}  + 
\frac{1}{\norm{\curv_s}_{L^\infty(\pa \Omega_s)}}-\frac{1}{\norm{\curv_t}_{L^\infty(\pa \Omega_t)}} \\
&\geq r_t + \frac{
\norm{\curv_t}_{L^\infty(\pa \Omega_t)}-\norm{\curv_s}_{L^\infty(\pa \Omega_s)} }{\norm{\curv_s}_{L^\infty(\pa \Omega_t)}\norm{\curv_s}_{L^\infty(\pa \Omega_s)}} \\
&\geq r_t + \frac{
\norm{\curv_t \circ \pi_t}_{L^\infty(\pa \Omega_s)}-\norm{\curv_s}_{L^\infty(\pa \Omega_s)} }{\norm{\curv_s}_{L^\infty(\pa \Omega_t)}\norm{\curv_s}_{L^\infty(\pa \Omega_s)}} \\
&\geq r_t - \frac{
\norm{\curv_t \circ \pi_t - \curv_s }_{L^\infty(\pa \Omega_s)} }{\norm{\curv_s}_{L^\infty(\pa \Omega_t)}\norm{\curv_s}_{L^\infty(\pa \Omega_s)}} \\
&\geq r_t - \frac{C(|\Omega_0|,P_0,K_0)}{r_0^{8}}\left(\frac{|t-s|}\vartheta\right)^\frac13.
\end{split}
\end{align}
Note that in the last estimate we also used the lower bound
\[
\frac{2\pi}{P_0} \leq \norm{\curv_t}_{L^\infty(\pa \Omega_t)}, \norm{\curv_s}_{L^\infty(\pa \Omega_s)}
\]
which can be seen valid by using the Gauss-Bonnet formula, the Jordan's curve theorem and the perimeter upper bound
$\P_t,P_s\leq P_0$.

In the second alternative, there are distinct points $x,y \in \pa \Omega_s$
satisfying $2r_s=|x-y|$ and $\langle \nu_s(x),\nu_s(y)\rangle=-1$.
Then by using \eqref{proof:UBC4}, Proposition \ref{prop:Hausdorff}, \eqref{est:Lipnu}, \eqref{proof:UBC8},  \eqref{est:r_t} and $h\leq1$
we estimate
\begin{align}
\label{proof:UBC12}
\begin{split}
r_s 
&= \frac{|x-y|}{2} \\
&\geq \frac{|\pi_t(x)-\pi_t(y)|}{2} - 2d_\H(\pa \Omega_t, \pa \Omega_s)  \\
&\geq \frac{|\pi_t(x)-\pi_t(y)|}{2} - \frac{C(P_0,K_0)}{r_0^{7}} \frac{|t-s|}\vartheta\\
&\geq \frac{|\n_t(\pi_t(x))-\n_t(\pi_t(y))|}{2}r_t - \frac{C(P_0,K_0)}{r_0^{7}} \frac{|t-s|}\vartheta \\
&\geq \frac{|\n_s(x) - \n_s(y)|}{2}r_t -  \frac{C(P_0,K_0)}{r_0^{7}}
\left(\frac{|t-s|}\vartheta+\left(\frac{|t-s|}\vartheta\right)^2\right) \\
&\geq r_t -
\frac{C(P_0,K_0)}{r_0^{7}} \left(\frac{|t-s|}\vartheta\right)^2.
\end{split}
\end{align}
We then infer \eqref{proof:UBC1} from \eqref{proof:UBC10}, \eqref{proof:UBC11}, \eqref{proof:UBC12},
 \eqref{est:r_t}  and $h\leq 1$.
\end{proof}

Combining the previous theorem with \eqref{est:time_lowerbound}
and \eqref{T1}
we conclude the following.

\begin{cor}
\label{cor:time_lowerbound}
Let \solT be a $C^3$-regular solution and let $\widetilde T_0$ be according to \eqref{maxT2}. 
There is a positive $T^*$ 
of the form
\beq
T^*= \vartheta c\left(P_0,\norm{\pa_{\tau_0}\curv_0}_{L^2(\pa \Omega_0)}\right) 
 \min \{ r_0^{17},r_0^{27}\}
\eeq
such that $\widetilde T_0 \geq \min\{T,T^*\}$.
\end{cor}

\subsection*{Data availability:} The manuscript contains no associated data.

\bibliographystyle{acm}
\bibliography{biblio}

\begin{thebibliography}{10}

\bibitem{abels2009existence}
{\sc Abels, H.}
\newblock Existence of weak solutions for a diffuse interface model for
  viscous, incompressible fluids with general densities.
\newblock {\em Communications in Mathematical Physics 289}, 1 (2009), 45--73.

\bibitem{abels2022non}
{\sc Abels, H.}
\newblock ({N}on-) convergence of solutions of the convective {A}llen--{C}ahn
  equation.
\newblock {\em Partial Differential Equations and Applications 3}, 1 (2022).

\bibitem{abels2023approximation}
{\sc Abels, H., Fischer, J., and Moser, M.}
\newblock Approximation of classical two-phase flows of viscous incompressible
  fluids by a {N}avier-{S}tokes/{A}llen-{C}ahn system.
\newblock {\em Archive for Rational Mechanics and Analysis 248}, 77 (2024).

\bibitem{abels2012thermodynamically}
{\sc Abels, H., Garcke, H., and Gr{\"u}n, G.}
\newblock Thermodynamically consistent, frame indifferent diffuse interface
  models for incompressible two-phase flows with different densities.
\newblock {\em Mathematical Models and Methods in Applied Sciences 22}, 3
  (2012), 1150013.

\bibitem{abels2014sharp}
{\sc Abels, H., and Lengeler, D.}
\newblock On sharp interface limits for diffuse interface models for two-phase
  flows.
\newblock {\em Interfaces and Free Boundaries 16}, 3 (2014), 395--418.

\bibitem{abels2018sharp}
{\sc Abels, H., and Liu, Y.}
\newblock Sharp interface limit for a {S}tokes/{A}llen--{C}ahn system.
\newblock {\em Archive for Rational Mechanics and Analysis 229}, 1 (2018),
  417--502.

\bibitem{AF}
{\sc Adams, R., and Fournier, J.}
\newblock {\em Sobolev spaces, 2nd Edition}.
\newblock Academic Press, 2003.

\bibitem{baldi2024liquid}
{\sc Baldi, P., Julin, V., and La~Manna, D.}
\newblock Liquid drop with capillarity and rotating traveling waves.
\newblock {\em arXiv preprint arXiv:2408.02333\/} (2024).

\bibitem{Bellettini}
{\sc Bellettini, G.}
\newblock {\em Lecture Notes on Mean Curvature Flow, Barriers and Singular
  Perturbations}.
\newblock Edizioni della Normale Pisa, 2013.

\bibitem{boulkhemairUniformPoincareInequality2007}
{\sc Boulkhemair, A., and Chakib, A.}
\newblock On the uniform {{Poincar\'e}} inequality.
\newblock {\em Communications in Partial Differential Equations 32}, 9 (2007),
  1439--1447.

\bibitem{denisova1994solvability}
{\sc Denisova, I., and Solonnikov, V.}
\newblock Solvability in {H}{\"o}lder spaces of a model initial-boundary value
  problem generated by a problem on the motion of two fluids.
\newblock {\em Journal of Mathematical Sciences 70}, 3 (1994), 1717--1746.

\bibitem{Duran}
{\sc Durán, R.}
\newblock An elementary proof of the continuity from {$L_0^2(\Omega)$} to
  {$H_0^1(\Omega)^n$} of {B}ogovskii’s right inverse of the divergence.
\newblock {\em Revista de la Unión Matemática Argentina 53}, 2 (2012),
  59--78.

\bibitem{FedererCurvature}
{\sc Federer, H.}
\newblock Curvature measures.
\newblock {\em Transactions of the American Mathematical Society 93\/} (1959),
  418--491.

\bibitem{fischer2020weak}
{\sc Fischer, J., and Hensel, S.}
\newblock Weak--strong uniqueness for the {N}avier--{S}tokes equation for two
  fluids with surface tension.
\newblock {\em Archive for Rational Mechanics and Analysis 236}, 2 (2020),
  967--1087.

\bibitem{Fu}
{\sc Fu, J.}
\newblock Tubular neighborhoods in {E}uclidean spaces.
\newblock {\em Duke Mathematical Journal 54}, 4 (1985), 1025--1046.

\bibitem{Galdi}
{\sc Galdi, G.}
\newblock {\em An Introduction to the Mathematical Theory of the
  {N}avier-{S}tokes Equations}.
\newblock Springer, 2011.

\bibitem{hensel2023sharp}
{\sc Hensel, S., and Liu, Y.}
\newblock The sharp interface limit of a {N}avier--{S}tokes/{A}llen--{C}ahn
  system with constant mobility: Convergence rates by a relative energy
  approach.
\newblock {\em SIAM Journal on Mathematical Analysis 55}, 5 (2023), 4751--4787.

\bibitem{julin2024flat}
{\sc Julin, V.}
\newblock Flat flow solution to the mean curvature flow with volume constraint.
\newblock {\em Advances in Calculus of Variations 17}, 4 (2024), 1543--1555.

\bibitem{julin2024priori}
{\sc Julin, V., and La~Manna, D.}
\newblock A priori estimates for the motion of charged liquid drop: A dynamic
  approach via free boundary {E}uler equations.
\newblock {\em Journal of Mathematical Fluid Mechanics 26}, 48 (2024).

\bibitem{JN2023}
{\sc Julin, V., and Niinikoski, J.}
\newblock Consistency of the flat flow solution to the volume preserving mean
  curvature flow.
\newblock {\em Archive for Rational Mechanics and Analysis 248}, 1 (2024).

\bibitem{KoOl}
{\sc Kondratiev, V., and Oleinik, O.}
\newblock On {K}orn’s inequalities.
\newblock {\em Comptes rendus hebdomadaires des séances de l'Académie des
  Sciences 308}, 16 (1989), 483--487.

\bibitem{padula2010local}
{\sc Padula, M., and Solonnikov, V.}
\newblock On the local solvability of free boundary problem for the
  {N}avier--{S}tokes equations.
\newblock {\em Journal of Mathematical Sciences 170}, 4 (2010), 522--553.

\bibitem{pruss2011analytic}
{\sc Pr{\"u}ss, J., and Simonett, G.}
\newblock {\em Analytic solutions for the two-phase {N}avier-{S}tokes equations
  with surface tension and gravity}.
\newblock Springer, 2011.
\newblock In: Escher, J., \emph{et al.} Parabolic Problems. Progress in
  Nonlinear Differential Equations and Their Applications, vol 80.

\bibitem{pruss2016moving}
{\sc Pr{\"u}ss, J., and Simonett, G.}
\newblock {\em Moving interfaces and quasilinear parabolic evolution
  equations}, vol.~105.
\newblock Springer, 2016.

\bibitem{Reilly}
{\sc Reilly, R.}
\newblock Applications of the {H}essian operator in a {R}iemannian manifold.
\newblock {\em Indiana University Mathematics Journal 26}, 3 (1977), 459--472.

\bibitem{solonnikov1977solvability}
{\sc Solonnikov, V.}
\newblock Solvability of a problem on the motion of a viscous incompressible
  fluid bounded by a free surface.
\newblock {\em Mathematics of the USSR-Izvestiya 11}, 6 (1977), 1323--1358.

\bibitem{solonnikov1986unsteady}
{\sc Solonnikov, V.}
\newblock On an unsteady flow of a finite mass of a liquid bounded by a free
  surface.
\newblock {\em Zapiski Nauchnykh Seminarov LOMI 152\/} (1986), 137--157.

\bibitem{solonnikov1991initial}
{\sc Solonnikov, V.}
\newblock An initial boundary value problem for the {S}tokes systems arising in
  the study of a problem with a free boundary.
\newblock vol.~3 of {\em Trudy Matematicheskogo Instituta imeni V.A. Steklova},
  Nauka, pp.~191--239.

\bibitem{solonnikov2002estimates}
{\sc Solonnikov, V.}
\newblock Estimates of solutions of the second initial boundary-value problem
  for the {S}tokes system in the spaces of functions having {H}{\"o}lder
  continuous derivatives with respect to spatial variables.
\newblock {\em Journal of Mathematical Sciences 109}, 5 (2002), 1997--2017.

\end{thebibliography}


\appendix


\section{Appendix. Useful results about geometry of boundaries}


\subsection{Diffeomorphisms between boundaries} \label{sec:diffeom}

Let  $\Omega$ and $\Omega'$ be bounded $C^k$-regular domains.
We say that a $C^l$-regular map $\Phi: \pa \Omega \rightarrow \pa \Omega'$ for  $1 \leq l \leq k$ 
is a  $C^l$-diffeomorphism,
if $\Phi^{-1} : \pa \Omega \rightarrow \pa \Omega'$ exists as a $C^l$-regular map.  This is valid
exactly when $\Phi$ is bijective and $\pa_{\tau} \Phi \neq 0$ on $\pa \Omega$. The change of variable formula then
simply reads as
\beq
\label{changeofvar}
\int_{\pa \widetilde\Omega} f \d \H^1 = \int_{\pa \Omega} f \circ \Phi \, |\pa_{\tau} \Phi| \d \H^1
\eeq
for every $\H^1$-integrable map $f:\pa  \Omega' \rightarrow \R^m$.

We are solely interested in orientation preserving diffeomorphisms between boundaries.
Using the introduced notations
we can characterize this  condition for $\Phi$ as $\langle \pa_\tau \Phi ,  \tau' \circ \Phi\rangle >0$ on $\pa \Omega$. In this case, we can consistently write
\beq
\label{transformed_normal_and_tangent}
\nu' \circ \Phi = \frac{- \langle \pa_\tau \Phi, \nu \rangle \tau + \div_\tau \Phi \, \nu}{|\pa_\tau \Phi|} \qquad \text{and} \qquad
\tau' \circ \Phi = \frac{\pa_\tau \Phi}{|\pa_\tau \Phi|}.
\eeq
Again using local extensions, chain rule and the latter identity in \eqref{transformed_normal_and_tangent}
we can represent tangential derivative along $\pa \Omega'$ with respect to $\pa \Omega$ and $\Phi$
as follows
\beq
\label{transformed_derivative}
(\pa_{\tau'} f) \circ \Phi = \frac{\pa_\tau(f \circ \Phi)}{|\pa_\tau \Phi|}
\eeq
for every $f \in C^1(\pa \Omega';\R^m)$. Further, using \eqref{transformed_normal_and_tangent}, \eqref{transformed_derivative}
and Weingarten identities \eqref{weingarten}
we can represent curvature and its tangential derivative on $\pa \Omega'$ as follows 
\beq
\label{prop:transformed_curv1}
\curv' \circ \Phi = \frac{\langle \pa_\tau^2 \Phi, \tau \rangle \langle \pa_\tau \Phi, \nu\rangle -  \langle \pa_\tau^2 \Phi, \nu \rangle  \div_\tau \Phi}{|\pa_\tau \Phi|^3}
\eeq
provided that $k \geq 2$ and
\begin{align}
\label{prop:transformed_curv2}
\begin{split}
(\pa_{\tau'} \curv') \circ \Phi
= &\frac{\langle \pa_\tau^3 \Phi, \tau \rangle \langle \pa_\tau \Phi, \nu\rangle -  \langle \pa_\tau^3 \Phi, \nu \rangle  \div_\tau \Phi}{|\pa_\tau \Phi|^4} \\
&\,- 3 \frac{\langle \pa_\tau^2 \Phi, \tau \rangle \langle \pa_\tau \Phi, \nu\rangle -  \langle \pa_\tau^2 \Phi, \nu \rangle  \div_\tau \Phi}{|\pa_\tau \Phi|^6}\langle \pa_\tau^2 \Phi,\pa_\tau \Phi\rangle
\end{split}
\end{align}
provided that $k \geq 3$.


\subsection{Identities related to geometric evolution}
In this subsection we consider a $C^k$-regular solution \solT (recall $k \geq 2$) and 
derive expressions for certain time derivatives of the form
\[
\frac{\d}{\d t} \int_{\pa \Omega_t} Q \d \H^1
\]
where $Q$ is a (possibly) time-dependent $\R^m$-valued function defined in $\{(t,x) : t \in [0,T), \, x \in \pa \Omega_t\}$.

While the required identities can also be derived by abstract arguments, the most convenient approach is to use local time-parametrizations. Note that the Picard-Lindelof theorem directly implies the following:

\begin{prop}
\label{prop:localtimeparam}
Let \solT be a $C^k$-regular solution. Then for every time $t \in [0,T)$ we find a relatively open subinterval $I$ containing $t$
and a parametrization $\Phi \in C^1(I;C^k(\pa \Omega_t; \R^2))$ 
such that $\Phi(s,\freearg)$ is an orientation preserving diffeomorphism between $\pa \Omega_t$
and $\pa \Omega_s$ for every $s \in I$ with $\Phi(t,\freearg) = \id$.
\end{prop}

Let us then fix $t \in [0,T)$ and let $I$ and $\Phi$ as in the previous proposition.
We denote $X_t = \pa_1 \Phi(t,\freearg)$.\footnote{Note that a straightforward proof of Proposition \ref{prop:localtimeparam} using $u$ as a vector field would result in $X_t = u(t,\cdot)$. However what follows is independent of that observation.} Then $X_t \in C^k(\pa \Omega_t;\R^2)$
and it is rather straightforward to check

\begin{equation}
\label{critical_flowmap_convergence}
\frac{\pa_{\tau_t}^l \Phi(t+h,\freearg) - \pa_{\tau_t}^l \Phi(t,\freearg) }{h} \longrightarrow \pa_{\tau_t}^l X_t
\end{equation}
uniformly on $\pa \Omega_t$ as $h$ tends to zero for every $l=0,1,\ldots,k$.
Since $\Phi(t,\freearg)=\id$, then by applying the implicit function theorem on Definition
\ref{def:normalvelocity} for the normal velocity
we conclude $\langle X_t, \n_t \rangle = v(t, \freearg)$. Thus the normal velocity condition in  Definition \ref{def:classicsol} implies 
the normal components must agree, i.e.
\beq
\label{normalparts}
\langle X_t, \n_t\rangle = \langle u, \n_t \rangle
\qquad \text{on} \ \ \pa \Omega_t.
\eeq

We assume that $Q$ is sufficiently regular at time $t$. To be more precise,
the following condition is expected to hold
\beq
\label{Qcondition}
\lim_{h \rightarrow 0} \int_{\pa \Omega_t} \left|\frac{Q(t+h,\Phi(t+h,\freearg) ) - Q(t,\freearg)}h
- \frac{\pa}{\pa s} Q\left(s,\Phi(s,\freearg)\right) \bigg |_{s=t} \right| \d \H^1 = 0.
\eeq

Using the change of variables and adding and subtracting terms we have

\begin{align*}
&\frac{1}{h} \left( \int_{\pa \Omega_{t+h}} Q(t+h,\freearg) \d \H^1- \int_{\pa \Omega_t}Q(t,\freearg)\d \H^1\right) \\
=  &\int_{\pa \Omega_t} |\pa_{\tau_t} \Phi(t+h,\freearg))| \frac{Q(t+h,\Phi(t+h,\freearg) ) - Q(t,\freearg)}h  \d \H^1 \\
+&\int_{\pa \Omega_t} \frac{|\pa_{\tau_t} \Phi(t+h,\freearg))|-
1 }{h}Q(t,\freearg) \d \H^1.
\end{align*}
It follows from the uniform convergence $\pa_{\tau_{t+h}} \Phi (t+h,\freearg) \rightarrow \tau_t$
on $\pa \Omega_t$
 and the condition \eqref{Qcondition} on $Q$
that the first integral in RHS converges to 
\[
 \int_{\pa \Omega_t} 
\frac{\pa}{\pa s} Q\left(s,\Phi(s,\freearg)\right) \bigg |_{s=t} \d \H^1
\]
as $h \rightarrow 0$. The latter transport term caused by the tangential Jacobian $|\pa_{\tau_t} \Phi(t+h,\freearg)|$ is also easy to evaluate.
Indeed, since $\pa_{\tau_t} \Phi (t, \freearg) = \pa_{\tau_t} \id = \tau_t$, then
 recalling \eqref{critical_flowmap_convergence}
we obtain
\[
\lim_{h \rightarrow 0} \int_{\pa \Omega_t}  \frac{|\pa_{\tau_t} \Phi(t+h,\freearg))|-1 }{h}Q(t+h,\freearg) \d \H^1 =
 \int_{\pa \Omega_t}  \div_{\tau_t} X_t \, Q(t,\freearg)\d \H^1.
\]
Thus we conclude
\begin{align}
\label{ratestructure}
\begin{split}
\frac{\d}{\d s} &\int_{\pa \Omega_s} Q(s,\freearg) \d \H^1\bigg|_{s=t}
 \\
= &\int_{\pa \Omega_t} \frac{\pa}{\pa s} Q\left(s,\Phi(s,\freearg)\right) \bigg |_{s=t}
+\div_{\tau_t} X_t \,  Q(t,\freearg)  \d \H^1.
\end{split}
\end{align}

If $Q$ is a regular ambient function, i.e., $Q \in C^1([0,T)\times \R^2;\R^m)$
the condition \eqref{Qcondition} is satisfied and 
using the surface divergence theorem \eqref{eq:surfacedivthm} and 
the identity \eqref{normalparts}
we have
\begin{align}
\label{ratestructure_ambient}
\begin{split}
\frac{\d}{\d s} &\int_{\pa \Omega_s}  Q(s,\freearg)  \d \H^1\bigg|_{s=t}
 \\
= &\int_{\pa \Omega_t} \pa_1 Q (t,\freearg) + \langle u, \n_t \rangle \left( \pa_{\n_t} 
Q(t,\freearg)+ \curv_t 
Q(t,\freearg) \right) \d \H^1.
\end{split}
\end{align}
By choosing $Q=1$
we obtain the first variation of perimeter
\begin{equation}
\label{dP/dt}
\frac{\d}{\d t} P(\Omega_t) =  \int_{\pa \Omega_t} \curv_t \langle u, \n_t \rangle \d \H^1.
\end{equation}
We then consider the following cases involving geometric quantities:
\begin{itemize}
\item[(i)] $Q(s,x) = \phi(x) \n_s(x)$ and $Q(s,x) = \phi(x) \curv_{\tau_s}(x)$  for $\phi \in C^1(\R^2)$,
\item[(ii)]  $Q(s,x) =\curv_{\tau_s}^2(x)$ and
\item[(iii)]  $Q(s,x) =(\pa_{\tau_s}\curv_{\tau_s})^2(x)$ (provided that $k \geq 3$).
\end{itemize}

To compute the rates of the corresponding energies it is more feasible to use
the expressions derived in the previous section than try to construct suitable extensions
for $Q$. These quantities are in reach of our scheme.  Indeed, we note it follows from \eqref{transformed_normal_and_tangent},  
\eqref{prop:transformed_curv1}, \eqref{prop:transformed_curv2} and the properties of $\Phi$ 
(in particular \eqref{critical_flowmap_convergence}) that the difference quotients
\[
\frac{\n_{t+h} \circ \Phi (t+h, \freearg) - \n_t}h, \ \ \frac{\curv_{t+h} \circ \Phi (t+h, \freearg)-\curv_t}h \ \ \text{and} 
\ \  \frac{\pa_{\tau_{t+h}}\curv_{t+h} \circ \Phi (t+h, \freearg) - \pa_{\tau_t} \curv_t}h \ \ \text{(if $k\geq 3$)}
\]
converge uniformly to the corresponding derivatives on $\pa \Omega_t$ as $h \rightarrow 0$. This ensures compatibility of the quantity $Q$ in (i) -- (iv) with the differentiability condition \eqref{Qcondition}.
Again, using \eqref{transformed_normal_and_tangent}, \eqref{prop:transformed_curv1}, \eqref{prop:transformed_curv2},  \eqref{critical_flowmap_convergence}, $\Phi(t, \freearg) = \id$ and Weingarten identies 
we compute
\begin{align*}
\frac{\pa}{\pa s} \left( \n_s \circ \Phi(s,\freearg) \right)  \bigg|_{s=t}&=- \langle \pa_{\tau_t} X_t, \n_t \rangle \tau_t \\
\frac{\pa}{\pa s} \left(\curv_{\tau_s} \circ \Phi(s,\freearg)\right) \bigg|_{s=t}&= - \langle \pa_{\tau_t}^2 X_t, \n_t\rangle  - 2 \curv_t \div_{\tau_t} X_t \ \ \text{and (if $k\geq 3$)} \\
\frac{\pa}{\pa s} \left(\pa_{\tau_s}\curv_{\tau_s} \circ \Phi(s,\freearg)\right) \bigg|_{s=t}&= - \pa_{\tau_t}(\langle\pa_{\tau_t}^2 X_t, \n_t \rangle + \curv_t \div_{\tau_t} X_t)
-\curv_t \pa_{\tau_t} \div_{\tau_t} X_t
-2 \pa_{\tau_t} \curv_t\div_{\tau_t} X_t.
\end{align*}
Then substitution into \eqref{ratestructure} eventually yields
\begin{align}
\label{daux1/dt}
\frac{\d}{\d t}\int_{\pa \Omega_t}  \phi \n_t \d \H^1  &= \int_{\pa \Omega_t} \langle X_t, \n_t \rangle \nabla \phi \d \H^1, \\
\label{daux2/dt}
\frac{\d}{\d t} \int_{\pa \Omega_t} \phi \curv_t \d \H^1 &= \int_{\pa \Omega_t} \pa_{\tau_t} \phi \langle \pa_{\tau_t} X_t, \n_t\rangle  
+ \curv_t \langle \nabla \phi, X_t \rangle  \d \H^1, \\
\label{eq:k2derivative}
\frac{\d}{\d t} \int_{\pa \Omega_t} \curv_t^2 \d \H^1&= - \int_{\pa \Omega_t} 2 \curv_t ( \langle \pa_{\tau_t}^2 X_t, \nu_t \rangle + \curv_t \div_{\tau_t} X_t)  + \curv_t^2
\div_{\tau_t} X_t \d \H^1 
\end{align}
and if $k \geq 3$
\begin{align} \label{eq:dk2derivative}
\begin{split}
\frac{\d}{\d t}& \int_{\pa \Omega_t} (\pa_t \curv_t)^2 \d \H^1 \\
&=- \int_{\pa \Omega_t} 2 \pa_{\tau_t} \curv_t  \pa_{\tau_t}(\langle\pa_{\tau_t}^2 X_t, \n_t \rangle + \curv_t \div_{\tau_t} X_t) + 2 \curv_t \pa_{\tau_t} \curv_t \pa_{\tau_t} \div_{\tau_t} X_t \d \H^1 \\
&\, - 3\int_{\pa \Omega_t} (\pa_{\tau_t} \curv_t)^2 \div_{\tau_t}  X_t \d \H^1.
\end{split}
\end{align}
Finally by using the condition \eqref{normalparts} we may replace $X_t$ with $u$ on the right hand
side of the identities  \eqref{daux1/dt}, \eqref{daux2/dt}, \eqref{eq:k2derivative}
and \eqref{eq:dk2derivative}. For \eqref{daux1/dt} and \eqref{daux2/dt} this simply follows
from integration by parts. 

For \eqref{eq:k2derivative} we observe that 
$\n_t$ and  $\tau_t$ are locally $C^k$-regular and $\curv_\tau$ is locally $C^{k-1}$-regular 
in the relatively open subset 
$\Gamma_t := \{ x \in \pa \Omega_t : \langle X_t(x) - u(x),\tau_t(x)\rangle \neq 0 \}$ provided that it is not empty. Suppose that $\Gamma_t \neq \varnothing$. Since its relative boundary is 
$\H^1$-negligible and shared with the complement  $\pa \Omega_t \cap \{ X_t = u\}$, 
we conclude for $Z=X_t - u$
\begin{align*}
\int_{\pa \Omega_t} &2 \curv_t  (\langle \pa_{\tau_t}^2 Z, \nu_t \rangle
+\curv_t \div_{\tau_t}Z )  + \curv_t^2
\div_{\tau_t} Z \d \H^1 \\
&=
\int_{\Gamma_t} 2 \curv_t  (\langle \pa_{\tau_t}^2 Z, \nu_t \rangle
+\curv_t \div_{\tau_t}Z )  + \curv_t^2
\div_{\tau_t} Z \d \H^1.
\end{align*}
Since $Z$ does not have the normal component in $\Gamma_t$, then
using Weingarten identities yields
\[
2 \curv_t  (\langle \pa_{\tau_t}^2 Z, \nu_t \rangle
+\curv_t \div_{\tau_t}Z )  + \curv_t^2
\div_{\tau_t} Z = -\pa_{\tau_t} (\curv_t^2 \langle Z, \tau_t \rangle)
\ \ \ \text{in} \  \ \Gamma_t.
\]
This, in turn, together with $Z=0$ in $\pa \Omega_t \cap \{ X_t = u\}$
allows us to conclude that the previous integral must equal to zero.

Again, we may argue similarly for the identity  \eqref{eq:dk2derivative}.


\subsection{Extended normal and tangent vector}
\label{app:ExtNormal}

In this section we will discuss some additional properties and identities regarding the extended normal and tangent vectors.
First of all, assuming regularity of $\partial \Omega$, we can also recover higher regularity for the signed distance function and the projection. For the sake of brevity we write $d=d_\Omega$ and $\pi=\pi_{\pa \Omega}$. 

\begin{prop}
\label{prop:ImprovedReg}
Let $\Omega \subset \R^2$ be a bounded $C^k$-regular domain with $k \geq 2$. 
Then the signed distance $d \in C^k(\mathcal N(\pa \Omega))$ and the distance projection $\pi \in C^{k-1}(\mathcal N(\pa \Omega);\R^2)$
satisfy the following identities in $\mathcal N(\pa \Omega)$
\beq
\label{eq:sgndistance2}
\nabla^2 d = \frac{\curv \circ \pi}{1+d \curv \circ \pi}   \left(I- \nabla d \otimes \nabla d\right)
 \ \ \text{and}  \ \ \nabla \pi= \frac{1}{1+d \curv \circ \pi}  \left(I- \nabla d \otimes \nabla d\right).
\eeq 
\end{prop}
\bP
Fix an arbitrary $z \in \mathcal N(\pa \Omega)$, choose a local $C^k$-parametrization $\gamma : I \rightarrow \pa \Omega$ at $\pi(z)$ and define  
$F : I \times (-r_\Omega,r_\Omega) \rightarrow \mathcal N (\pa \Omega)$
by setting $F(s,\rho) = \gamma(s) + \rho \nu \circ \gamma (s)$. Then $F$ is an injection
with $F^{-1} = (\gamma^{-1} \circ \pi, d_\Omega)$ in its image which, in particular, contains $z$.
Clearly $F \in C^{k-1}( I \times (-r_\Omega,r_\Omega); \R^2)$ pointwise
satisfying
\begin{align*}
\pa_1 F (s,\rho) &= \left(1+ \rho \curv \circ \gamma(s)\right) \langle \gamma'(s),\tau \circ \gamma(s) \rangle\tau \circ \gamma(s) \ \ \text{and} \\
\pa_2 F (s,\rho)&= \nu \circ \gamma(s). 
\end{align*}
Since $|\curv| \leq 1/r_\Omega$, then $ \left(1+ \rho \curv \circ \gamma(s)\right)\neq 0$ and we conclude that $\nabla F (s,\rho)$ is invertible for every 
$(s,\rho) \in I \times (-r_\Omega,r_\Omega)$. Thus it follows from the inverse function theorem that $F^{-1}$ is locally $C^{k-1}$-regular
in the open set $F(I \times (-r_\Omega,r_\Omega))$ further implying $\pi$ is locally 
$C^{k-1}$-regular at $z$. Therefore we conclude 
$\pi \in C^{k-1}_\loc( \mathcal N(\pa \Omega); \R^2)$ which, in turn, implies via the second identity in \eqref{eq:sgndistance1} that
$d \in C^k_\loc( \mathcal N( \pa \Omega))$.

The second identity in \eqref{eq:sgndistance1} gives  $|\nabla d|^2=1$ in $\mathcal N(\pa \Omega)$ so differentiating this and using the fact that $\nabla^2 d$ is
a symmetric matrix field yields in the two dimensional case
\beq
\label{prop:ImprovedReg1}
\nabla^2 d = \Delta d \, (I-\nabla d \otimes \nabla d) \ \ \text{in} \ \ \mathcal N (\pa\Omega).
\eeq
Again the second identity in \eqref{eq:sgndistance1} states that $\nabla d$ is a locally $C^1$-regular extension
of $\nu$ so we immediately infer from the previous expression for $\nabla^2 d$ that $\curv = \Delta d$ on $\pa \Omega$.
Next we differentiate the identities \eqref{eq:sgndistance1}
to obtain
\[
I  = \nabla \pi + d \nabla^2 d + \nabla d \otimes \nabla d
\ \ \text{and} \ \ 
\nabla^2 d = (\nabla^2 d \circ \pi) \nabla \pi
\]
and combining these with the second identity in \eqref{eq:sgndistance1}, \eqref{prop:ImprovedReg1} and $\Delta d \circ \pi= \curv \circ \pi$ we further obtain
\begin{align*}
 \nabla^2 d &= \curv \circ \pi (I-\nabla d \otimes \nabla d) \nabla \pi \ \ \text{and}\\
(I + d \curv \circ \pi (I-\nabla d \otimes \nabla d))\nabla \pi &= I-\nabla d \otimes \nabla d.
\end{align*}
Since $|d \curv \circ \pi|<1$, then
\[
(I + d \curv \circ \pi (I-\nabla d \otimes \nabla d))^{-1} = \nabla d \otimes \nabla d+ \frac1{1+d \curv \circ \pi} (I- \nabla d \otimes \nabla d)
\]
and we conclude the identities \eqref{eq:sgndistance2}.
\eP
Recall that we denote the normal and tangent extensions simply by $\nu$ and $\tau$ respectively.
Since we may write $\nu = \nabla d = \nu \circ \pi$ and $\tau = \tau \circ \pi$, we immediately recover the following
from the previous proposition:
\begin{prop}[Basic identities for extended normal and tangent vector]
Let $\Omega$ be a bounded $C^k$-regular domain in $\R^2$ with $k \geq 2$, 
then $\nu,\tau \in C^{k-1}(\mathcal N(\pa \Omega);\R^2)$
satisfies the following identities in $\mathcal N(\pa \Omega)$:
\begin{align}
\label{basicid} 
\nabla \nu = \div \nu \, \tau \otimes \tau  \qquad  \text{and} \qquad
\nabla \tau = - \div \nu \, \nu \otimes \tau
\\
\label{divn2D}
\div \, \nu = \frac{\curv \circ \pi}{1+ d \curv \circ \pi}
\\
\label{pa_n H}
\pa_\n \div \, \n = -(\div \, \n)^2.
\end{align}
\end{prop}
The identities \eqref{basicid}  can be seen as generalized Weingarten identities for the extensions.

In the case $\Omega$ is $C^3$-regular we have the estimate  
\beq
\label{curv_bulk_L2est}
\int_{\mathcal N_{r_\Omega/2}(\pa \Omega)} (\pa_\tau \div \, \n)^2 \d x \leq \frac{3r_\Omega}2\int_{\pa \Omega} (\pa_\tau \curv)^2 \d \H^1
\eeq
\begin{proof}[Proof of \eqref{curv_bulk_L2est}] The second identity in \eqref{eq:sgndistance2} gives us
\[
\pa_\tau \pi = \frac{\tau \circ \pi}{1-d \curv \circ \pi}
\]
and again using the invariance $\pi = \pi \circ \pi$ we infer from the previous
\[
\pa_\tau (\curv \circ \pi) = \nabla (\curv \circ \pi) \pa_\tau \pi =\frac{(\pa_\tau \curv) \circ \pi}{1-d \curv \circ \pi}.
\]
Thus differentiating \eqref{divn2D} in the direction $\tau$ yields
\[
(\pa_\tau \div \, \nu)^2 = \frac{(\pa_\tau \curv)^2 \circ \pi }{(1+d \curv \circ \pi)^4} =  \frac{(\pa_\tau \curv)^2 \circ \pi |\pa_\tau \pi|}{(1+d \curv \circ \pi)^3}
\]
Since $\pi : \pa \{d < s\}=\{d=s\} \rightarrow \pa \Omega$ is an orientation preserving diffeomorphism
and $|\nabla d|=1$ in $\mathcal N(\pa \Omega)$, 
then using the coarea formula, the previous expression as well as the change of variables we obtain
\begin{align*}
\int_{\mathcal N_{r_\Omega/2}(\pa \Omega)} (\pa_\tau \div \, \n)^2 \d x &=\int_{-r_\Omega/2}^0 \int_{\pa \Omega} \frac{(\pa_\tau \curv)^2 }{(1+s \curv)^3}  \d \H^1  \d s
\end{align*}
and \eqref{curv_bulk_L2est} follows.
\end{proof}

For the rest of discussion of this subsection we assume that $\Omega$
is a bounded $C^2$-regular  subset of $\R^2$ unless otherwise stated. 
We continue with a trivial observation.
If $\phi$ is a scalar field differentiable (a.e.) in some open subset  $U \subset \mathcal N(\pa \Omega)$
we have
\beq
\label{div(f*n)0}
\div \left( \phi \, \nu \right) =   \pa_\nu \phi +  \div \, \nu \, \phi \ \ \text{(a.e.) in} \ \ U.
\eeq
This simple identity is a rather useful when coupled with a suitable cutoff function and the divergence theorem
since it allows us to transform boundary integrals into bulk integrals. Suppose that $\phi \in C^{0,1}
(\mathcal N_{r_\Omega/2} (\pa \Omega) \cap \Omega)$ and let $\eta$ be the cutoff for $\Omega$ according to \eqref{eta}. Then applying the divergence theorem and consequently
\eqref{div(f*n)0} on  $\eta \phi   \nu$ results in
\beq
\label{div(f*n)}
\int_{\pa \Omega}\phi \d \H^1 = \int_{\pa \Omega}\eta  \phi \d \H^1 = \int_\Omega \phi \pa_\nu \eta +\eta \pa_\nu \phi + \eta \, \phi   \div \, \nu  \d x
\eeq
and we have the estimates $|\pa_\nu \eta| \leq 2/r_\Omega$, $|\div \, \nu | \leq 2/r_\Omega$ in $\spt \eta$. Again, recalling \eqref{pa_n H} we may formally iterate \eqref{div(f*n)0} as $\div \left( \div \, \nu \, \phi \, \nu \right) =   \div \, \nu   \pa_\nu \phi$. Hence we conclude
\beq
\label{div(divn*f*n)}
\int_{\pa \Omega} \curv \phi \d \H^1 = \int_{\Omega}  \eta \, \pa_\nu \phi \,  \div \, \nu  +  \phi  \, \pa_\nu \, \eta \div \, \nu  \d x
\eeq
for every $\phi \in C^{0,1}
(\mathcal N_{r_\Omega/2} (\pa \Omega) \cap \Omega)$.
We may bypass the formal step first  approximating  $\div \, \nu$ by
\[
\frac{(\div \, \nu)_\eps \circ \pi}{1+d \, (\div \, \nu)_\eps \circ \pi} 
\]
in $\spt \eta \subset  \mathcal N_{r_\Omega/2}(\pa \Omega)$ where $(\div \, \nu)_\eps$ is a standard $\eps$-mollification
of $\div \, \nu$
with sufficiently small $\eps$, then applying \eqref{div(f*n)} and finally letting $\eps \rightarrow 0$.

Since we do a lot of second and third order calculation
with respect to the vector frame $\{\nu,\tau\}$ in tubular neighborhoods
we list some frequently used basic identities. 
\begin{lemma}[Higher order identities]
Let $X$ be a $C^2$-regular vector field in an open subset $U \subset \mathcal N(\pa \Omega)$. Then the following identities hold
in $U$
 \begin{align}
\label{changeoforder}
\pa_\tau \pa_\n X &= \pa_\nu  \pa_\tau X  + \div \, \n \, \pa_\tau X \\
\label{Delta X}
\Delta X &= \pa_\nu^2 X + \pa_\tau^2 X + \div \, \nu \, \pa_\nu X \\
\label{|nabla^2X|^2}
|\nabla^2  X|^2 &= |\pa_\nu^2 X|^2 + 2|\pa_\nu \pa_\tau  X|^2 +|\pa_\tau^2 X + \div \nu \, \pa_\nu X|^2.
\end{align}
Moreover, if $\Omega$ and $X$ are $C^3$-regular in the respective manner, then
\begin{align}
\begin{split}
|\nabla^3 X|^2&=|\pa_\n^3 X|^2 + 2 |\pa_\n^2\pa_\tau X|^2 +|\pa_\n\pa_\tau^2 X+ \div  \n \, \pa_\n^2 X - (\div  \n)^2\pa_\n X |^2  \\
&+|\pa_\tau \pa_\n^2 X - 2  \div  \n \, \pa_\n \pa_\tau X|^2 + 2|\pa_\n\pa_\tau^2 X + \div  \n \, \pa_\n^2 X- (\div  \n)^2 \pa_\n X |^2  \\
&+|\pa_\tau^3 X +3  \div  \n \, \pa_\n\pa_\tau X +  \pa_\tau \div  \n \, \pa_\n X+ (\div  \n)^2\pa_\tau  X |^2.
\label{|nabla^3X|^2}
\end{split}
\end{align}
\end{lemma}

\begin{proof}
The first identity is a direct consequence of the Weingarten identities \eqref{basicid}
as we compute
\begin{align*}
\pa_\tau \pa_\nu X = [\pa_\tau \nabla X]\nu + [\nabla X] \pa_\tau \nu = 
 [\pa_\nu \nabla X]\tau + \div \, \nu \, \pa_\tau X = \pa_\nu\pa_\tau X + \div \, \nu \, \pa_\tau X.  
\end{align*}
Since $\{\n,\tau\}$ is an orthonormal frame, we have $\nabla^k X = \pa_\n \nabla^{k-1} X \otimes \n + \pa_\tau \nabla^{k-1} X \otimes \tau$
 for any $k \in \N$ provided that $X$ is $C^k$-regular. Then
\begin{align}
\label{nabla^2X}
\begin{split}
\nabla^2 X 
=  & \pa_\n \nabla X \otimes \n + \pa_\tau \nabla X \otimes\n\\
=   &\pa_\n (\pa_\n X \otimes \n+\pa_\tau X \otimes \tau) \otimes \n + \pa_\tau(\pa_\n X \otimes \n+\pa_\tau X \otimes \tau) \otimes\tau \\
\overset{\eqref{basicid}}{=} &(\pa_\n^2 X \otimes \n_\Omega + \pa_\n\pa_\tau X \otimes \tau) \otimes \n\\
+& (\pa_\tau\pa_\n X \otimes \n + \div \, \n \, \pa_\n X \otimes \tau + \pa_\tau^2 X \otimes \tau
- \div \, \n \, \pa_\tau X \otimes \n ) \otimes \tau \\
\overset{\eqref{changeoforder}}{=}  &(\pa_\n^2 X \otimes \n + \pa_\n\pa_\tau X \otimes \tau) \otimes \n \\
+ &\left(\pa_\n\pa_\tau X \otimes \n +  (\pa_\tau^2 X+\div \, \n \, \pa_\n X) \otimes \tau
 \right ) \otimes \tau.
\end{split}
\end{align} 
We then infer \eqref{Delta X} and \eqref{|nabla^2X|^2} from this representation. 
For the last identity \eqref{|nabla^3X|^2} we compute $|\pa_\n \nabla^2 X|^2$ and $|\pa_\tau \nabla^2 X|^2$ separately.
The first one is given by
\begin{align*}
\begin{split}
\pa_\n  \nabla^2 X 
\overset{\eqref{nabla^2X}}{=} & \pa_\n\left( (\pa_\n^2 X \otimes \n + \pa_\n \pa_\tau X \otimes \tau) \otimes \n\right) \\
 + &\pa_\n\left(\left(\pa_\n\pa_\tau X \otimes \n +  (\pa_\tau^2 X+\div \, \n \, \pa_\n X) \otimes \tau
 \right ) \otimes \tau\right)  \\
\overset{\eqref{basicid}}{=} &    (\pa_\n^3 X \otimes \n + \pa_\n^2\pa_\tau X \otimes \tau) \otimes \n \\
+&\left(\pa_\n^2\pa_\tau X \otimes \n +  (\pa_\n \pa_\tau X -(\div \, \n)^2 \pa_\n X
+\div \, \n  \, \pa_\n^2 X ) \otimes \tau
 \right ) \otimes \tau 
\end{split}
\end{align*}
so
\[
|\pa_\n  \nabla^2 X|^2 = |\pa_\n^3 X|^2 + 2 |\pa_\n^2\pa_\tau X|^2 +|\pa_\n\pa_\tau^2 X+ \div  \, \pa_\n^2 X - (\div \, \n)^2\pa_\n X |^2.
\]
Computing the latter term is more lengthy. We have
\begin{align*}
\pa_\tau& \nabla^2 X \\
\overset{\eqref{nabla^2X}}{=} & \pa_\tau\left( (\pa_\n^2 X \otimes \n + \pa_\n\pa_\tau X \otimes \tau) \otimes \n\right) \\
 + &\pa_\tau \left(\left(\pa_\n\pa_\tau X \otimes \n +  (\pa_\tau X+\div \, \n \, \pa_\n X) \otimes \tau
 \right ) \otimes\tau \right)  \\
\overset{\eqref{basicid}}{=}&
(\pa_\tau\pa_\n^2 X \otimes \n
+\div \, \n \, \pa_\n^2 X \otimes \tau
 + \pa_\tau\pa_\n\pa_\tau X \otimes \tau
- \div \, \n \, \pa_\n\pa_\tau X \otimes \n
) \otimes \n \\
+&(\div \, \n \,\pa_\n^2 X \otimes \n + \div \, \n \,\pa_\n\pa_\tau X \otimes \tau) \otimes \tau\\
+&(\pa_\tau  \pa_\n\pa_\tau X \otimes \n + \div \, \n \, \pa_\n\pa_\tau X \otimes \tau) \otimes \tau\\
+&( (\pa_\tau^3 X+ \pa_\tau \div \, \n \, \pa_\n X + \div \, \n \, \pa_\tau \pa_\n X ) \otimes \tau - (\div \, \n \, \pa_\tau^2 X+(\div \, \n )^2 \pa_\n X ) \otimes \n ) \otimes \tau 
\\
-&\left(\div \, \n \, \pa_\n\pa_\tau X \otimes \n +  (\div \, \n \,\pa_\tau^2 X+(\div \, \n )^2 \pa_\n X) \otimes \tau
 \right ) \otimes \n \\
=&\left((\pa_\tau \pa_\n^2 X - 2 \, \div \, \n \, \pa_\n \pa_\tau X)\otimes \n \right) \otimes \n \\
+&\left((\pa_\tau\pa_\n\pa_\tau X + \div \, \n \, \pa_\n^2 X- \div \, \n \,\pa_\tau^2 X- (\div \, \n)^2 \pa_\n X )\otimes \tau \right) \otimes \n \\
+&\left((\pa_\tau\pa_\n\pa_\tau X + \div \, \n \, \pa_\n^2 X- \div \, \n \, \pa_\tau^2 X- (\div \, \n)^2 \pa_\n X )\otimes \n \right) \otimes \tau \\
+&\left((\pa_\tau^3 X +2 \, \div \, \n \, \pa_\n\pa_\tau X + \div \, \n \,\pa_\tau \pa_\n X +  \pa_\tau \div \, \n \, \pa_\n X )\otimes \tau \right) \otimes \tau \\
\overset{\eqref{changeoforder}}{=} &\left((\pa_\tau \pa_\n^2 X - 2 \, \div \, \n \, \pa_\n \pa_\tau X)\otimes \n \right) \otimes \n \\
+&\left(\pa_\n\pa_\tau^2 X + \div \, \n \, \pa_\n^2 X- (\div \, \n)^2 \pa_\n X )\otimes \tau \right) \otimes \n \\
+&\left((\pa_\n\pa_\tau^2 X + \div \, \n \, \pa_\n^2 X- (\div \, \n)^2 \pa_\n X )\otimes \n \right) \otimes \tau \\
+&\left((\pa_\tau^3 X +3 \, \div \, \n \, \pa_\n\pa_\tau X + \pa_\tau \div \, \n \, \pa_\n X+ (\div \, \n)^2\pa_\tau  X )\otimes \tau \right) \otimes \tau
\end{align*}
and thus
\begin{align*}
|\pa_\tau  \nabla^2 X|^2 
=& |\pa_\tau \pa_\n^2 X - 2 \, \div \, \n \, \pa_\n \pa_\tau X|^2 + 2|\pa_\n\pa_\tau^2 X + \div \, \n \, \pa_\n^2 X- (\div \, \n)^2 \pa_\n X |^2 \\
+&|\pa_\tau^3 X +3 \, \div \, \n \, \pa_\n \pa_\tau X +  \pa_\tau \div \, \n \, \pa_\n X+ (\div \, \n)^2\pa_\tau  X |^2.
\end{align*}
Summing the expressions for $|\pa_\n  \nabla^2 X|^2$ and  $|\pa_\tau \nabla^2 X|^2$  then results in \eqref{|nabla^3X|^2}.
\end{proof}

With the identities \eqref{changeoforder}, \eqref{Delta X}, \eqref{divn2D}  and $|\nabla X|^2 = |\pa_\nu X|^2 + |\pa_\tau X|^2$
it is rather straightforward to derive Reilly's formula in the special case \eqref{vReilly2D}.
Indeed, for given $X \in C^2(\overline \Omega;\R^2)$
we compute
\begin{align}
\label{ReillyComp}
\begin{split}
\int_\Omega &|\nabla^2 X|^2 - |\Delta X|^2 \d x \\
&=\int_\Omega \sum_{i,j} \langle \pa_i\pa_j X,\pa_i\pa_j X \rangle- \sum_i |\pa^2_i X|^2 \d x \\
&=\int_{\pa \Omega} \sum_{i,j} \langle \pa_j X,\pa_i\pa_j X \rangle \langle e_i,\nu \rangle 
-  \langle \pa_j X,\pa_i^2 X \rangle \langle e_j,\nu \rangle  \d \H^1 \\
&=\int_{\pa \Omega} \frac12 \pa_\nu |\nabla X|^2- \langle \pa_\n X,\Delta X \rangle \d \H^1 \\
&=\int_{\pa \Omega}  \langle\pa_\nu X, \pa_\nu^2 X \rangle
+\langle\pa_\tau X, \pa_\nu\pa_\tau X \rangle
- \langle \pa_\nu X,\pa_\nu^2 X+\pa_\tau^2 X+\curv \pa_\nu X \rangle \d \H^1\\
&=\int_{\pa \Omega}  \langle\pa_\tau X, \pa_\tau  \pa_\nu X \rangle
- \curv \langle\pa_\tau X, \pa_\tau X\rangle 
-\langle \pa_\nu X,\pa_\tau^2 X+\curv \pa_\nu X \rangle \d \H^1\\
&=- \int_{\pa \Omega} 2 \langle \pa_\nu X,\pa_\tau^2 X \rangle + \curv |\nabla X|^2 \d \H^1.
\end{split}
\end{align}


\section{Appendix. Proofs of Subsection~\ref{subsec:domainIndep}}
In this section, we prove the domain independent integral estimates for UBC domains presented in Section \ref{subsec:domainIndep}. 
We recall that $c$ and $C$ stand for independent positive constants which may change from line to line.

By a standard rescaling argument we may restrict ourselves proving the estimates in special cases.
\begin{prop}[Rescaling argument]
\label{prop:rescale}
Let $\Omega \subset \R^2$ be a UBC domain and $\lambda \in \R_+$. The rescaled
domain $\widetilde \Omega=\lambda \Omega$ satisfies
\[
|\widetilde \Omega|=\lambda^2 |\Omega|, \ \  P(\widetilde \Omega)=\lambda \P(\Omega) \ \
\text{and} \ \ r_{\widetilde \Omega} = \lambda r_\Omega.
\]
Moreover, if $f \in H^k(\Omega)$, then $\tilde f = f \circ (\id/\lambda) \in H^k(\widetilde \Omega)$, $(\pa_\alpha \tilde f)_{\lambda A} = \lambda^{-l} (\pa_\alpha f)_A$ and
\[
\int_{\lambda A} |\pa_\alpha \tilde f-\delta (\pa_\alpha \tilde f)_{\lambda A}|^2 \d x =
\lambda^{2(1-l)}  \int_{A} |\pa_\alpha  f-\delta (\pa_\alpha f)_A|^2 \d x
\]
for every $\alpha \in \{1,2\}^l$ with $l=0,1,\ldots,k$, $\delta \in \{0,1\}$ and measurable subset $A$ of $\Omega$.
\end{prop}

The first quantified Gagliardo-Nirenberg interpolation inequality \eqref{interpolation:L2} can be obtained by quantifying the standard trace theorem.
\begin{lemma}[Quantified trace theorem]
\label{lem:tracethm}
Let $\Omega \subset \R^2$ be a UBC domain and $p \in [1,\infty)$. Then
\beq
\label{est:tracethm}
\int_{\pa \Omega} |f|^p \d \H^1 \leq  \int_{\Omega} \frac{2}{r_\Omega}|f|^p+ p|f|^{p-1}|\nabla f| \d x
\eeq
for every $f \in W^{1,p}(\Omega)$. Here $f$ on $\pa \Omega$ is the image of $f$ under the trace operator.
\end{lemma}
\begin{proof}
By approximating $\Omega$ we may assume that $\Omega$ is $C^2$-regular. Again we may assume
$f \in C^1(\overline \Omega)$. By the divergence theorem and 
\eqref{div(f*n)0} we  have 
\beq
\label{proof:tracethm1}
\int_{\pa \Omega } \phi f \d \H^1 = \int_{\Omega} \phi \pa_\n f + f \pa_\n \phi + \div \n \, \phi f
\eeq
for every $\phi \in C^1_0(\mathcal N(\pa \Omega))$.
We choose a sequence $(g_k)_k \subset C_0^1 (-r_\Omega,r_\Omega)$ such that 
$g_k(0) \rightarrow 1$, $g(t)=g(0)$ in a small interval $(-\delta_k,\delta_k)$,
\[
|g_k(t)| \leq 1 - \frac{|t|}{r_\Omega} \quad \text{and} \quad |g_k'(t)| \leq \frac{1}{r_\Omega} 
\]
for every $t \in (-r_\Omega,r_\Omega)$. By setting $\phi = g_k \circ d_\Omega$
and recalling $\n =  \nabla d_\Omega$ and \eqref{divn2D}
we have
\[
|\pa_{\n} \phi|, |\div \n \, \phi|  \leq \frac1{r_\Omega}
\]
in $\mathcal N(\pa \Omega)$. Combining this with \eqref{proof:tracethm1} and letting $k \rightarrow \infty$
results in
\beq
\label{proof:tracethm2}
\int_{\pa \Omega} f \d \H^1 \leq  \int_{\Omega} \frac{2}{r_\Omega} |f|+ |\nabla f| \d x.
\eeq
For given $p \in [1,\infty)$ it holds  $|f|^p \in W^{1,1}(\Omega)$ with $|\nabla |f|^p|=p|f|^{p-1}|\nabla f|$ a.e. in $\Omega$.
Since for every $\eps \in (0,r_\Omega)$ the sublevel set $\{d_\Omega < -\eps\}$ is a $C^2$-regular with the boundary $\{d_\Omega = -\eps\}$ and a maximal UBC radius at least $r_\Omega-\eps$, then by  
 mollifying $|f|^p$ , using \eqref{proof:tracethm2} for $\{d_\Omega < -\eps\}$ and passing the mollification to limit 
yields
\[
\int_{\{d_\Omega=-\eps\}} |f|^p \d \H^1 \leq  \int_{\Omega} \frac{2}{r_\Omega-\eps} |f|^p+ p|f|^{p-1}|\nabla f| \d x.
\]
Finally, since $f$ is uniformly continuous in $\overline \Omega$ we may let $\eps$ tend to zero and obtain \eqref{est:tracethm} 
for $f$. 
\end{proof}

\begin{proof}[Proof of Lemma \ref{lem:interpolation}]
To obtain \eqref{interpolation:L2} for $f \in H^2(\Omega)$ we integrate by parts and
apply Lemma \ref{lem:tracethm} for $f$ and $\nabla f$ to estimate
\begin{align*}
&\int_{\Omega} |\nabla f|^2 \d x \\
&= \int_{\pa \Omega} f \langle \nabla f, \n \rangle \d \H^1 - \int_{\Omega} f \Delta f \d x \\
&\leq \|\nabla f\|_{L^2(\pa \Omega)} \|f\|_{L^2(\pa \Omega)} +2 \|\nabla^2 f\|_{L^2(\Omega)} \|f\|_{L^2(\Omega)} \\
&\leq C\left(\frac{1}{r_\Omega}\|\nabla f\|_{L^2(\Omega)}^2 +\|\nabla^2 f\|_{L^2(\Omega)}\|\nabla f\|_{L^2(\Omega)}  \right)^\frac12 
\left(\frac{1}{r_\Omega}\|f\|_{L^2(\Omega)}^2 +\|\nabla f\|_{L^2(\Omega)}\| f\|_{L^2(\Omega)}\right)^\frac12 \\
 & \ +2 \|\nabla^2 f\|_{L^2(\Omega)} \|f\|_{L^2(\Omega)} \\
&\leq \frac{C}{r_\Omega}\|\nabla f\|_{L^2(\Omega)}\| f\|_{L^2(\Omega)} + \frac{C}{\sqrt{r_\Omega}}
\|\nabla f\|_{L^2(\Omega)}^\frac32\| f\|_{L^2(\Omega)}^\frac12
+\frac{C}{\sqrt{r_\Omega}}
\|\nabla^2 f\|_{L^2(\Omega)}^\frac12 \|\nabla f\|_{L^2(\Omega)}^\frac12\| f\|_{L^2(\Omega)} \\
& \ + C \|\nabla^2 f\|_{L^2(\Omega)}^\frac12 \|\nabla f\|_{L^2(\Omega)}\| f\|_{L^2(\Omega)}^\frac12
+2 \|\nabla^2 f\|_{L^2(\Omega)} \|f\|_{L^2(\Omega)} 
\end{align*}
and \eqref{interpolation:L2} follows after applying Young's inequality several times.

 The second interpolation inequality \eqref{interpolation:Linfty} follows directly from the uniform ball condition
and the special case $\Omega=B_1$. Indeed,  since the uniform ball condition implies for every point $x \in \Omega$ an existence of $y \in \Omega$
such that $x \in B_{r_\Omega}(y) \subset \Omega$, it suffices to prove \eqref{interpolation:Linfty} for the ball 
$B_{r_\Omega}$. Again, by Proposition
\ref{prop:rescale} we only need consider the case $r_\Omega=1$.
We recall the well-known interpolation
\[
\norm{f}_{L^\infty(B_1)} \leq C \|f\|_{H^2(B_1)}^\frac12\|f\|_{L^2(B_1)}^\frac12
\]
for every $f \in H^2(B_1)$, see for instance \cite[Thm 5.9]{AF}.
Thus by applying $\eqref{interpolation:L2}$ and Young's inequality on $\|\nabla u\|_{L^2(B_1)}$
we infer  $\eqref{interpolation:Linfty}$ for $B_1$.
\end{proof}

\begin{proof}[Proof of Lemma \ref{lem:interiorreg}]
We consider first the case  where $X \in H^2(B_\rho(z);\R^2)$ is biharmonic. By virtue of the standard elliptic estimates
we may assume $X$ is smooth. Then using the condition $\Delta^2 X =0$ in $\Omega$ it is rather straightforward to get an estimate
\[
\int_{B_\rho(z)} \phi |\nabla^3 X|^2 \d x \leq C\int_{B_\rho(z)} (|\nabla \phi|^2 +|\nabla^2 \phi|)  |\nabla^2 X|^2 + |\nabla^2 \phi|^2|\nabla X|^2 + |\nabla^3 \phi|^2 |X|^2 \d x
\]
for every $\phi \in C_0^\infty(B_\rho(z))$. We choose $\phi$ such that 
such that $|\phi| \leq 1$, $\phi = 1$ in $B_{\rho/2}(z)$ and $|\nabla^l \phi|^2 \leq C/\rho^l$ for $l=1,2,3$. Hence we obtain
\[
\int_{B_{\rho/2}(z)}  |\nabla^3 X|^2 \d x \leq C\int_{B_\rho(z)}\frac1{\rho^2} |\nabla^2 X|^2 + \frac1{\rho^4}|\nabla X|^2 +\frac1{\rho^6}|X|^2 \d x.
\]
The claim then follows this estimate combined with the Besicovitch covering theorem applied on the collection
$\{B_{\rho/2}(z) : d_\Omega(z) < - \rho\}$. 
\end{proof}

Proving the quantified Poincaré and Korn-type inequalities will be done in two steps. First we prove their local versions on the scale of maximal UBC radius. Then we combine the local results with a covering argument to conclude the global versions.

\begin{lemma}[Local estimates]
\label{lem:localestimates}
Let $\Omega \subset \R^2$ be a UBC domain
and let us denote $\Omega_y = \Omega \cap B_{r_\Omega/4}(y)$ and  
 $\widetilde \Omega_y = \Omega \cap B_{r_\Omega}(y)$
for every $y \in \Omega$.
There is an independent constant $C$
such that 
\begin{align}
\label{est:local1}
\int_{\Omega_y}  \abs{f-f_{\Omega_y}}^2 \d x 
&\leq C r_\Omega^2 \int_{\widetilde \Omega_y} \abs{\nabla f}^2 \d x, \\
\label{est:local2}
\int_{\Omega_y}  \abs{\nabla X}^2 \d x 
&\leq C\left( \int_{\Omega_y} \abs{\nablasym X}^2  \d x
+ \frac1{r_\Omega^2}  \int_{\Omega_y} \abs{X}^2 \d x \right), \\
\label{est:local3}
\int_{\Omega_y}  \abs{\nabla X-(\nabla X)_{\Omega_y}}^2 \d x
&\leq C \int_{\Omega_y} \abs{\nablasym X-(\nablasym X)_{\Omega_y}}^2  \d x,
\end{align}
for every $y,z \in \Omega$, $f \in H^1(\Omega)$ and $X \in H^1(\Omega;\R^2)$.
\end{lemma}

\begin{proof}[Proof of Lemma \ref{lem:localestimates}] By considering the rescaled domain $r_\Omega^{-1}\Omega$ and using Proposition
\ref{prop:rescale} it suffices to prove that in the case $r_\Omega=1$ we have  estimates
\begin{align}
\label{proof:local1}
\int_{\Omega_y}  \abs{f-f_{\Omega_y}}^2 \d x 
&\leq C \int_{\widetilde \Omega_y} \abs{\nabla f}^2 \d x, \\
\label{proof:local2}
\int_{\Omega_y}  \abs{\nabla X}^2 \d x 
&\leq C\left( \int_{\Omega_y} \abs{\nablasym X}^2  \d x
+ \int_{\Omega_y} \abs{X}^2 \d x \right), \\
\label{proof:local3}
\int_{\Omega_y}  \abs{\nabla X-(\nabla X)_{\Omega_y}}^2 \d x
&\leq C \int_{\Omega_y} \abs{\nablasym X-(\nablasym X)_{\Omega_y}}^2  \d x.
\end{align}

If $d_\Omega (y) \leq -1/4$, then $\Omega_y=B_{1/4}(y)$ is a fixed domain and the estimates
\eqref{proof:local1}, \eqref{proof:local2} and \eqref{proof:local3} are classical. Otherwise, by rotating and translating the coordinates we may
assume $\pi_{\pa \Omega}(0)=0$ and $\nu(0)=e_2$.
Using the quantified local graph representation in \cite[Lemma 2.9]{JN2023} for $\Omega$
in the rectangle $Q= (-1/4,1/4) \times (-1/2,1/2)$ and recalling $r_\Omega=1$ 
one can conclude by a straighforward calculation that there is $y' \in \Omega_y$ and independent $l \in (0,1/8)$ such that $\Omega_y$ is a star-shaped domain with respect  to $y'$ and $B_l(y')\subset \Omega_y$.
Then the estimates \eqref{proof:local2}
and \eqref{proof:local3} for $\Omega_y$ follow from the standard theory, see \cite{KoOl} and \cite{Duran}
respectively. 

Again the graph represention implies that the domain $W=\Omega \cap Q$  satisfies so called
 $\theta$-cone property with an independent number $\theta \in (0,\pi/2]$ meaning that
for every $z \in \pa W$ there is a unit vector $\omega$
such that $\mathcal C(w,\omega,\theta) \subset W$ for all
$w \in B_\theta (z) \cap \overline{W}$ where 
\[
\mathcal C(w,\omega,\theta) = \left\{v \in \R^2 : \langle v-w, \omega \rangle > |v-w| \cos \theta \ \ 
\text{and} \ \ |v-w| < \theta  \right\}.
\]
Since  we also have $W \subset B_1$, we may then apply \cite[Thm 1]{boulkhemairUniformPoincareInequality2007}  to conclude
\[
\int_W |f-f_W|^2 \d x \leq C \int_W |\nabla f|^2 \d x
\]
and hence using this estimate and  $\Omega_y \subset W \subset \widetilde \Omega_y$ we conclude\footnote{Also recall $\int_A |g-g_A|^2
\d x \leq \int_A |g-t|^2 
\d x$ for every $t \in \R_+$.} \eqref{proof:local3} by
\begin{equation*}
\int_{\Omega_y} |f-f_{\Omega_y}|^2 \d x \leq \int_{\Omega_y} |f-f_W|^2 \d x
\leq \int_W |f-f_W|^2 \d x \leq C \int_{\widetilde\Omega_y} |\nabla f|^2 \d x. \qedhere
\end{equation*}
\end{proof}

Now to compare averages around possibly distant points we require a ball chain type of argument. For this in turn, we need to know that distant parts of the domain are connect by a thick enough strip. As this is one essential feature of the UBC, we thus make the following observation without proof:

\begin{prop}
 \label{prop:connectivity}
 Let $\Omega \subset \R^2$ be a UBC domain. There are independent constants $\delta \in (0,1/16)$ and $C$ such that for every 
every distinct $x,y \in \{d_\Omega \leq -r_\Omega/8\}$ we find points $x_0,x_1,\ldots x_N \in \Omega$,
with $N \leq CP(\Omega)/r_\Omega$,  satisfying $x_0=x$ and $x_N=y$, $|x_i-x_{i+1}|\leq \delta r_\Omega$ for every $i=0,\ldots N-1$,
$B_{2\delta r_\Omega}(x_i) \subset \Omega$ for every $i=0,\ldots N$ and $\sup_{z \in \R^2} \# \{ i : z \in B_{2\delta r_\Omega}(x_i)\} \leq C$. 
\end{prop}

\begin{lemma}[Comparing averages]
\label{lem:averageEstimates}
Let $\Omega \subset \R^2$ be a UBC domain
and let us denote $\Omega_y = \Omega \cap B_{r_\Omega/4}(y)$
for every $y \in \Omega$.
There is an independent constant $C$
such that
\begin{align}
\label{est:avg1}
\abs{f_{\Omega_z}-f_{\Omega_y}}^2
&\leq C \frac{\P(\Omega)}{r_\Omega} \int_\Omega |\nabla f|^2 \d x \quad \text{and} \\
\label{est:avg2}
\abs{(\nabla X)_{\Omega_z}-(\nabla X)_{\Omega_y}}^2
&\leq C \frac{\P(\Omega)}{r_\Omega^3}\int_\Omega \abs{\nablasym X -(\nablasym X)_\Omega}^2 \d x
\end{align}
for every $y,z \in \Omega$, $f \in H^1(\Omega)$ and $X \in H^1(\Omega;\R^2)$.
\end{lemma}

\begin{proof}
As before, by rescaling it is enough to show that in the case $r_\Omega=1$ it holds
\begin{align}
\label{proof:avg1}
\abs{f_{\Omega_z}-f_{\Omega_y}}^2
&\leq C \P(\Omega) \int_\Omega |\nabla f|^2 \d x \quad \text{and} \\
\label{proof:avg2}
\abs{(\nabla X)_{\Omega_z}-(\nabla X)_{\Omega_y}}^2
&\leq C \P(\Omega) \int_\Omega \abs{\nablasym X -(\nablasym X)_\Omega}^2 \d x.
\end{align}

By the definition of the subdomains $\Omega_y$ and $\Omega_z$ as well as the uniform ball condition for $\Omega$
we find $y' \in \Omega_y$ and $z' \in \Omega_z$ such that $B_{1/8}(y') \subset \Omega_y$ and $B_{1/8}(z') \subset \Omega_z$.
Let $x_0=y',x_1,\ldots x_N=z'$ and $\delta$ be as in Proposition \eqref{prop:connectivity}.

Since $\delta \leq 1/16$, by using \eqref{proof:local1} we estimate
\begin{align}
\begin{split}
\label{proof:avg3}
|B_\delta| \abs{f_{\Omega_y} - f_{B_\delta(y')}}^2
&= \int_{B_\delta(y')}\abs{f_{\Omega_y} - f_{B_\delta(y')}}^2 \d x \\
&\leq 2 \int_{B_\delta(y')}\abs{f-f_{\Omega_y}}^2 \d x  + 2 \int_{B_\delta(y')} \abs{f-f_{B_\delta(y')}}^2 \d x \\
&\leq 4 \int_{\Omega_y}\abs{f-f_{\Omega_y}}^2 \d x
\leq C \int_{\widetilde \Omega_y}|\nabla f|^2 \d x \leq  C \int_{\Omega}|\nabla f|^2 \d x
\end{split}
\end{align}
and similarly by  \eqref{proof:local3}
\beq
\label{proof:avg4}
|B_\delta| \abs{(\nabla X)_{\Omega_y} - (\nabla X)_{B_\delta(y')}}^2
\leq  C \int_{\Omega} \abs{\nablasym X - (\nablasym X)_\Omega}^2 \d x. 
\eeq

Since $B_\delta(x_i),B_\delta(x_{i+1}) \subset B_{2\delta}(x_i)$, then by writing
\begin{align*}
 &f_{B_\delta(x_i)} -  f_{B_\delta(x_{i+1})} \\
&= \frac1{|B_\delta|} \left(
\int_{B_\delta(x_i) \setminus B_\delta(x_{i+1})} f - f_{B_{2\delta}(x_i)} \d x -  
\int_{B_\delta(x_{i+1}) \setminus B_\delta(x_i)} f - f_{B_{2\delta}(x_i)} \d x
 \right)
\end{align*}
and 
\begin{align*}
(&\nabla X)_{B_\delta(x_i)}-(\nabla X)_{B_\delta(x_i)} \\
&=\frac{1}{|B_\delta|} \left(\int_{B_\delta(x_i) \setminus B_\delta(x_{i+1}) } \nabla X - (\nabla X)_{B_{2\delta}(x_i)} \d x
-  \int_{B_\delta(x_{i+1}) \setminus B_\delta(x_i) } \nabla X - (\nabla X)_{B_{2\delta}(x_i)} \d x \right)
\end{align*}
we conclude
\begin{align*}
\abs{f_{B_\delta(x_i)} -  f_{B_\delta(x_{i+1})}}^2
&\leq C \int_{B_{2\delta}(x_i)} \abs{f-f_{B_{2\delta}(x_i)}}^2 \d x \qquad \text{and} \\
\abs{(\nabla X)_{B_\delta(x_i)} -  (\nabla X)_{B_\delta(x_{i+1})}}^2
&\leq C \int_{B_{2\delta}(x_i)} \abs{\nabla X-(\nabla X)_{B_{2\delta}(x_i)}}^2 \d x.
\end{align*}
Therefore by applying Poincaré inequality for the balls of fixed radius $2\delta$ and 
Korn-Poincaré inequality for balls we infer from previous
\begin{align*}
\abs{f_{B_\delta(x_i)} -  f_{B_\delta(x_{i+1})}}^2
&\leq C \int_{B_{2\delta}(x_i)} \abs{\nabla f}^2 \d x \qquad \text{and} \\
\abs{(\nabla X)_{B_\delta(x_i)} -  (\nabla X)_{B_\delta(x_{i+1})}}^2
&\leq C \int_{B_{2\delta}(x_i)} \abs{\nablasym X-(\nablasym X)_{B_{2\delta}(x_i)}}^2 \d x.
\end{align*}
Hence using Cauchy-Schwarz inequality, Proposition \ref{prop:connectivity} and the previous observations we further have
\begin{align}
\label{proof:avg5}
\begin{split}
\abs{ f_{B_\delta(z')} -  f_{B_\delta(y')} }^2
& \leq \abs{ \sum_i  f_{B_\delta(x_i)} -  f_{B_\delta(x_{i+1})} }^2 \\
& \leq N \sum_i  \abs{f_{B_\delta(x_i)} -  f_{B_\delta(x_{i+1})}}^2  \\
&\leq CN \sum_i \int_{B_{2\delta}(x_i)} |\nabla f|^2 \d x \\
&\leq CP(\Omega) \int_{\bigcup_i B_{2\delta}(x_i)} |\nabla f|^2 \d x \leq CP(\Omega) \int_\Omega |\nabla f|^2 \d x
\end{split}
\end{align}
and similarly 
\beq
\label{proof:avg6}
\abs{ (\nabla X)_{B_\delta(z')} -  (\nabla X)_{B_\delta(y')} }^2 \leq C P(\Omega) \int_\Omega \abs{\nablasym X-(\nablasym X)_\Omega}^2 \d x.
\eeq
Thus combining the  estimates \eqref{proof:avg3}, \eqref{proof:avg4}, \eqref{proof:avg5} and \eqref{proof:avg6}  as well as recalling $1 \leq P(\Omega)$ we conclude \eqref{proof:avg1} and \eqref{proof:avg2}.
\end{proof}

\begin{proof}[Proof of Lemma \ref{lem:Poincare}, Lemma \ref{lem:Korn1} and \ref{lem:Korn2}]
By replacing $\Omega$ with the rescaled domain $|\Omega|^{-1/2} \Omega$ and recalling Proposition
\ref{prop:rescale} it suffices to prove that in the case $|\Omega|=1$ instead of \eqref{Poincare:est} and 
\eqref{Korn:est2} 
we have estimates
\begin{align}
\label{proof:Poincare}
\int_{\Omega} \abs{f-f_\Omega}^2 \d x &\leq C \frac{\P(\Omega)}{r_\Omega} \int_\Omega \abs{\nabla f}^2 \d x \qquad \text{and} \\
\label{proof:Korn2}
\int_{\Omega} \abs{\nabla X-(\nabla X)_\Omega}^2 \d x &\leq C \left(\frac{\P(\Omega)}{r_\Omega^3}+1\right) \int_\Omega  \abs{\nablasym X-(\nablasym X)_\Omega}^2 \d x.
\end{align}
Note that after rescaling \eqref{proof:Korn2}  becomes
\[
\int_{\Omega} \abs{\nabla X-(\nabla X)_\Omega}^2 \d x \leq
C \left(\frac{|\Omega|\P(\Omega)}{r_\Omega^3}+1\right) \int_\Omega  \abs{\nablasym X-(\nablasym X)_\Omega}^2 \d x
\]
and thus recalling \eqref{est:rOmega}  $1$ can be absorbed into the factor $|\Omega|\P(\Omega)/r_\Omega^3$.
Again, we note that  the cutoff version \eqref{Korn_cutoff:est} of the Korn inequality
and the alternative formulation \eqref{Korn:est3} of the Korn-Poincaré inequality 
are rather straightforward consequences by possibly increasing the constants and, hence,
we omit their derivations.

 By virtue of the Besicovitch covering theorem and the boundedness of $\Omega$ 
there are an independent number $\xi \in \N$ and collections $\{B_{r_\Omega/4}(x_i^k)\}_{i=1}^{N_k}$, $k=1,\ldots,\xi$, of pairwise disjoint balls such that each $x_i^k \in \Omega$ and
\[
\Omega \subset \bigcup_k \bigcup_i B_{r_\Omega/4}(x_i^k).
\] 
We then set $\Omega_i^k =\Omega_{x_i^k}= \Omega \cap B_{r_\Omega/4}(x_i^k)$ and $U_k = \bigcup_i \Omega_i^k$ for each $k=1,\ldots,\xi$ and $i=1,\ldots N_k$. Since
$\{\Omega_i^k\}_i$ is a partition of $U_k$ for each $k$
and $\{U_k\}_k$ is a cover of $\Omega$, then by applying \eqref{est:local2} on $\Omega_i^k$s
we conclude immediately the Korn inequality \eqref{Korn:est1}.

To prove \eqref{proof:Poincare} and \eqref{proof:Korn2}
we first estimate $\int_{\Omega}\abs{g - g_\Omega}^2 \d x$
for  an arbitratry $g \in L^2(\Omega)$ by using the cover $\{U_k\}_k$. 
Since $\{\Omega_i^k\}_i$ is a partition of $U_k$ and the uniform ball condition implies that each $\Omega_i^k$ contains a ball of radius $r_\Omega/8$,
then $N_k |B_{r_\Omega/8}| \leq |U_k| \leq |\Omega| =1$
and we conclude an upper bound $N_k \leq C/r_\Omega^2$
for each $k=1,\ldots \xi$. The collection $\{U_k\}_k$ is a cover of $\Omega$ implying existence of $m \in \{1,\ldots,\xi\}$ such that $|U_m|^{-1} \leq \xi$. Again, there must be $l\in \{1,\ldots,\xi\}$
such that
\[
\int_{\Omega}\abs{g - |U_m|^{-1}\sum_i  \int_{\Omega_i^m} g \d y }^2 \d x \leq \frac1\xi
\int_{U_l}\abs{g- |U_m|^{-1}\sum_i  \int_{\Omega_i^m} g \d y }^2 \d x.
\]
Thus by using the previous estimate, $|U_m|^{-1} \leq \xi$, $|U_m|=\sum_i |\Omega_i^m|$, Cauchy-Schwarz inequality $N_m,N_l \leq C/r_\Omega^2$ and $|\Omega_i^m|,|\Omega_j^l| \leq Cr_\Omega^2$
we estimate as follows
\begin{align*}
\int_{\Omega}\abs{g - g_\Omega}^2 \d x 
&\leq \int_{\Omega}\abs{g - |U_m|^{-1}\sum_i  \int_{\Omega_i^m} g \d y }^2 \d x \\
&\leq C \int_{U_l}\abs{g - |U_m|^{-1}\sum_i  \int_{\Omega_i^m} g \d y }^2 \d x \\
&\leq C \int_{U_l} |U_m|^{-2} \abs{ \sum_i |\Omega_i^m| \abs{g -  g_{\Omega_i^m}}}^2 \d x \\
&\leq C \int_{U_l} \abs{ \sum_i |\Omega_i^m| \abs{g -  g_{\Omega_i^m}}}^2 \d x \\
&\leq C \int_{U_l} \sum_i |\Omega_i^m|^2  \sum_i \abs{g -  g_{\Omega_i^m}}^2 \d x \\
&\leq C N_m r_\Omega^4 \int_{U_l} \sum_i \abs{g -  g_{\Omega_i^m}}^2 \d x \\
&\leq C r_\Omega^2 \sum_j \sum_i \int_{\Omega_j^l}\abs{g -  g_{\Omega_j^l}}^2 +
\abs{g_{\Omega_j^l} -  g_{\Omega_i^m}}^2 \d x \\
&\leq CN_lN_m \max_j |\Omega^l_j| r_\Omega^2 \max_{i,j} \abs{g_{\Omega_j^l} -  g_{\Omega_i^m}}^2
+C r_\Omega^2 \sum_j \sum_i \int_{\Omega_j^l}\abs{g -  g_{\Omega_j^l}}^2 \d x \\
&\leq C \max_{i,j} \abs{g_{\Omega_j^l} -  g_{\Omega_i^m}}^2
+C r_\Omega^2 N_m \sum_j \int_{\Omega_j^l}\abs{g -  g_{\Omega_j^l}}^2 \d x \\
&\leq C \max_{i,j} \abs{g_{\Omega_j^l} -  g_{\Omega_i^m}}^2
+C \sum_j \int_{\Omega_j^l}\abs{g -  g_{\Omega_j^l}}^2 \d x.
\end{align*}

For \eqref{proof:Poincare}  by choosing $g=f$ and using \eqref{est:local1}, \eqref{est:avg1},
$N_m,N_l \leq C/r_\Omega^2$ and \eqref{est:rOmega}
we estimate
\begin{align*}
\int_{\Omega}\abs{f - f_\Omega}^2 \d x 
\leq &
C\frac{P(\Omega)}{r_\Omega} \int_\Omega \abs{\nabla f}^2 \d x
+C r_\Omega^2 \sum_j \int_{B_{r_\Omega}(x_j^l)} \abs{\nabla f}^2 \d x \\
\leq &
C\frac{P(\Omega)}{r_\Omega} \int_\Omega \abs{\nabla f}^2 \d x.
\end{align*}

To prove \eqref{proof:Korn2} we choose $g$ to be an entry function of $\nabla X$ in the standard coordinates.
Then using the estimates  \eqref{est:local3}, \eqref{est:avg2}, $N_m \leq C/r_\Omega^2$ and the
fact that $\Omega_j^l$s form a partition of $U_l$
we have
\begin{align*}
&\int_{\Omega}\abs{g - g_\Omega}^2 \d x \\
\leq &
C \max_{i,j} \abs{\nabla X_{\Omega_j^l} - (\nabla X)_{\Omega_i^m}}^2
+C \sum_j \int_{\Omega_j^l}\abs{\nabla X -  (\nabla X)_{\Omega_j^l}}^2 \d x \\
\leq & C \frac{P(\Omega)}{r_\Omega^3} \int_{\Omega}\abs{\nablasym X -  (\nablasym X)_{\Omega}}^2 \d x
+C \sum_j \int_{\Omega_j^l}\abs{\nablasym X -  (\nablasym X)_{\Omega_j^l}}^2 \d x \\
\leq & C \frac{P(\Omega)}{r_\Omega^3} \int_{\Omega}\abs{\nablasym X -  (\nablasym X)_{\Omega}}^2 \d x
+C \int_{U_l}\abs{\nablasym X -(\nablasym X)_{\Omega}}^2 \d x \\
\leq & C \frac{P(\Omega)}{r_\Omega^3} \int_{\Omega}\abs{\nablasym X -  (\nablasym X)_{\Omega}}^2 \d x
+C\int_{\Omega}\abs{\nablasym X -(\nablasym X)_{\Omega}}^2 \d x
\end{align*}
and, in turn, summing these estimates  together  results in \eqref{proof:Korn2}.
\end{proof}

\end{document}